\documentclass[smallextended]{svjour3}

\usepackage[hidelinks,pagebackref=true,pdfauthor={Andrew}]{hyperref}

\usepackage[square,sort,comma,numbers]{natbib}

\usepackage{xcolor}
\definecolor{dark-red}{rgb}{0.4,0.15,0.15}
\definecolor{dark-blue}{rgb}{0.15,0.15,0.4}
\definecolor{medium-blue}{rgb}{0,0,0.5}
\hypersetup{
    colorlinks, linkcolor={blue},
    citecolor={blue}, urlcolor={dark-red}
}

\usepackage{mathptmx}
\usepackage{float}
\usepackage[section, subsection]{extraplaceins}

\usepackage{amssymb}
\usepackage{amsmath}
\usepackage{graphicx}
\usepackage{enumerate}
\usepackage{subfigure}

\usepackage{color}
\usepackage{wrapfig}
\usepackage{bigints}

\usepackage[final, noinline]{fixme}

\def\norm#1{\|#1\|}

\def\R{\mathbb{R}} 
\def\T{\mathbb{T}} 
\def\P{\mathbb{P}}

\let\norm\relax
\newcommand{\Norm}[1]{\left\lVert #1 \right\rVert}
\newcommand{\norm}[1]{\left\lvert #1 \right\rvert}

\title{A multiscale analysis of diffusions on rapidly varying surfaces}

\author{A. B. Duncan \and C. M. Elliott \and G. A. Pavliotis \and A. M. Stuart}
\institute{A. B. Duncan
\at Mathematics Institute, University of Oxford, Woodstock Road, Oxford, OX2 6GC,  UK, \\\email{Andrew.Duncan@maths.ox.ac.uk.}
\and
C.M. Elliot 
\at Mathematics Institute, Warwick University, Coventry, CV4 7AL, UK, \\\email{C.M.Elliott@warwick.ac.uk.}
\and
G.A. Pavliotis 
\at Department of Mathematics, Imperial College London, London, SW7 2AZ, UK, \\\email{g.pavliotis@imperial.ac.uk.}
\and
A.M. Stuart 
\at  Mathematics Institute, Warwick University, Coventry, CV4 7AL, UK, \\\email{A.M.Stuart@warwick.ac.uk.}
}

\begin{document}

\maketitle

\begin{abstract}
Lateral diffusion of molecules on surfaces plays a very important role in various biological processes, including lipid transport across the cell membrane, synaptic transmission and other phenomena such as exo- and endocytosis, signal transduction, chemotaxis and cell growth. In many cases, the surfaces can possess
spatial inhomogeneities and/or be rapidly changing shape. Using a generalisation of the model for a thermally excited Helfrich elastic membrane, we consider the problem of lateral diffusion on quasi-planar surfaces, possessing both spatial and temporal fluctuations. Using results from homogenisation theory, we show that, under the assumption of scale separation between the characteristic length and time scales of the membrane fluctuations and the characteristic scale of the diffusing particle, the lateral diffusion process can be well approximated by a Brownian motion on the plane with constant diffusion tensor $D$ which depends in a highly nonlinear way on the detailed properties of the surface.  The effective diffusion tensor will depend on the relative scales of the spatial and temporal fluctuations and, for different scaling regimes, we prove the existence of a macroscopic limit in each case.
\end{abstract}




\section{Introduction}
Diffusion processes are ubiquitous in physics, chemistry and biology  \cite{crank1979mathematics, berg1993random, van1992stochastic}.   In biology, diffusion plays a fundamental role in many processes occurring at the cellular and sub-cellular level and is one of the basic mechanisms for intracellular transport \cite{bressloff2013stochastic}.  Diffusion not only occurs within the cell,  but can also occur along the cell membrane.  This lateral diffusion of molecules along the surface of cells also plays a key role in various cellular processes.  The lipid molecules and integral membrane proteins which constitute the cell membrane themselves undergo diffusion along the membrane as a result of thermal agitation \cite{almeida1995lateral}.    Lateral diffusion of postsynaptic membrane proteins between synapses is known to play a fundamental part in synaptic transmission \cite{borgdorff2002regulation, ashby2006lateral}. Other phenomena in cellular biology in which diffusion over interfaces is involved include vision \cite{poo1974lateral}, exo- and endocytosis, signal transduction, chemotaxis and cell growth (see \cite{sbalzarini2006simulations,almeida1995lateral}).
\\\\
Experimental techniques such as single particle tracking \cite{saxton1997single},  fluorescence recovery after photobleaching (FRAP)  \cite{axelrod1976mobility} and nuclear magnetic resonance (NMR) \cite{lindblom1994nmr} have made it possible to accurately measure displacement in a laboratory fixed plane of molecules diffusing laterally on the surface,  and thus to measure the macroscopic diffusion tensor $D$ of the diffusion process, projected into the plane.
\\\\
Biological interfaces, however, are not typically flat.  Indeed, many membranes will  exhibit a nonzero curvature which is induced by the natural spontaneous curvature of the constituent lipids \cite{seifert1997configurations}.  They may also be rough, i.e. possess spatial microstructure.   Moreover,  the shape of the membrane is  changing in time due to thermal fluctuations and possibly also non-thermal fluctuations induced by active membrane proteins on the surface \cite{gov2004membrane}. 
\\\\
The geometry of the membrane will cause the macroscopic diffusion tensor $D$ to be significantly different from the molecular diffusion tensor $D_0$ of the diffusing protein on the surface itself.  The relationship between the molecular diffusion tensor and the macroscopic diffusion tensor has been widely studied for different types of biomembrane.  Previous work such as \cite{gustafsson1997diffusion,naji2007diffusion,halle1997diffusion, sbalzarini2006simulations}  focus on the problem of lateral diffusion of a particle on a static membrane.  Various estimates for $D$ in terms of the surface fluctuation were derived, most notably the effective medium approximation and area scaling approximation \cite{king2004apparent,gustafsson1997diffusion,gov2006diffusion, naji2007diffusion}.  Other studies such as \cite{reister2007lateral, reister2007hybrid,reister2010diffusing} have focussed on the problem of diffusion on a thermally excited biomembrane fluctuating in a hydrodynamic medium and derived expressions for the effective diffusion tensor as a function of surface parameters such as bending rigidity, surface tension and fluid viscosity. 
\\\\
The common factor in these models is the presence of small length and time scales in the resulting evolution equations, which enter due to  spatial surface microstructure, or due to rapid temporal fluctuations of the surface or possibly both.  The objective of this paper is to investigate the macroscopic behaviour of a laterally diffusive process on surfaces possessing microscopic space and time scales using a single, unified mathematical approach. By doing so we provide rigorous justification for some existing approximations advocated in the literature,  clearly explaining the parametric regimes in which they apply and develop a systematic methodology which can be used to study other similar problems.  Under the assumption that the slow and fast scales are well-separated it is possible to show that the diffusion process can be approximated by a Brownian motion on the plane, independent of the small scale, but which accounts for the macroscopic effects of the fine spatial structure and rapid fluctuations. We use the classical methods of averaging and homogenisation \cite{bensoussan1978asymptotic,pavliotis2008multiscale}, and derive expressions for the coefficients of the macroscopic process in terms of averages with respect to a relevant measure reflecting the rapid fluctuations and involving the solution of an auxiliary  {\em cell problem}  in the case of homogenisation. Although these coefficients will not have a closed form in general, they can be computed numerically, accurately and efficiently without having to simulate effects at the microscopic level and they are amenable to analysis in various parameter regimes of interest.  In particular, we can obtain bounds on the coefficients in terms of the properties of the surface.
\\\\
The use of multiscale methods to study lateral diffusion on membranes has been considered before, with varying degrees of rigour.   In \cite{gustafsson1997diffusion}, the authors derive the correct macroscopic diffusion tensor for a particle diffusing on a surface with periodic spatial fluctuations basing their result on \cite{jackson1963effective,lifson1962self} and \cite{festa1978diffusion}, who consider the analogous situation of diffusion in a periodic potential.   Under the assumption of symmetry in the spatial fluctuations the authors then proceed to derive variational bounds for the effective diffusion and provide heuristic arguments for a number of other, tighter approximations.  In \cite{naji2007diffusion}, the authors study lateral diffusion on a Helfrich membrane undergoing thermal fluctuations.  They identify two limiting regimes: the diffusive limit (homogenization) of a diffusion on a quenched surface and the annealed limit (averaging) of diffusion on a rapidly fluctuating membrane, based on a formal analysis of the Fokker-Planck equation describing the evolution of the system, using an adiabatic elimination of the fast variable as in \cite[Section 8.3]{risken1996fokker}. They then use numerical methods to study the dynamics of the intermediate regimes where there is no separation of scales.   
\\\\
However, to our knowledge, there are no studies which adopt a rigorous multiscale approach to solving this problem, nor are we aware of any work which unifies the study of lateral diffusion on surfaces with both rapid spatial and temporal fluctuations in a single framework. Moreover, we are not aware of any study which makes use of multiscale methods to compute the effective diffusion tensor directly rather than via numerical simulation of the multiscale process, with the exception of \cite{abdulle2006heterogeneous} in which the authors  describe an HMM (heterogenous multiscale method) scheme for computing the solution of an elliptic partial differential equation (PDE) on a static surface possessing fine locally-periodic undulations and rigorously prove convergence of the scheme.
\\\\
In this paper, we consider two simple models for diffusion on a fluctuating surface.  The first model describes lateral diffusion on a static surface possessing rapid, periodic fluctuations.  Lateral diffusion on quasi-planar periodic surfaces have been previously studied in the literature, mainly in the context of biological interfaces.  The first such work we are aware of is \cite{aizenbud1985diffusion} where the authors consider the problem of diffusion on a curve possessing rapid periodic fluctuations with the objective of explaining the slowing down of diffusion of \emph{succiny-concanavalin A} receptors on the surfaces of adherent mouse fibroblast.  The authors derive the effective diffusion tensor, in this case, given by $D = \frac{1}{Z^2}$, where $Z$ is the average excess surface area of the curved surface relative to its projection on the plane.  In \cite{halle1997diffusion},  the authors study the same problem in two-dimensions,  and by recognizing the problem as diffusion in a periodic potential they use standard results to obtain the homogenised  diffusion tensor $D$ in terms of the solution of an auxiliary PDE (i.e. the cell problem).  Under some implicit symmetry assumptions on the surface, they then derive variational bounds for $D$.   The authors discuss various non-variational bounds for the $D$, and propose two estimates: the effective medium approximation, given by $$D_{ema} = \left\langle \sqrt{\norm{g}}({z})\right\rangle^{-1} \mathbf{I},$$ where $\sqrt{g}$ is the infinitesimal surface element of the surface and $\langle \cdot \rangle$ denotes the average with respect to the surface measure and the area scaling approximation given by \begin{equation}D_{as} = \frac{1}{Z} \mathbf{I}.\end{equation} The authors claim that $D_{ema}$ is the better approximation (which is at odds with the conclusion of Proposition \ref{prop:ema}).   In this paper we show that in the high-frequency, low amplitude limit of the surface fluctuations, the diffusion process behaves like a pure diffusion process (i.e. a Brownian motion) on $\R^d$ with constant macroscopic diffusion coefficient $D$.  Moreover, in two dimensions, we prove that, provided $D$ is isotropic,  then $D$ is equal to $D_{as}$.
\\\\
The second model we consider is a generalisation of the thermally excited Helfrich elastic membrane model \cite{gov2006diffusion,naji2007diffusion,reister2007lateral}.  The surface is defined by a time dependent random field undergoing rapid spatial and temporal fluctuations.  The macroscopic behavior of laterally diffusing particles on the surface will depend on the relative speed between the spatial and temporal fluctuations.  We identify a number of natural distinguished limits for this problem and study the effective properties of the corresponding limit processes and in particular providing a rigorous justification of the effective diffusion estimates derived in \cite{naji2007diffusion,reister2007lateral,gustafsson1997diffusion} for lateral diffusion on rapidly fluctuating surfaces. 
\\\\
In Section \ref{sec:background} we describe the formulation of lateral diffusion and Brownian motion on a time dependent, quasi-planar surface.  In Section \ref{sec:model} we introduce the framework for describing diffusion on a fluctuating surface, applicable to both models. Moreover, we identify four different scaling regimes for the second model.  In Section \ref{sec:helfrich} we describe a motivating example, namely that of a quasi-planar membrane with Helfrich elastic free energy and show that for small deformations the dynamics can be described by an Ornstein-Uhlenbeck process for the Fourier modes of the surface \cite{granek1997semi,naji2007diffusion}.    In Section \ref{sec:static} we focus on the first model and  using classical periodic homogenisation methods, we show that in the macroscopic limit, the projected lateral diffusion process converges to a Brownian motion on the plane with a constant diffusion tensor $D$, which can be expressed in terms of the solution to an auxiliary Poisson equation.  We study the properties of $D$ and provide  rigorous justification for a number of existing results for similar problems \cite{halle1997diffusion, naji2007diffusion}. 
\\\\
 In Sections \ref{sec:caseI} to \ref{sec:caseIV} we study the macroscopic limits of the second model under different scaling regimes.  In Section \ref{sec:caseI} we apply the results of Section \ref{sec:static} to study the asymptotic behaviour of the fluctuating membrane model in the quenched fluctuation regime.  In Section \ref{sec:caseII} we consider the problem of lateral diffusion on a surface possessing only rapid temporal fluctuations, a regime which has been well studied for the particular model of a thermally excited Helfrich elastic membrane.  Using formal multiscale expansions, we derive the limiting behaviour of this model, e.g. \cite{reister2007lateral, naji2007diffusion}.  For the particular case of the Helfrich elastic membrane we recover the  estimates for the effective diffusion tensor given in \cite{naji2007diffusion, halle1997diffusion, reister2007hybrid}.  In Section \ref{sec:caseIII} we consider diffusion on a surface possessing both spatial and temporal fluctuations with comparable length/time scales.  We derive expressions for the limiting equation and study properties of the effective drift and diffusion tensor.  Finally, in Section \ref{sec:caseIV} we study the asymptotic behaviour of diffusion on a rapidly fluctuating surface possessing both spatial and temporal fluctuations, but where the temporal fluctuations occur at a faster scale than the spatial fluctuations.   
 \\\\
 In Section \ref{sec:other_scalings} we consider the particular case of diffusion on a two-dimensional fluctuating Helfrich membrane and exhaustively study the particular limits of this problem.  Conclusions regarding the unifying nature and novelty of the multiscale  approach to this problem, as well as further avenues of research are summarized in Section \ref{sec:conclusion}.  Formal justifications of the limit theorems given in this paper are provided in Appendix \ref{app:A}.


\section{Diffusion on Time Dependent Surfaces}
\label{sec:background}
\subsection{Preliminaries}
\label{sec:prelims}
In this section we describe the formulation of Brownian motion moving on a time dependent surface embedded in $\R^{d+1}$.  We are primarily interested in quasi-planar membranes,  so  we will restrict our attention to surfaces which can be represented in the \emph{Monge parametrisation},  that is, surfaces which can be expressed as the graph of a sufficiently smooth function $H:\R^{d}\times [0,\infty) \rightarrow \R$.  Such a  surface $S(t)$ can then be parametrised over $\R^{d}$ by $J:\R^{d}\times [0,\infty) \rightarrow \R^{d+1}$ given by 
	\begin{equation*}
		J(x,t) = \left(x, H\left(x,t\right)\right).
	\end{equation*}
The function $H$ is known as the \emph{Monge gauge}.  In local coordinates $x \in \R^{d}$, the metric tensor of $S(t)$  induced from $\R^{d+1}$, can be written as		
\begin{equation}
	\label{eq:metric_tensor}
	G(x,t) = I + \nabla H(x,t) \otimes \nabla H(x,t)
\end{equation}
and the infinitesimal surface area element is given by $\sqrt{\norm{G}(x,t)}$, where
\begin{equation}
\label{eq:surface_area}
 |G|(x,t) := \det\bigl(G(x,t)\bigr) = 1 + \norm{\nabla H(x,t)}^{2}.
\end{equation}
It is clear that for any unit vector $e \in \R^{d}$,
\begin{equation}
\notag
	1 \leq e\cdot G(x,t) e \leq \norm{G}(x,t), \quad \mbox{for all } x \in \R^{d},
\end{equation}
so that $G^{-1}$ is positive definite (though not necessarily uniformly so, since $\norm{G}(x,t)$ can be arbitrarily large).  Given $F:\R^{d+1} \rightarrow \R$ smooth in a neighbourhood of $S(t)$,  the tangential gradient of $F$ is given in local coordinates by
$$ \nabla_{S(t)}F(J(x,t)) = \mathcal{P}(x, t)\nabla F(J(x,t)) = \nabla J(x,t)^\top G^{-1}(x,t)\nabla \left(F\circ J\right)(x, t).$$
Here, $\mathcal{P}(x,t)$ projects vectors in $\R^{d+1}$ onto the tangent space of $S(t)$ at local coordinate $x$, that is, $$\mathcal{P}(x, t) = I - \nu(x,t)\otimes \nu(x,t),$$ where $\nu(x,t)$ is the surface normal of $S(t)$.
In local coordinates, $\mathcal{L}_t$ acts on smooth functions $F \in C^2(\mathcal{\R}^{d+1})$ as follows, \cite{deckelnick2005computation, dziuk2013finite}:
			\begin{equation}
				\notag
				\label{eq:laplacebeltrami}
				\mathcal{L}_t F(J(x,t)) = \frac{1}{\sqrt{\norm{G}(x,t)}}\nabla\cdot\bigl(\sqrt{\norm{G}(x,t)}G^{-1}(x,t)\nabla \left(F\circ J\right)(x,t)\bigl), \quad \mbox{ for } x \in \R^d.
			\end{equation}
We note that for a flat surface, for which $H \equiv 0$, the operator  reduces to the standard Laplace operator.

\subsection{Brownian Motion on an Evolving Surface}
\label{sub:bm}
While the properties of Brownian motion on static surfaces have been widely studied in the applied literature \cite{van1985brownian, sbalzarini2006simulations,almeida1995lateral,naji2007diffusion}, Brownian motion on time dependent surfaces has been given less consideration.   In \cite{naji2007diffusion} the authors formally derive the overdamped Langevin equation for diffusion on a surface in the Monge gauge as the limit of a random walk constrained to the surface.   In \cite{coulibaly2011brownian}, the author provides a rigorous definition of Brownian motion on a manifold with a time dependent metric.  As we are working entirely in the Monge gauge we provide the following natural definition of Brownian motion on a fluctuating Monge gauge surface, which is equivalent to that given in \cite{coulibaly2011brownian} in the graph representation.

\begin{definition}
\label{def:bm}
Let $(\Omega, \mathcal{F}, \P)$ be a complete probability space endowed with a right-continuous filtration $(\mathcal{F}_t)_{t \geq 0}$.  Let $S(t)$ be a time dependent surface, with corresponding  Monge gauge $H(x, t)$, where $H(\cdot, t) \in C^2(\R^d)$, for all $t \geq 0$.   Then, an $\R^d$-valued process $X_x(t)$ defined on $\Omega\times [0,T)$ is called a Brownian motion on $S(t)$ started at $X_x(0)= x \in \R^d$, if $X(t)$ is almost surely continuous, adapted with respect to $\mathcal{F}_t$, and if for every smooth function $f:\R^d \rightarrow \R$,
	$$
		f(X_x(t)) - f(x) - \int_0^t \mathcal{L}_s f(X_x(s))\,ds,
	$$
is a local martingale \cite[Definition 5.5]{karatzas1991brownian},  where $\mathcal{L}_s$ is the Laplace-Beltrami operator (\ref{eq:laplacebeltrami}) in local coordinates on $\R^d$.
\end{definition}
\begin{remark}We note that in the case where $H \equiv 0$, Definition \ref{def:bm} reduces to standard Brownian motion on $\R^d$.
\end{remark}

Let $S(t)$ be a time dependent surface with Monge gauge $H(x,t)$ such that for $t \in [0, \infty)$, $H(\cdot, t) \in C^2(\R^d)$.  It is possible to obtain a standard It\^{o} SDE that describes Brownian motion on $S(t)$.  To this end, define $X_x(t)$ to be the solution of the following It\^{o} SDE
\begin{equation}
\label{eq:sde}
\begin{aligned}
	dX_x(t) & = F(X(t), t)\,dt  + \sqrt{2\Sigma(X(t),t)}\,dB(t), \\
	X_x(0) &= x,
\end{aligned}
\end{equation}
where $$F(x, t)  = \frac{1}{\sqrt{\norm{G(x, t)}}}\nabla\cdot \left(\sqrt{\norm{G(x, t)}} G^{-1}(x, t)\right),$$
$$\Sigma(x,t) = G^{-1}(x, t),$$
and $B(\cdot)$ is a standard $\R^d$-valued Brownian motion. By It\^{o}'s formula (Theorem 3.3, \cite{karatzas1991brownian}), for smooth $f:\R^d \rightarrow \R$:
\begin{equation*}
\begin{aligned}
	f(X_x(t)) - f(x) = &\int_0^t \frac{1}{\sqrt{\norm{G}(X_x(s),s)}}\nabla\cdot\left(\sqrt{\norm{G}(X_x(s),s)}\right)\cdot\nabla_x f(X_x(s))\,ds \\ 
		& + \int_0^t G^{-1}(X_x(s),s):\nabla_x \nabla_x f(X_x(s))\,ds \\
		& + \int_0^t \sqrt{G^{-1}(X_x(s),s)}\nabla_x f(X_x(s))\,dB(s) \\
		&= \int_0^t \mathcal{L}_s f(X_x(s))\, ds + M(t),
\end{aligned}
\end{equation*}
where $M(t)$ is a local martingale.  It follows that $X_x(t)$ satisfies the conditions of Definition \ref{def:bm} to be a Brownian motion on the evolving Monge gauge surface $S(t)$.
\\\\
Independently, we may derive from first principles the evolution equation for the probability density $\rho(z, t)$ of a particle undergoing Brownian motion on a time dependent surface given in the Monge gauge. The equation corresponds to the Fokker-Planck equation for the SDE (\ref{eq:sde}).  Consider a particle undergoing Brownian motion moving on the time dependent surface $S(t)$, and suppose that the process possesses a density $\rho(t,z)$ with respect to the Lebesgue measure on $S(t)$.  Let $\Theta$ be an arbitrary bounded region in $\R^{d}$ with smooth boundary, and let $\mathcal{M}(t)$ be the corresponding region on the fluctuating surface, that is,
\begin{equation*}
	\mathcal{M}(t) = J(\Theta, t).
\end{equation*}
The density $\rho(z,t)$ is conserved on the surface $S(t)$ for all $t$ such that $$\int_{S(t)} \rho(z,t)\,dz = 1, \quad \mbox{ for } t \geq 0.$$ Moreover, we assume that $\rho$ flows from one region of $S(t)$ to another with local Fickian flux $-\nabla_{S(t)}\rho(z,t)$ where $\nabla_{S(t)}$ is the tangential derivative on the surface $S(t)$. It follows that $\rho(z,t)$ satisfies the following equation 
\begin{equation*}
	\frac{\partial}{\partial t}\int_{\mathcal{M}(t)} \rho(z,t)\, dz = -\int_{\partial \mathcal{M}(t)} \nabla_{S(t)} \rho(z,t)\cdot n(z, t)\, dz = \int_{\mathcal{M}(t)} \Delta_{S(t)} \rho(z,t)\, dz,
\end{equation*}
where $n(z,t)$ is the conormal vector along the boundary of $\mathcal{M}(t)$.  See \cite{deckelnick2005computation} for details.   Changing variables from $z \in S(t)$ to local coordinates $x \in \R^d$ induces a change of measure $dz = \sqrt{\norm{G}(x,t)}dx$ where $\norm{G}$ is given by (\ref{eq:surface_area}).  We can thus rewrite the above equation in local coordinates as
\begin{equation*}
\begin{split}
	\frac{\partial}{\partial t}&\int_{\Theta} \rho(J(x,t),t)\sqrt{\norm{G}(x,t)}\,dx \\ &=  \int_{\Theta} \nabla\cdot\left(\sqrt{\norm{G}(x,t)}G^{-1}(x,t)\nabla \rho(J(x,t),t)\right) \, dx.
\end{split}
\end{equation*}
As we are only interested in the diffusion process projected onto the plane, we weight the  density $\rho$ with the surface area element $\sqrt{\norm{G}(x,t)}$ to compensate for the local changes in area of the surface. To this end, define the density $q:\R^d \times [0,\infty)\rightarrow \R$ with respect to the Lebesgue measure on $\R^d$ by:
\begin{equation*}
	q(x, t) := \rho\left(J\left(x,t\right),t\right)\sqrt{\norm{G}(x,t)}. 
\end{equation*}
It is straightforward to check that $\int_{\R^d} q(x,t) \,dx = 1$ for all time $t$.  Substituting $q(x,t)$ in the previous equation, and noting that $\Theta$ is arbitrary, we obtain the following PDE for $q$ on $\R^{d}$:
\begin{equation}
	\label{eq:f_planck}
	\frac{\partial}{\partial t}q(x,t) =  \nabla\cdot\left(\sqrt{\norm{G}(x,t)}G^{-1}(x,t)\nabla \left(\frac{q(x,t)}{\sqrt{\norm{G}(x,t)}}\right)\right) = \mathcal{L}_t^*q(x,t),
\end{equation}
where $\mathcal{L}_t$ is the local coordinate representation of the Laplace-Beltrami operator
$$
	\mathcal{L}_t f = \frac{1}{\sqrt{G(x,t)}}\nabla\cdot\left(\sqrt{G(x,t)}G^{-1}(x,t)\nabla f(x)\right), \qquad f \in C^2_b(\R^d).
$$
We note that the PDE (\ref{eq:f_planck}) is the Fokker-Planck evolution PDE for a diffusion process with infinitesimal generator given by $\mathcal{L}_t$  \cite[Chapter 6]{friedman1975stochastic},  in particular for the SDE (\ref{eq:sde}).  The corresponding backward Kolmogorov equation for the observable $u(x,t) := \mathbb{E}\left[u_0(X_x(t)\right]$ is given by the following PDE:
\begin{subequations}
\label{eq:kbe1}
\begin{align}
	\frac{\partial u(x,t)}{\partial t} &= \mathcal{L}_t u(x,t),	\quad (x,t) \in \R^d \times (0, \infty), \\
	u(x, 0) &= u_0(x) \qquad x \in \R^d.
\end{align}
\end{subequations}

\section{A Simple Model for Membrane Fluctuations}
\label{sec:model}
In this section we introduce a simple model for a fluctuating membrane which is based on the model for the thermally excited Helfrich membrane derived in \cite{gov2004membrane, naji2007diffusion} and  \cite{reister2007lateral}.  The fluctuating membrane surface is represented in the Monge gauge by a time dependent random field $H(x,t)$ over the region $[0,L]^d$.  We assume that for each $t \geq 0$, $H(x, t)$ is smooth in $x$ and  $H(x, t)$ is periodic in $x$ with period $L_H$ for each $t \geq 0$.  Moreover,  we shall assume there is a characteristic timescale $T_H$ associated to $H(x,t)$; it can be a correlation time when $H$ is random, or the period when $H(x,t)$ is periodic.
Consider a particle diffusing on a realisation of the surface $H(x,t)$ with an isotropic molecular diffusion tensor $D_0$.  Let $X(t)$ denote the projected trajectory on $\R^d$, and let $L$ and $T$ be the macroscopic characteristic length and time scales at which the process  $X(t)$ is being observed.  We introduce the notation
\begin{equation}
	\label{eq:scaling}
\begin{aligned}
	x = L x^*, & \qquad t = T t^*, \\
	X(Tt^*) &= L X^*(t^*),\\
	H(x, t) &= \tilde{H}{H}^*\left(\frac{x}{L_H}, \frac{t}{T_H}\right),
\end{aligned}
\end{equation}
where $\tilde{H}$ is a scaling constant, so that rescaled function ${H}^*$ has period $1$ in space.  Define the parameters $\delta$ and $\tau$ to be 
\begin{equation}
	\delta = \frac{L_H}{L} \quad \mbox{ and } \quad \tau = \frac{T_H}{T},
\end{equation}
which quantify the scale separation between the diffusion process $X(t)$ and the spatial and temporal fluctuations respectively.  To make explicit the relationship between spatial and temporal fluctuations we will assume that $\delta = \epsilon^\alpha$ and $\tau = \epsilon^\beta$ for constants $\alpha > 0 $ and $\beta \in \R$.  The assumption of rapid fluctuations implies that $\epsilon \ll 1$.   Rescaling (\ref{eq:sde}) using (\ref{eq:scaling}), setting $\tilde{H} = L_H$ and dropping the stars we obtain the following rescaled SDE
\begin{equation}
\label{eq:nondim_diffusion_sde}
\begin{split}
d{X^\epsilon}(t) &=\frac{1}{\epsilon^{\alpha}} \frac{1}{\sqrt{\norm{G}\left(\frac{{X^\epsilon}(t)}{\epsilon^{\alpha}}, \frac{t}{\epsilon^{\beta}}\right)}}\nabla_{y}\cdot\left(\sqrt{\norm{{G}}}{G}^{-1} \right)\left(\frac{{X^\epsilon}(t)}{\epsilon^{\alpha}}, \frac{t}{\epsilon^{\beta}}\right)\, dt \\ &+ \sqrt{2{G}^{-1}\left(\frac{{X^\epsilon}(t)}{\epsilon^{\alpha}},\frac{t}{\epsilon^{\beta}}\right)} \,dB(t),
\end{split}
\end{equation}
where 
\begin{equation}
{G}(y,s) = I + \nabla_{y}{H}(y, s) \otimes \nabla_{y}{H}(y, s).
\end{equation}
\\\\
Let $\mathbb{K}$ be a finite index set with cardinality $K = \norm{\mathbb{K}}$.  As a generalisation of the Helfrich elastic fluctuating membrane model, we will assume that the random field $H(x,t)$ can be written as
$H(x, t)  = h(x, \eta(t))$,
where $$h(x , \eta) = \sum_{k \in \mathbb{K}} \eta_k(t)e_k(x) = \langle \eta(t), e(x) \rangle,$$
where $e_k \in C^\infty(\T^d)$; these functions can be extended to $\R^d$ by periodicity.  We model the stochastic process $\eta(t)$ as an $\R^K$-valued Ornstein-Uhlenbeck (OU) process given by 
\begin{equation}
	d\eta(t) =  -\Gamma \eta(t) \,dt + \sqrt{2\Gamma \Pi}\,dW(t),
\end{equation}
where $W(\cdot)$ is a standard $\R^K$-valued Brownian motion.  The drift and diffusion matrices $\Gamma$ and $\Pi$ are symmetric, positive definite and are assumed to commute.  Substituting this definition of $H(x,t)$ into (\ref{eq:nondim_diffusion_sde}), the evolution of the system can be described by the joint process $(X^\epsilon(t), \eta^\epsilon(t))$ satisfies the following It\^{o} SDE
\begin{subequations}
\label{eq:nondim_diffusion_sde_coupled}
\begin{align}
d{X^\epsilon}(t) &=\frac{1}{\epsilon^{\alpha}}F\left(\frac{{X^\epsilon}(t)}{\epsilon^{\alpha}}, \eta^\epsilon(t)\right)\,dt  + \sqrt{2{\Sigma}\left(\frac{{X^\epsilon}(t)}{\epsilon^{\alpha}},\eta^\epsilon\left(t\right)\right)}\,dB(t),\\
d\eta^\epsilon(t) &= -\frac{1}{\epsilon^\beta}\Gamma \eta^\epsilon(t)dt + \sqrt{\frac{2\Gamma\Pi}{\epsilon^\beta}}\,dW(t)
\end{align}	
\end{subequations}
where $F:\T^d \times \R^K \rightarrow \R^d$ is given by
\begin{equation}
\label{eq:drift_term}
 F(x, \eta) := \frac{1}{\sqrt{\norm{g}(x, \eta)}}\nabla\cdot\left(\sqrt{\norm{{g}}}{g}^{-1} \right)\left(x, \eta\right),
\end{equation}
$\Sigma:\T^d \times \R^K \rightarrow \R^{2\times 2}_{\rm sym}$ is 
\begin{equation}
\label{eq:diffusion_term}
\Sigma(x, \eta) := g^{-1}(x, \eta),
\end{equation}
and $g(x , \eta) := I + \nabla h(x, \eta) \otimes \nabla h(x, \eta)$.  Since $\Gamma$ and $\Pi$ commute, it is straightforward to check that the OU process $\eta^\epsilon(t)$ is ergodic,  with unique invariant measure given by 
\begin{equation}
	\label{eq:mu_infty}
	\mu_\eta = \mathcal{N}(0, \Pi),
\end{equation}
with density $$\rho_\eta(\eta) \, \propto \,\exp\left(-\frac{\eta \cdot \Pi^{-1}  \eta}{2}\right),$$
with respect to the Lebesgue measure on $\R^K$.   Since $\eta(t)$  converges exponentially fast to its invariant distribution, it is reasonable to assume that  $\eta(t)$ is started in the stationary distribution i.e. $\eta(0) \sim \mu_\eta$, for the sake of simplicity.
\\\\
An equivalent approach, as described in Section \ref{sub:bm}, is to consider the asymptotic behaviour of the backward Kolmogorov equation (\ref{eq:kbe1}) corresponding to the system of SDEs (\ref{eq:nondim_diffusion_sde_coupled}) for the evolution of an observable $u^\epsilon:\R^d\times \R^K \times [0,T) \rightarrow \R$:
\begin{subequations}
	\label{eq:nondim_diffusion_pde_coupled}
	\begin{align}
	\frac{\partial u^\epsilon(x, \eta, t)}{\partial t} & = \mathcal{L^\epsilon}u^\epsilon(x,\eta, t), &\mbox{ for } (x, \eta, t) \in \R^d \times \R^K \times (0, T),	\\
	u^\epsilon(x, \eta, 0) &= v(x, \eta),	&\mbox{ for } (x, \eta) \in \R^d \times \R^K.
	\end{align}
\end{subequations}
The infinitesimal generator $\mathcal{L}^\epsilon$ can be written as $$\mathcal{L}^\epsilon f(x, \eta) =  \mathcal{L}^\epsilon_1 f(x, \eta) + \mathcal{L}^\epsilon_2 f(x, \eta). $$
The operator
 $$\mathcal{L}_1^\epsilon f(x) := \frac{1}{\sqrt{\norm{g}(x/\epsilon^\alpha, \eta)}} \nabla_x\cdot\left(\sqrt{\norm{g}(x/\epsilon^\alpha, \eta)}g^{-1}(x/\epsilon^\alpha, \eta)\nabla_x f(x)\right),$$
encodes the effect of the rapid spatial fluctuations, while 
$$\mathcal{L}_2^\epsilon f(\eta) := \frac{1}{\epsilon^{\beta}}\left(-\Gamma \eta \cdot \nabla_\eta  + \Gamma \Pi:\nabla_\eta \nabla_\eta f\right), $$
describes the rapid temporal fluctuations.
\\\\
The following proposition establishes the well-posedness of equation (\ref{eq:nondim_diffusion_sde_coupled}) for the joint process $(X^\epsilon(t), \eta^\epsilon(t))$.
\begin{proposition}
	\label{prop:well_posedness_sde}
	Let $X_0$ and $\eta_0$ be random variables, independent of $B(\cdot)$ and $W(\cdot)$ such that $\mathbb{E}\left[X_0\right]^{2} < \infty$ and $\mathbb{E}\left[\eta\right]^{2} < \infty$.  Then the system of SDEs (\ref{eq:nondim_diffusion_sde_coupled}) has a unique strong solution $(X^\epsilon(t), \eta^\epsilon(t))$ satisfying $X(0) = X$ and $\eta(0) = \eta$.  Moreover, the solution $(X^\epsilon(t), \eta^\epsilon(t)) \in C([0,T]; \R^d\times\R^K)$.
\end{proposition}
\smartqed \qed
	
Our objective is to study the behaviour of $X^\epsilon(t)$ and of solutions to the corresponding backward Kolmogorov equation as $\epsilon \rightarrow 0$.  The parameters $\alpha$ and $\beta$ quantify the relative speed between the spatial and temporal fluctuations respectively.  Thus, we expect that the limiting behaviour will vary for different values of $\alpha$ and $\beta$.   In this paper we will study the asymptotic behaviour of the coupled system in the following cases, which are demonstrative of the different possible limiting behaviours of the system.

\begin{description}
\item[Case I:  $\alpha = 1$ and $\beta = -\infty$]\hfill \\
In this regime the temporal fluctuations occur on a time scale slower than the characteristic time scale of the diffusion process, so that the regime captures the macroscopic behaviour of a  particle diffusing laterally over a stationary realisation of the random surface field $h(x, \eta)$.  This situation has been studied in the case of diffusion on a Helfrich elastic membrane with quenched thermal fluctuations in \cite{naji2007diffusion}.
\\
\item[Case II: $\alpha = 0$ and $\beta = 1$]\hfill \\
In this regime, the microscopic fluctuations are due to the temporal fluctuations.  The motivating example in this regime is that of diffusion on a Helfrich elastic membrane with annealed fluctuations; this problem has been studied in detail \cite{naji2007diffusion,gustafsson1997diffusion,reister2007lateral}.
\\
\item[Case III: $\alpha = 1$ and $\beta = 1$]\hfill \\
In this regime we consider lateral diffusion on surfaces possessing both rapid spatial and temporal fluctuations, with the spatial and temporal fluctuations occurring at comparable scales.  While this regime has not been studied before, it naturally extends the work covered in \cite{halle1997diffusion,naji2007diffusion,reister2007lateral}  and helps provide a complete picture.  
\\
\item[Case IV: $\alpha = 1$ and $\beta = 2$]\hfill \\
In this regime we consider surfaces with both rapid spatial and temporal fluctuations but the temporal fluctuations occur at a faster scale compared to the spatial fluctuations.  As in Case III, this regime has not been considered previously.
\end{description}

In each of the above cases we will show that the lateral diffusion process $X^\epsilon$ is asymptotically characterized by a limiting diffusion process with constant diffusion tensor and drift term which will be qualitatively different in each case. 
\\\\
Before considering the above four regimes for the fluctuating membrane model,  we first consider the problem of lateral diffusion on a static, periodic surface. The diffusion process $X^\epsilon(t)$ can be described by an SDE with rapidly varying, periodic coefficients.  In Section \ref{sec:static} we use standard periodic homogenization techniques to show that the asymptotic behaviour is that of a Brownian motion on $\R^d$ with constant diffusion tensor $D$.  The resulting analysis serves as a basis for the subsequent models.
\\\\
In the Case I regime, considered in Section \ref{sec:caseI},  each stationary realisation of the random field gives rise to a homogenized diffusion tensor.   As in \cite{naji2007diffusion}, we consider the average homogenized diffusion coefficient as the effective diffusion tensor in this regime.  In the Case II regime considered in Section \ref{sec:caseII},  the limiting behaviour is determined by the properties of the stationary distribution of the OU process $\eta(t)$ and deriving the effective diffusion process can be viewed as an averaging problem \cite{pavliotis2008multiscale}. 
\\\\
In the regimes covered by Case III and Case IV we must consider the interaction between the temporal and spatial fluctuations.   In the Case III regime, considered in Section \ref{sec:caseIII}, the spatial fluctuations homogenise the diffusion process ``faster" than the temporal fluctuations, and the result is that the effective diffusion tensor will merely be the effective diffusion tensor from Case I averaged over the invariant measure of the OU process $\eta(t)$.  This macroscopic limit was considered in \cite{garnier1997homogenization}.   Deriving the asymptotic behaviour in the Case IV regime, considered in Section \ref{sec:caseIV}, proves more complicated due to the fact that the ``fast process'' which characterises the rapid spatial and temporal fluctuations does not possess  an explicit invariant measure.  Once the geometric ergodicity of the fast process with respect to a unique invariant measure is established,  the approach is similar to the classical probabilistic homogenisation arguments of \cite{bensoussan1978asymptotic}.   Although a limiting equation is established,  the lack of an explicit invariant measure makes it hard to establish bounds on the effective diffusion tensor.
\\\\
We have not yet addressed the question of the limiting behaviour of $X^\epsilon(t)$ for other values of $\alpha$ and $\beta$ besides those considered in Cases I - IV.  The answer to this question is dependent on the properties of the surface $H(x,t)$.  However, in Section \ref{sec:other_scalings}, for the particular case of diffusion on a thermally fluctuating Helfrich surface we will show that the limits corresponding to Case I to Case IV are exhaustive, in the sense that these are the only distinguished limits that can arise from this system.

\section{Motivating Example: The Helfrich Elasticity Membrane Model}
\label{sec:helfrich}
In this section we describe a particular instance of the above model which describes lateral diffusion of particles on a thermally-excited Helfrich membrane.  Throughout the paper we shall consider  this model as a prototypical example to which the theory can be applied.  The Canham-Helfrich theory for biological membranes, originally developed by Canham \cite{canham1970minimum} and Helfrich \cite{helfrich1973elastic}, is a classical continuum model for studying the macroscopic properties of lipid bilayers. The problem of diffusions on Helfrich elastic surfaces in the context of integral protein diffusion on lipid bilayers was studied in various works, in particular \cite{naji2007diffusion,reister2007lateral,lin2004dynamics,king2004apparent}.  As in all the previous works, we represent  the quasi-planar membrane as a  2D time dependent fluctuating surface with Monge Gauge $H$, periodic over the square $\mathcal{D} = [0,L]^2$, where the equilibrium of the fluctuations is governed by the following harmonic approximation to the  Helfrich Hamiltonian
\begin{equation}
	\mathcal{H}[H] = \frac{1}{2}\int_{\mathcal{D}} \left[\kappa (\Delta H(y))^2 + \sigma (\nabla H(y))^2\, \right] \, dy.
\end{equation}
Here the scalars $\kappa$ and $\sigma$ denote the bending rigidity and surface tension, respectively.  The surface is coupled with a low-Reynolds number hydrodynamic medium.  Using an analogous approach as in \cite{doi1988theory} for polymer dynamics, under the assumption of linear response, the dynamics of the surface fluctuations will be described by the following SPDE:
\begin{equation}
	\label{eq:spde_helfrich}
	\frac{d H(t)}{dt} = -R\, A H(t) + \zeta(t),
\end{equation}
where $A H(t)$ is the restoring force for the free energy $\mathcal{H}$,  that is,
	$$A H(t) = -\frac{\delta \mathcal{H}}{\delta H}[H(t)] = -\kappa\Delta^2 H(t) + \sigma \Delta H(t)$$
The operator $R$ which characterises the effect of  nonlocal interactions of the membrane through the hydrodynamic medium is given by
$$
	R f(x) = (\Lambda* f)(x),  \qquad f \in L^2([0,1]^2)
$$
where $*$ denotes convolution, and $\Lambda(x)$ is given by the diagonal part of the Oseen tensor, \cite{kim1991microhydrodynamics}:
	$$\Lambda(x) = \frac{1}{8 \pi \lambda\norm{x}},$$
where $\lambda$ is the viscosity of the surrounding hydrodynamic medium.
\\\\
The space-time fluctuations are given by $\zeta(t)$, a centered Gaussian random field white in time and with spatial fluctuations having covariance operator $2 k_B T R,$ where $k_B T$ is the system temperature.   From linear response theory, it follows that the dynamics in (\ref{eq:spde_helfrich}) satisfy the fluctuation-dissipation relation, required to ensure that, formally, the invariant measure is proportional to $\exp(-\mathcal{H}/({k_BT}))$. 
\\\\
Let $\left\lbrace e_k \, | \, k \in \mathbb{K}_\infty \right\rbrace$ be the standard Fourier basis for $L^2([0,1]^2; \R)$ with periodic boundary conditions,  indexed by $$\mathbb{K}_\infty = \mathbb{Z}^2\setminus\lbrace (0,0)\rbrace.$$ 
It is straightforward to check that the invariant measure of $H(t)$ is given by the Gaussian measure $\mathcal{N}(0, \mathcal{C})$ where
\begin{equation}
\notag
\mathcal{C} = {k_BT}\,\sum_{k \in \mathbb{K}_\infty}\left(\kappa \norm{2\pi k}^4 + \sigma L^2\norm{2\pi k}^2\right)^{-1} e_k(x)\otimes e_k(x).
\end{equation}
  The operator $\mathcal{C}^{-\frac{1}{2}}$ satisfies Assumptions 2.9 (i)-(iv) of \cite{stuart2010inverse}, so that its spectrum grows commensurately with the spectrum of $-\Delta$.  It follows from \cite[Lemma 6.25]{stuart2010inverse} that the stationary realisations of the random field will be, H\"{o}lder continuous with exponent $\alpha < 1$, but not for $\alpha = 1$.  This implies that realisations are not sufficiently regular to allow well-defined tangents at every point on the surface. Indeed, $H(x,t)$ will be almost surely nowhere differentiable with respect to $x$ so that it is not possible to consider a laterally diffusive process on a realisation of this random field. We must thus introduce an ultraviolet cut-off by setting $\langle e_k, \mathcal{C} e_k \rangle = 0$ for wavenumbers $k \not \in \mathbb{K}$, where
$$\mathbb{K} = \lbrace k \in \mathbb{K}_\infty \, | \norm{k} \leq c \rbrace,$$
for some fixed constant $c > 0$  and define $K = \norm{\mathbb{K}}$.
\\\\
 Looking for solutions $H(x,t)$ of the form $H(x,t) = h(x, \eta(t))$, for $h:\T^2 \times \R^K$ given by
$$h(x, \eta) = \sum_{k \in \mathbb{K}}\eta_k e_k(x),$$
 after substituting in (\ref{eq:spde_helfrich}) we note that the SPDE diagonalises to obtain a system of Ornstein-Uhlenbeck processes describing the dynamics of the Fourier modes.   If we consider the trajectory $X(t)$  of a particle undergoing Brownian motion on the surface $h(x,\eta(t))$ with scalar molecular diffusion tensor $D_0$ then after nondimensionalisation we obtain the following system of equations
 \begin{subequations}
\begin{align}
d{X^\epsilon}(t) &=F\left({{X^\epsilon}(t)}, \eta^\epsilon(t)\right)\,dt  + \sqrt{2{\Sigma}\left({{X^\epsilon}(t)},\eta^\epsilon\left(t\right)\right)}\,dB(t),\\
d\eta^\epsilon(t) &= -\frac{1}{\epsilon}\Gamma \eta^\epsilon(t)dt + \sqrt{\frac{2\Gamma\Pi}{\epsilon}}\,dW(t)
\end{align}	
\end{subequations}
where  $\epsilon = D_0 \lambda L$ and where $F$ and $\Sigma$ are given by (\ref{eq:drift_term}) and (\ref{eq:diffusion_term}), respectively.  The OU process coefficients are determined by $\Gamma = \mbox{diag}(\Gamma_k)$ with
\begin{equation}
	\label{eq:helfrich_drift}
	\Gamma_{k} =  \frac{\kappa^*\norm{2\pi k}^4 + \sigma^* \norm{2\pi k}^2}{\norm{2\pi k}}
\end{equation}
and  $\Pi = \mbox{diag}(\Pi_k)$ with
\begin{equation}
	\label{eq:helfrich_diffusion}
	\Pi_{k} = \frac{1}{\kappa^* \norm{2\pi k}^4 + \sigma^* \norm{2\pi k}^2},
\end{equation}
where $\kappa^*$ and $\sigma^*$ are the nondimensional constants given by $$\kappa^* = \frac{\kappa}{k_B T} \quad \mbox{ and } \quad \sigma^* = \frac{\sigma L^2}{k_B T},$$ respectively.  The invariant measure of the $\R^K$-valued OU process $\eta(t)$ is then given by
\begin{equation}
	\notag
	\mathcal{N}(0, \Pi) = \prod_{k \in \mathbb{K}}\mu_k,
\end{equation}
where for each $k \in \mathbb{K}$,  $\mu_k$ is the invariant measure of $\eta_k$ given by
\begin{equation}
\label{eq:invariant_measure_helfrich}
\mu_k = \mathcal{N}\left(0, \Pi_k\right).
\end{equation}
The parameter $\chi = \epsilon^{-1}$ had already been considered in \cite{naji2007diffusion} where it was called the \emph{dynamic coupling 	parameter} because it controls the scale separation between the diffusion and the surface fluctuations.  For the particular case of band-3 protein diffusion on a human red blood cell, the typical values of parameters give $\epsilon \approx 0.3$, which suggests that $\epsilon$ is an appropriate small-scale parameter.  Of course, the value of $\epsilon$ will vary greatly for different scenarios.   We will discuss the limiting behaviour of the Helfrich model in Section \ref{sec:caseI}, however, we will also study this model in the other regimes as it provides an illuminating example.  In particular, for each scaling we will consider the effects of the parameters $\kappa^*$, $\sigma^*$ on the limiting diffusion tensor.


\section{Case 0: Diffusion on static, periodically varying surfaces}
\label{sec:static}
In this section we consider the first case described in Section \ref{sec:model}, namely lateral diffusion on a prescribed static surface with periodic fluctuations about the plane.  More specifically, we  consider a surface $S^\epsilon$ with Monge gauge:
	\begin{equation}
		\label{eq:h}
		h^\epsilon(x) = \epsilon h\left(\frac{x}{\epsilon}\right),
	\end{equation}
where $\epsilon$ is a small scale parameter and $h$ is a sufficiently smooth real-valued function on $\R^d$ such that $h$ and its derivatives are periodic with period $1$ in every direction.  The excess surface area is conserved as $\epsilon \rightarrow 0$, which suggests that the scaling in (\ref{eq:h}) is justified,  see \cite{duncan2013thesis} for further details.  It is straightforward to see that $S^\epsilon$ has metric tensor $g^\epsilon(x) = g(x/\epsilon)$, where
\begin{equation}
	g(x) = I + \nabla h(x) \otimes \nabla h(x), 	\qquad x \in \T^d.
\end{equation}
Consider a particle diffusing along the surface $S^\epsilon$ and let  $X^\epsilon(t)$ denote the projection onto the plane. Following the derivation in Section \ref{sub:bm} with $H(x, t)= h(x)$, independent of time,   the evolution of $X^\epsilon(t)$ is given by the following It\^{o} SDE
\begin{equation}
	\label{eq:periodic_sde}
	dX^\epsilon(t) = \frac{1}{\epsilon}F(X^\epsilon(t)/\epsilon) \,dt + \sqrt{2\Sigma(X^\epsilon(t)/\epsilon)}\,dB(t),
\end{equation}
where $$F(x) = \frac{1}{\sqrt{\norm{g}(x)}}\nabla\cdot\left(\sqrt{\norm{g}(x)}g^{-1}(x)\right),$$
and $$\Sigma(x) = g^{-1}(x).$$
\\
Equivalently, we can consider the diffusion equation for an observable $u^\epsilon(t,x)$ diffusing laterally on a surface $h^\epsilon(t)$ which can be written in local coordinates
\begin{subequations}
\label{eq:periodic_pde}
\begin{align}
	\frac{\partial u^\epsilon(x,t)}{\partial t} &= \mathcal{L}^\epsilon u^\epsilon(x,t), 	\quad &(x,t) \in \R^d \times (0,T], \\
	u^\epsilon(t,x ) &= u(x),	\quad &(x,t) \in \R^d\times\lbrace 0 \rbrace.
\end{align}
\end{subequations}
where 
\begin{equation}
	\mathcal{L}^\epsilon f(x) = \frac{1}{\sqrt{\norm{g}(x/\epsilon)}}\nabla_x\cdot\left(\sqrt{\norm{g}(x/\epsilon)})g^{-1}(x/\epsilon)\nabla_x f(x)\right),
\end{equation}
and where $u \in C_b(\R^d)$.  The process $X^\epsilon(t)$ and $f^\epsilon(x,t)$ are connected via the  backward Kolmogorov equation  \cite[Chapter 6]{friedman1975stochastic}.
\\\\
Our objective is to study the effective behaviour of $X^\epsilon(t)$ and $u^\epsilon(x,t)$ as $\epsilon \rightarrow 0$.  We will show that as $\epsilon\rightarrow 0$, the $\R^d$-valued process $X^\epsilon(t)$ will converge weakly to a Brownian motion on $\R^d$ with constant diffusion tensor $D$ which depends on the surface map $h(x)$.  Equivalently,  we show that $u^\epsilon$ converges pointwise to the solution $u^0$ of the PDE:

\begin{subequations}
\label{eq:periodic_eff_pde}
\begin{align}
	\frac{\partial u^0(x,t)}{\partial t} &= D:\nabla_x\nabla_x u^0(x,t), 	\quad &(x,t) \in \R^d \times (0,T], \\
	u^0(t,x ) &= v(x),	\quad &(x,t) \in \R^d\times\lbrace 0 \rbrace.
\end{align}
\end{subequations}
 Since (\ref{eq:periodic_sde}) (respectively (\ref{eq:periodic_pde})) is a SDE (resp. PDE) with periodic coefficients, the problem is amenable to classical periodic homogenisation methods, such as those of \cite{bensoussan1978asymptotic,zhikov1994homogenization,pavliotis2008multiscale}.  In Section \ref{sec:static_homog} we state the homogenisation result for this model.   The result will be justified formally by using perturbation expansions of the PDE in (\ref{eq:periodic_pde}).  The rigorous proof of this result is presented in \cite{duncan2013thesis}.

\subsection{The Homogenization Result}
\label{sec:static_homog}
For convenience, we introduce the fast process $Y(t) = \frac{X(t)}{\epsilon} ~ \mathbf{mod} ~ \T^d$.  We can then express (\ref{eq:periodic_sde}) as the following fast-slow system
\begin{subequations}
\label{eq:sde2}
\begin{align}
	dX^{\epsilon}(t) &= \frac{1}{\epsilon}F(Y^\epsilon(t))\,dt + \sqrt{2 \Sigma ( Y^{\epsilon}(t)}\, dB(t),\\
	dY^{\epsilon}(t) &= \frac{1}{\epsilon^2}F(Y^\epsilon(t))\,dt + \sqrt{\frac{2}{\epsilon^2} \Sigma ( Y^{\epsilon}(t)})\, dB(t),
\end{align}
\end{subequations}
where $X^\epsilon(t) \in \R^d$, $Y^\epsilon(t) \in \T^d$ and $B(t)$ is a standard Brownian motion on $\R^d$. The infinitesimal generator of the fast process is the $L^2(\T^d)$ closure of 
\begin{equation}
\label{eq:static_L0}
	\mathcal{L}_0 f(y) = \frac{1}{\sqrt{\norm{g}(y)}}\nabla_{y}\cdot\left(\sqrt{\norm{g}(y)}g^{-1}(y)\nabla_y f(y)\right), \quad f \in C^2(\T^d).
\end{equation}
It is straightforward to see that $\mathcal{L}_0$ is a uniformly elliptic operator with null\-space containing only constants, that is $$\mathcal{N}[\mathcal{L}_0] = \lbrace \mathbf{1} \rbrace,$$ and $$\mathcal{N}[\mathcal{L}^*_0] = \lbrace \rho(y) \rbrace,$$
where $\rho(y) = \frac{\sqrt{\norm{g}(y)}}{Z}$, for $Z = \int_{\T^d}\sqrt{\norm{g}(y)}\, dy$.
\\\\
We expect to be able to compute the homogenising effect of the fast process $Y^{\epsilon}$ on the slow process $X^{\epsilon}$ and thereby obtain an effective equation which accounts for, but removes explicit reference to, the small scale.  Given $v \in C^{2}_{b}(\R^{2}\times\T^{d})$, the observable $$v^{\epsilon}(x,y) := \mathbb{E}\left[v\left(X^{\epsilon}(t), Y^{\epsilon}(t)\right) \, | \, X^{\epsilon}(0) = x, Y^{\epsilon}(0) = 0\right]$$ satisfies the  backward Kolmogorov equation given by:
\begin{equation}
	\label{eq:kbe_ms}
	\begin{split}
			\frac{\partial v^\epsilon(x, y,t)}{\partial t} = \mathcal{L}^\epsilon v^{\epsilon}(x,y,t) &, \quad (x,y,t) \in \R^d\times\T^{d}\times(0,T],
	\end{split}
\end{equation}
where
\begin{equation}
\label{eq:Leps}
\mathcal{L}^\epsilon = \mathcal{L}_{2} + \frac{1}{\epsilon}\mathcal{L}_{1} + \frac{1}{\epsilon^{2}}\mathcal{L}_{0}
\end{equation}
for
\begin{equation}
\label{eq:l1}
\begin{split}
\mathcal{L}_{1}v(x,y) :=  F(y)\cdot \nabla_{x}v(x,y) + 2\Sigma(y):\nabla_{x}\nabla_{y}v(x,y),
\end{split}
\end{equation}
and 
\begin{equation}
\label{eq:l2}
\mathcal{L}_{2}v(x,y) := \Sigma(y):\nabla_{x}\nabla_{x}v(x,y),
\end{equation}
and where $\mathcal{L}_0$ is given by (\ref{eq:static_L0}).  Note that the last term in (\ref{eq:l1}) reflects the correlation of the noise between the fast and slow processes.  We wish to study the behaviour of $X^\epsilon$ and $v^\epsilon$ in the limit as $\epsilon \rightarrow 0$, homogenising over the fast variable $Y^\epsilon$ to identify a constant coefficient diffusion equation which approximates the slow process.    As the corresponding SDE and PDE have periodic coefficients, we  can apply results from classical homogenisation theory such as \cite{bensoussan1978asymptotic,zhikov1994homogenization} to prove convergence of $X^{\epsilon}$ and $v^{\epsilon}$ to solutions of limiting equations.  In this Section we will state the homogenisation result for this problem.
\\\\
The macroscopic effect of the fast scale fluctuations is characterised by a \emph{corrector} $\chi:\T^d \rightarrow \R^{d}$ which is the solution of the following Poisson problem:
\begin{equation}
\label{eq:celleqn_case1}
	\mathcal{L}_0 \chi(y) = -F(y),	\qquad y \in \T^d.
\end{equation}  

By the Fredholm alternative and elliptic regularity theory, \cite{gilbarg2001elliptic}, we have that (\ref{eq:celleqn_case1}) has a unique, mean zero solution $\chi \in C^2(\T^d; \R^d)$.
\\\\
The following theorem states the homogenisation result for this scaling regime.   A formal derivation using perturbation expansions will be given in Appendix \ref{sec:case1_app} which can be used as the basis for a rigorous proof. However, a probabilistic approach based on  \cite[Theorem 3.1]{pardoux1999homogenization} or \cite{bensoussan1978asymptotic} are more succint.   In what follows we will adopt the convention that $$\left(\nabla_{y}\chi(y)\right)_{ij} = \frac{\partial \chi_{i}}{\partial y_{j}}(y), \quad \mbox{ for } i, j \in 1,\ldots, d$$ see Chapter 2 of \cite{gonzalez2008first}.

\begin{theorem}
\label{thm:periodic_homog}
The process $X^{\epsilon}$ converges weakly in $C([0,T]; \R^{d})$ to a Brownian motion $X^0(t)$ with effective diffusion tensor $D$ given by 
\begin{equation}
\label{eq:eff_diff_matrix}
D = \frac{1}{Z}\int_{\T^{d}}\left(I + \nabla_{y}\chi(y)\right) g^{-1}(y)\left(I + \nabla_{y}\chi(y)\right)^{\top}\sqrt{\norm{g}(y)}\,dy,
\end{equation}
where $Z$ is the surface area of a single cell of the surface given by 
\begin{equation}
	\label{eq:Z}
	Z = \int_{\T^d}\sqrt{\norm{g}(y)}\,dy.
\end{equation}
 Moreover, if equation (\ref{eq:periodic_pde}) has initial data $u$ independent of the fast variable such that $u \in C^2_b(\R^d)$, then the solution $u^{\epsilon}$ of (\ref{eq:periodic_pde}) converges pointwise to the solution $u^{0}$ of (\ref{eq:periodic_eff_pde}) uniformly with respect to $t$ over $[0,T]$. \smartqed\qed
\end{theorem}

\begin{remark}
In the second part of Theorem \ref{thm:periodic_homog} we assume that the initial condition of the backward Kolmogorov equation for this system depends only on the slow variable $x$.   While this assumption greatly simplifies the analysis,  it is not essential.  Indeed,  if the initial condition $u$ also depends on the fast variable, then an \emph{initial layer} arises at $t = 0$, which can be resolved by introducing auxiliary correction terms to the multiscale expansion which decay exponentially fast in time. Refer to \cite{khasminskii1996asymptotic, MR1948866} for a similar scenario.
\end{remark}

\subsection{Properties of the Effective diffusion tensor}
\label{sec:static_properties}

In the one dimensional case,  by integrating (\ref{eq:celleqn_case1}) directly, one can show \cite{pavliotis2008multiscale} that  the effective diffusion coefficient is given by $$D = \frac{1}{Z^{2}},$$ so that the homogenised equation (\ref{eq:periodic_eff_pde}) becomes
\begin{equation}
\frac{\partial u_{0}(x,t)}{\partial t} = \frac{1}{Z^{2}}\Delta u_{0}(x,t),	\quad (x,t)\in \R^{1}\times (0,T].
\end{equation}

In two-dimensions or higher it is not possible to solve for the corrector $\chi$ explicitly and thus $D$ has no closed form.  However, we can identify certain properties of the effective diffusion tensor.  Let $$H^{1}_{per}(\T^{d}) := \left\lbrace v \in H^{1}(\T^{d}) \, \Bigg| \, \int_{\T^{d}}v(y)\,dy = 0 \right\rbrace,$$ and $S^{d} = \lbrace e \in \R^{d+1} \, | \, |e|=1 \rbrace$.  The following proposition illustrates the basic properties of $D$, valid in all dimensions.
\\
\begin{proposition}
\label{prop:properties}
Let $e \in S^{d-1}$, then
\begin{enumerate}
	\item $D$ is a symmetric, positive definite matrix.
	\item $D$ can be characterised via the expression  
	\begin{equation}
		\label{eq:minimisation}
		e \cdot D e  = \inf_{v \in H^{1}_{per}(\T^{d})} L[v,e],
\end{equation}
where \begin{equation*}
	L[v,e] := \frac{1}{Z}\int_{\T^{d}}\left( e+ \nabla v(y)\right)\cdot g^{-1}( y)\left(e + \nabla v(y)\right)\sqrt{\norm{g}(y)}\,dy
\end{equation*}

	and $\chi^e = \chi_1 e_1 + \chi_2 e_2$ is the unique minimiser of (\ref{eq:minimisation}).
	\item The following Voigt-Reuss bounds \cite[Section 1.6]{zhikov1994homogenization} hold,
	\begin{equation*}
		e\cdot D_{*} e \leq e\cdot D e \leq e\cdot D^{*} e \end{equation*}
	where 
	\begin{equation}
\label{eq:bound_lower}
			D_{*} = \frac{1}{Z}e\cdot\left(\int_{\T^{d}}\frac{g(y)}{\sqrt{\norm{g}(y)}}\,dy\right)^{-1}e, 
	\end{equation}
and  \begin{equation}
		\label{eq:bound_upper}
		D^{*} = \frac{1}{Z}e\cdot \left(\int_{\T^{d}}g^{-1}( y)\sqrt{\norm{g}( y)}\,dy\right)e.
\end{equation}

\label{it:depleted}
\item In particular, since the microscopic diffusion tensor in the non\-dimensionalized diffusion equation (\ref{eq:sde2}) is $I$, the homogenised diffusion tensor $D$ is bounded from above by the molecular diffusion.
\end{enumerate}
\end{proposition}

\begin{proof}
Property (ii) follows by noting that the Euler-Lagrange equation for the  minimiser of (\ref{eq:minimisation}) is given by (\ref{eq:celleqn_case1}) which has a unique solution $\chi\cdot e$ in $H^{1}_{per}(\T^{d})$.  Moreover, $D = L[\chi^e, e]$.
It follows that for each  unit vector $e \in \R^{d}$, 
\begin{equation}
	\notag
	 D \leq L[0,e] = \frac{1}{Z}e \cdot\left(\int_{\T^{2}}g^{-1}(y)\sqrt{\norm{g}(y)} \,dy\right)e =: e\cdot D^{*}e,
\end{equation}
proving the second inequality of (iii).   To derive the Voigt-Reuss type lower bound in (iii) we note that for fixed $e\in S^{d-1}$, 
\begin{equation}
	\notag
D_{*} :=\inf_{\substack{\Phi \in L^{2}(\T^{d})^{d}\\  \int_{\T^{d}}\Phi(y)\,dy = 0}}\frac{1}{Z}\left(e + \Phi(y) \right)\cdot g^{-1}(y)\left(e + \Phi(y)\right)\sqrt{\norm{g}(y)}\,dy \leq D.
\end{equation}
The corresponding Euler-Lagrange equation is given by 
\begin{equation}
	\notag
	g^{-1}(y)\left(e + \Phi(y)\right)\sqrt{\norm{g}(y)} = C,
\end{equation}
where $C$ is a Lagrange multiplier for the constraint $\int_{\T^{d}}\Phi(y)\,dy = 0$, this can be solved explicitly to show that
\begin{equation}
	\notag
	\Phi(y) + e = \frac{g(y)}{\sqrt{\norm{g}(y)}}\left(\int_{\T^{d}}\frac{g(y)}{\sqrt{\norm{g}(y)}}\right)^{-1}e,
\end{equation}
so that 
\begin{equation}
	\notag
D_{*} = \frac{1}{Z}e\cdot\left(\int_{\T^{2}}\frac{g(y)}{\sqrt{\norm{g}(y)}}\,dy\right)^{-1}e,
\end{equation}
thus proving (iii).  Moreover,  the positive-definiteness of $D$ follows immediately from that of $D_{*}$.  Using the fact that $\norm{g^{-1}} \leq 1$, it follows that
\begin{equation}
	\notag
	e\cdot D e \leq e\cdot D^{*}e \leq 1,
\end{equation}
and thus proving (iv).   The symmetry of $D$ follows from the symmetry of the inverse metric tensor, proving (i).
\end{proof}

By using expression (\ref{eq:minimisation})  it is possible to obtain variational bounds for $D$ other than $D^{*}$ by minimising over a proper closed subset of $H^{1}_{per}(\T^{d})$.  By minimising over larger subsets it is possible to obtain increasingly tighter bounds (see \cite[Chapter 4]{duncan2013thesis}.  However we have not been able to obtain bounds which are consistently tight over different periodic surfaces using this approach.

\subsection{The Area Scaling Approximation}
\label{sec:case1_duality_transformation}
For surface fluctuations which are genuinely two-dimensional (i.e. not constant along a particular axis) we cannot expect to find an explicit expression for the solution of the cell equation.   Nonetheless, for a large class of two dimensional surfaces, using a \emph{duality transformation} argument we  are able to exploit symmetries which exist exclusively in the two dimensional case to obtain an explicit expression for the effective diffusion tensor, namely the area scaling approximation 
\begin{equation}
\label{eq:D_area_scaling}
D_{as} = \frac{1}{Z} \mathbf{I},
\end{equation}
 where $Z$ is the excess surface area.  Therefore,  provided $D$ is isotropic, it will depend only on the average excess surface area and not on the particular microstructure of the rough surface.
\\\\
In their simplest forms, duality transformations provide a means of relating the effective diffusion tensor $\sigma_*$ obtained through homogenising an elliptic PDE of the form 
\begin{align*}
	-\nabla\cdot \left(\sigma\left(\frac{x}{\epsilon}\right)\nabla v^\epsilon(x)\right) &= f, \qquad x \in \Omega \Subset \R^2 \\
	v^\epsilon(x) &= 0, \qquad x \in \partial \Omega,
\end{align*}
the effective diffusion tensor $\sigma_*'$ of a dual problem, where $\sigma' = Q^\top \sigma Q$ for a $90^\circ$ rotation $Q$.  The existence of such a duality depends strongly on the fact that in two dimensions the $90^\circ$ rotation of a divergence-free field is curl-free and vice versa. Such transforms were used firstly in conductance problems in  \cite{keller1963conductivity} and subsequently by \cite{matheron1967elements,dykhne1971conductivity,mendelson1975theorem,kohler1982bounds}.

\begin{proposition}
\label{prop:ema}
In two dimensions, $D$ satisfies the following relationship 
\begin{equation}\label{eq:det}\det(D) = \frac{1}{Z^2}.\end{equation}
Consequently, if $\lambda_1$ and  $\lambda_2$ are the eigenvalues of $D$ with $\lambda_1 \leq \lambda_2$, then
\begin{equation}
	\label{eq:det_relation}
	\frac{1}{Z^2} \leq \lambda_1 \leq \frac{1}{Z} \leq \lambda_2 \leq 1.
\end{equation}
In particular, if $D$ is isotropic, then $D$ can be written explicitly as 
		$$D = \frac{1}{Z}\mathbf{I}.$$
\end{proposition}

\begin{proof}
	The above result follows from a straightforward modification of the standard duality transformation.  For a proof using variational principles see \cite{duncan2013thesis}.
	\smartqed \qed
\end{proof}

When $D$ is not isotropic, Proposition (\ref{prop:ema}) still provides us with useful constraints on the anisotropy of the effective diffusion tensor.   We see that if the macroscopic diffusion is unhindered in the direction corresponding to $\lambda_1$ then the effective diffusion will be $\frac{1}{Z^2}$ in the orthogonal direction, corresponding to a diffusion on a one-dimensional surface.  In the other extreme,  if $\lambda_1 = \lambda_2$ then we have an isotropic diffusion tensor and by the above proposition $\lambda_1 = \lambda_2 = \frac{1}{Z}$.

\subsection{A Sufficient Condition for Isotropy}
\label{sec:static_isotropy}
In all of the previous literature regarding lateral diffusion on two-dimensional biological membranes, it is always assumed that the macroscopic diffusion tensor is isotropic, i.e. a scalar multiple of the identity. While this is a natural assumption, it is clearly not true in general. In this section we identify a natural sufficient condition to guarantee the isotropy of the effective diffusion tensor.  The condition we will assume is the following:
\begin{equation}
	\label{cnd:isotropic}
	h(x) = h(Qx),	\quad x \in \R^{2},
\end{equation}
where $Q:\R^{2}\rightarrow \R^{2}$ is a $\frac{\pi}{2}$ rotation about some point $\mathbf{O} \in \R^{2}$. Without loss of generality we assume that $\mathbf{O} = (0,0)$.
\\
\begin{lemma}
	\label{lemm:metric_tensor_rotation}
	Let $Q \in \R^{2\times 2}$ be any rotation about the origin.  If (\ref{cnd:isotropic}) holds, then 
	\begin{equation}
		\label{eq:metric_tensor_condition}
				g^{-1}(Q\,x) = Q\,g^{-1}(x)Q^\top
	\end{equation}
	and 
	\begin{equation}
		\label{eq:area_element_condition}
				\norm{g}(Q\,x) = \norm{g}(x),
	\end{equation}
	for all $x \in \R^{2}$.
\end{lemma}
\smartqed \qed
%

We now prove that the above condition is a sufficient condition for the effective diffusion tensor to be isotropic.  The proof we present here is based on Schur's lemma \cite{james2001representations}.  A similar approach can be found in  \cite[Section 1.5]{zhikov1994homogenization}. 

\begin{theorem}
	\label{thm:isotropy_periodic}
	If condition (\ref{cnd:isotropic}) holds, then $D$ is isotropic.
\end{theorem}
\begin{proof}
	We use a characterisation of $D$ given by (\ref{eq:minimisation}), namely
	\begin{equation}
	\notag
		e\cdot D e = \frac{1}{Z}\inf_{v \in H^{1}_{per}(\T^{2})}\left(e + \nabla v(y)\right)\cdot g^{-1}(y)\left(e + \nabla v(y)\right)\sqrt{\norm{g}(y)}\,dy.
	\end{equation}
	Changing variables $Q\,z = y$, using (\ref{eq:metric_tensor_condition}) and (\ref{eq:area_element_condition}),we obtain:
	\begin{equation}
	\notag
	\begin{split}
		& e \cdot D e = \frac{1}{Z}\inf_{v \in H^{1}_{per}(\T^{2})}\int_{\T^2}\left(e + \nabla v(Q\, z)\right)\cdot Q\,g^{-1}(z)Q^\top\left(e + \nabla v(Q\, z)\right)\sqrt{\norm{g}(z)}\,dz \\ 
		&= \frac{1}{Z}\inf_{v \in H^{1}_{per}(\T^{2})}\int_{\T^2}\left(Q^\top e + Q^\top\nabla v(Q\,z)\right)\cdot g^{-1}(z)\left(Q^\top e + Q^\top \nabla v(Q\, z)\right)\sqrt{\norm{g}(z)}\,dz.
	\end{split}
	\end{equation}
	Noting that $Q^\top\,\nabla v(Q\,z) = \nabla w(z),$ where $w(z) = v\circ Q(z)$, since $Q$ is a $\frac{\pi}{2}$ rotation, it is clear that $w \in H^{1}_{per}(\T^{2})$ if and only if $v \in H^{1}_{per}(\T^{2})$.  It follows that
	\begin{equation}
	\notag
		e\cdot D e = e\cdot Q\,D\,Q^\top\, e,
	\end{equation}
	for all $e \in S^{1}$.  Since $D$ is symmetric, it follows by Schur's lemma that $D$ is isotropic.
\end{proof}

\subsection{Numerical Method}
\label{sec:periodic_numeric_scheme}
\fxnote{This is new!}
To compute the effective diffusion tensor for a general two dimensional surface,  rather than adopt the MCMC approach as in \cite{pavliotis2007homogenization} of generating sample paths of $X^\epsilon(t)$ using an Euler scheme and estimating $D$ from the mean square deviation, we instead use a finite element scheme to solve the cell equation and compute  $D$ directly from (\ref{eq:eff_diff_matrix}).  The latter approach is prefereable in this case as, in two-dimensions, assembly of the corresponding linear system of equations is still cheap and the resulting matrix problem is relatively well conditioned,  so that it can be solved efficiently using a preconditioned conjugate gradient method.   On the other hand, a Monte Carlo approach requires long time simulations or many realisations of the SDE to recover $D$.   Finally,  the finite element scheme described in this section can be applied to compute the effective diffusion tensors in the fluctuating membrane model considered in Sections \ref{sec:caseI}, \ref{sec:caseII}, \ref{sec:caseIII} and \ref{sec:caseIV} whereas the corresponding Monte Carlo simulations would involve solving an extremely stiff system of equations over very long time intervals.
\\\\
The corrector $\chi$ is approximated numerically with piecewise linear elements to solve the cell problem (\ref{eq:celleqn_case1}).  For the approximation of $\chi$ we use a regular triangulation of the domain $[0,1]^2$ with mesh-width $h$.  To impose the periodic boundary conditions of (\ref{eq:celleqn_case1}) we identify the boundary nodes of the mesh periodically.  Thus, for $h = \frac{1}{M}$, $M \in \mathbb{N}$, the resulting finite element scheme has $M^2$ degrees of freedom.
\\\\
The stiffness matrix corresponding to the elliptic differential operator $\mathcal{L}_0$ is assembled using nodal quadrature \cite{larsson2009partial} to compute the local contribution of each triangular element to the stiffness matrix.  The load vector corresponding to the right hand side of the cell equation is computed similarly. Thus, the derivatives of the surface map $h_{x_1}$ and $h_{x_2}$ are evaluated only at the nodes of the mesh.   For simple surfaces,  the derivatives can be computed directly for each node.  The stiffness matrix $S$ corresponding to $\mathcal{L}_0$ is positive semi-definite, with kernel consisting of constant functions.  Since the right hand side of the finite element approximation of the cell equation is orthogonal to $S$, the corresponding matrix equation is solvable. 
\\\\
Once the stiffness matrix and right hand side have been assembled,  the resulting symmetric  matrix equation is solved using a preconditioned conjugated gradient method.  Given a piecewise linear approximation $\chi^{e_i}_h$ of the solution of the cell equation, we approximate $Z$ and the effective diffusion tensor $D$ using linear quadrature.  
We note that $D$ can be written as
$$
	e\cdot D e = \frac{1}{Z} \left[\int_{\T^d} e \cdot g^{-1}(y) e \sqrt{\norm{g}(y)}\,dy - \langle \chi^e, \chi^e\rangle_V\right],
$$
where $$\langle \chi^e, \chi^e\rangle_V = \int_{\T^d}\nabla \chi^e(y) \cdot g^{-1}(y)\nabla\chi^e(y)\sqrt{\norm{g}(y)}\,dy$$ is the energy norm of $\chi^e$.   Thus, using standard a priori error estimates \cite{brenner2008mathematical} for the finite element approximation $D_h$  one can show a rate of convergence $h^2$ of the approximate effective diffusion tensor $D_h$ to the exact value $D$. 

\subsection{Numerical Examples}
\label{sec:static_numerics}
To illustrate the properties described in the previous sections, we apply the numerical scheme of Section \ref{sec:periodic_numeric_scheme} to numerically compute the effective diffusion tensor for diffusions on various classes of surfaces, along with the bounds $D^{*}$, $D_{*}$ derived in Proposition \ref{prop:properties} as well as the area scaling estimate (\ref{eq:D_area_scaling}). 
\\\\
In Figure \ref{fig:eff1} we consider a surface defined by $h(x) = A \sin(2\pi x_1)\sin(2\pi x_2)$ for varying $A$.  The isotropy condition holds, so that $D$ is isotropic and is given by (\ref{eq:D_area_scaling}).  The Voigt-Reuss bounds are sharp in the weak disorder regime (small $A$) but become increasingly weak as $A$ increases, with $D^{*}$ approaching $\frac{1}{2}$ in the strong disorder limit (large $A$) while $D_{*}$ converges to $0$.  As predicted by Proposition \ref{prop:ema}, the area scaling approximation correctly determines the effective diffusion tensor. 
\\\\
In Figures  \ref{fig:eff2}  and \ref{fig:eff2_ev}  we consider the surface given by $$h(x) = \sin(2\pi x_1)\sin(6\pi x_2) + A \sin(6\pi x_1)\sin(2\pi x_2).$$ The effective diffusion tensor will not be isotropic except for $A = 1$.  In Figure \ref{fig:eff2} we plot $D_{11} := e_1 \cdot D e_1$ for varying $A$.  The Voigt-Ruess bounds $D^*$ and $D_*$ are not tight for any $A$ with $D^*$ converging to $0.5$ as $A \rightarrow \infty$.  We also see that the area scaling approximation (\ref{eq:D_area_scaling}) agrees for $A = 1$, at which $D$ is isotropic. In Figure \ref{fig:eff2_ev} we plot the maximal and minimal eigenvalues $D_{max}$ and $D_{min}$ of the effective diffusion tensor.  As predicted by (\ref{eq:det_relation}), $\frac{1}{Z}$ lies between $D_{max}$ and $D_{min}$, meeting at $A = 1$.
\\\\
In Figure \ref{fig:eff3} we consider a surface given by a periodic tiling of the standard ``bump" function with center $c = \left(\frac{1}{2}, \frac{1}{2}\right)$, radius $r = 0.45$ and amplitude $A$, that is
\begin{equation}
\label{eq:bump_function}
\begin{split}
	h({x}) = A\exp\left(-\frac{1}{1 - |\frac{{x} -c}{r}|^{2}}\right) ,&\qquad |{x} - c| \leq r	\\
	h({x}) = 0	&\qquad |{x}-c| > r.
\end{split}
	\end{equation}
 A plot of the corresponding multiscale surface for $\epsilon=\frac{1}{3}$ is shown in Figure \ref{fig:bump}.  It is clear that the symmetries of the surface fluctuations will induce an isotropic effective diffusion tensor.  We note from Figure \ref{fig:eff3} that for $A < 1.0$, the effective diffusion is not very sensitive to changes in amplitude, but that it rapidly diminishes as we increase $A$ beyond $2$. 

\begin{figure}[h!]
	\includegraphics[scale=0.45]{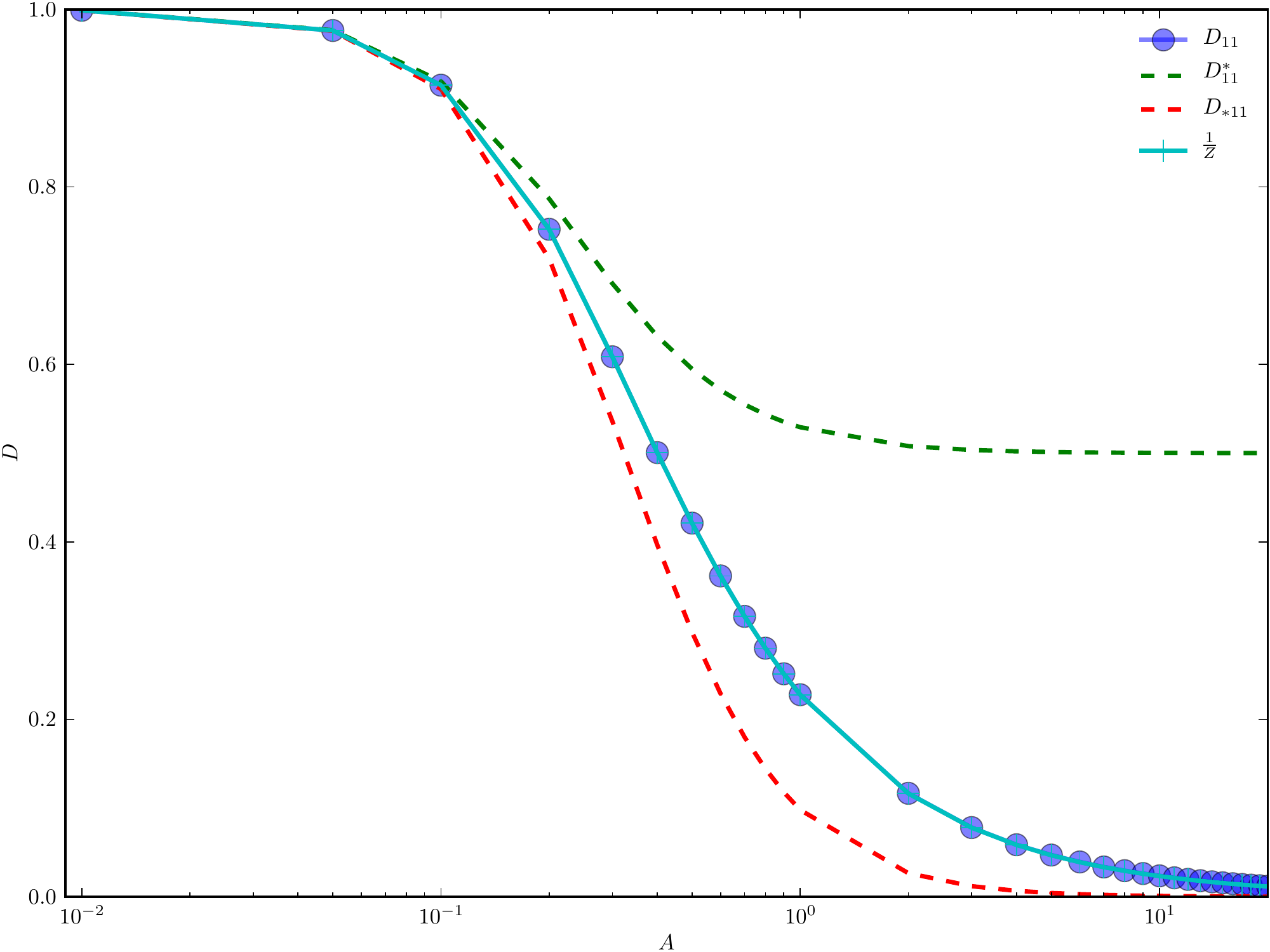}
	\caption[Effective diffusion tensor for a periodic bump type surface]{The effective diffusion tensor for an ``egg-carton" surface with Monge gauge $h(x) = A\sin(2\pi x_1)\sin(2\pi x_2)$ for varying $A$. The dots indicate computed values of $D$.  The dash-dotted line shows $D_{*}$ and the dashed line shows $D^{*}$.}
	\label{fig:eff1}
\end{figure}
\begin{figure}[h!]
	\includegraphics[scale=0.45]{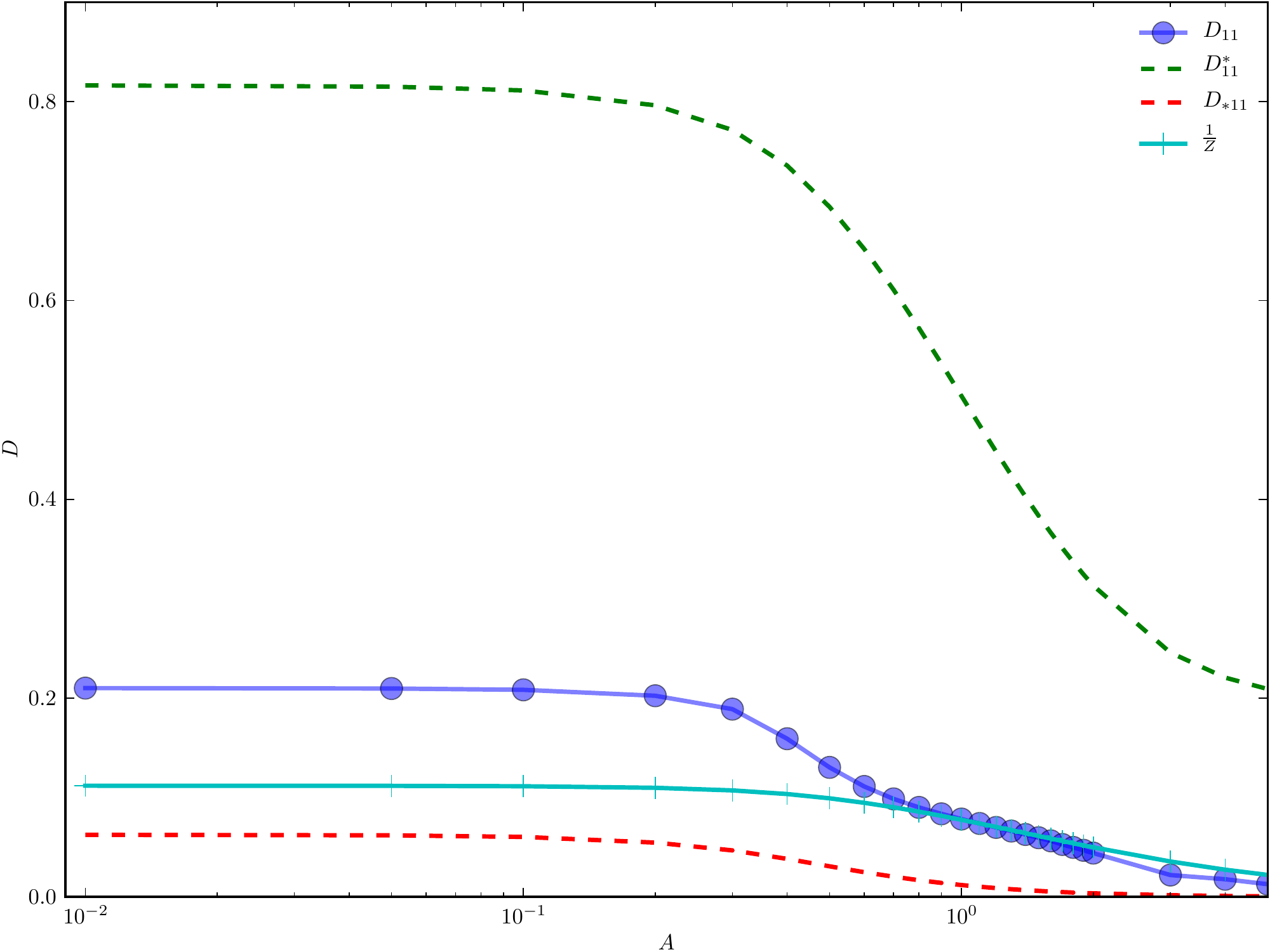}
	\caption[Effective diffusion tensor for a static, periodic non-symmetric surface.]{The effective diffusion tensor for $h(x) = \sin(2\pi x_1)\sin(6\pi x_2) + A  \sin(6\pi x_1)\sin(2\pi x_2)$.  We plot $D$ in the $e_1$ direction.  }
\label{fig:eff2}
\end{figure}
\begin{figure}[h!]
\includegraphics[scale=0.45]{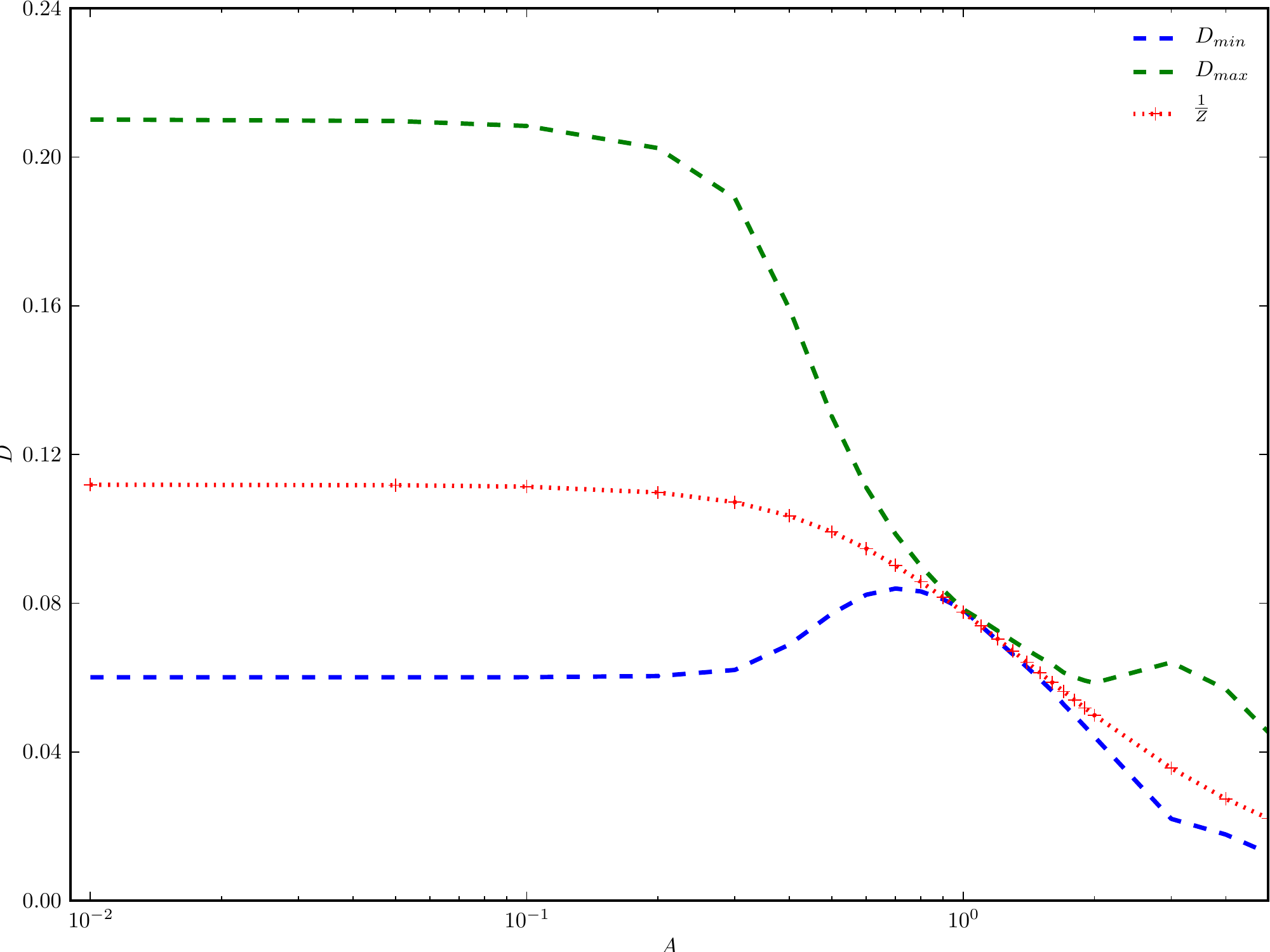}
\caption[Eigenvalues of $D$ for a static, periodic non-symmetric surface]{The effective diffusion tensor for $h(x) = \sin(2\pi x_1)\sin(6\pi x_2) + A  \sin(6\pi x_1)\sin(2\pi x_2)$. $D$ is  anisotropic except for $A = 1$.  The maximal and minimal eigenvalues of $D$, $D_{max}$ and $D_{min}$ respectively, are plotted along with the area scaling approximation $D_{as}$, illustrating the bound on the eigenvalues given by (\ref{eq:det_relation}).}
\label{fig:eff2_ev}
\end{figure}
\begin{figure}[h!]
	\includegraphics[scale=0.45]{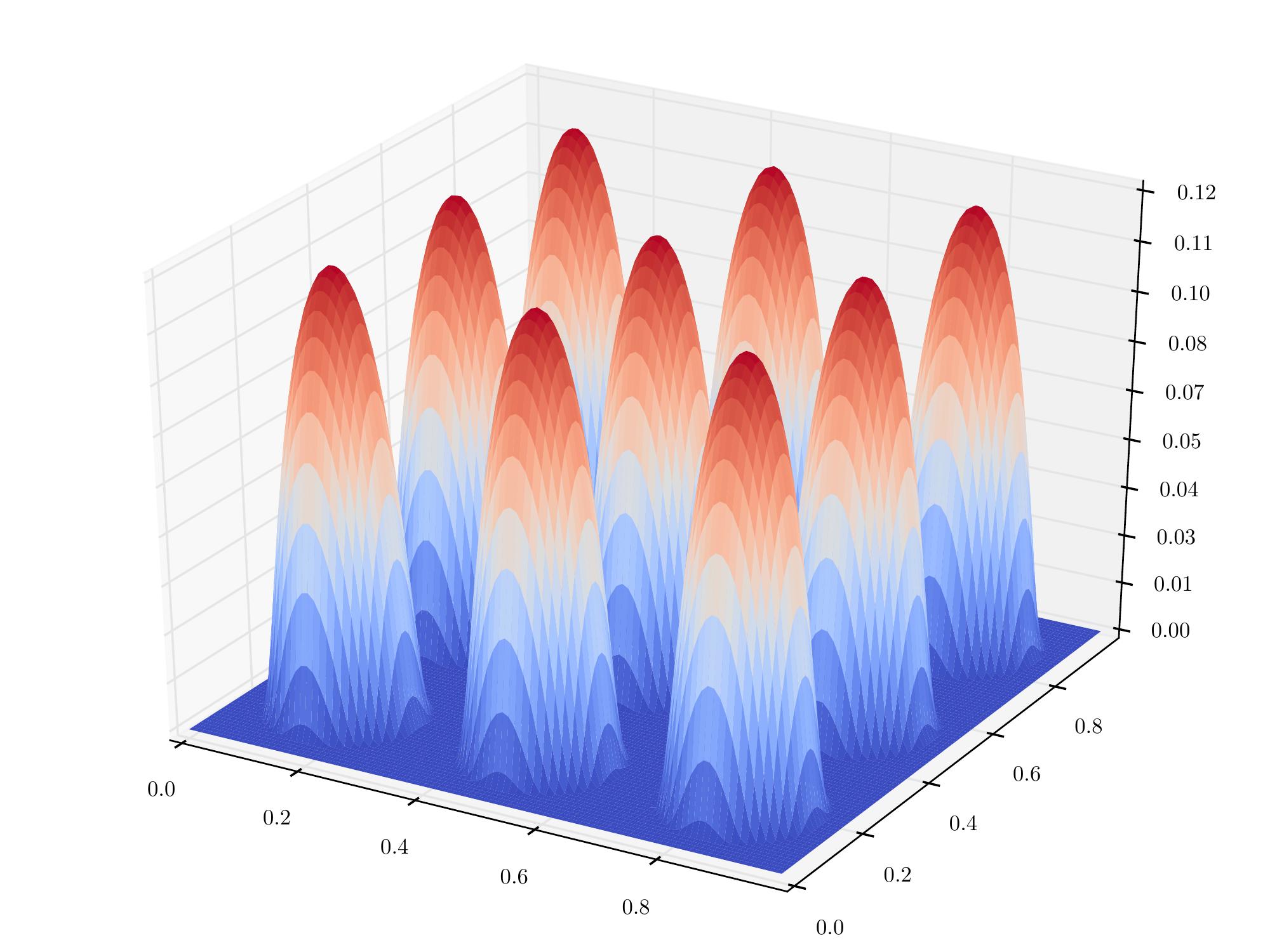}
	\caption[Plot of bump function]{A plot of the periodic bump surface with Monge Gauge $h^\epsilon(x)$ for $h(x)$ given by (\ref{eq:bump_function}) and with $\epsilon=\frac{1}{3}$.}
	\label{fig:bump}
\end{figure}

\begin{figure}[h!]
	\includegraphics[scale=0.45]{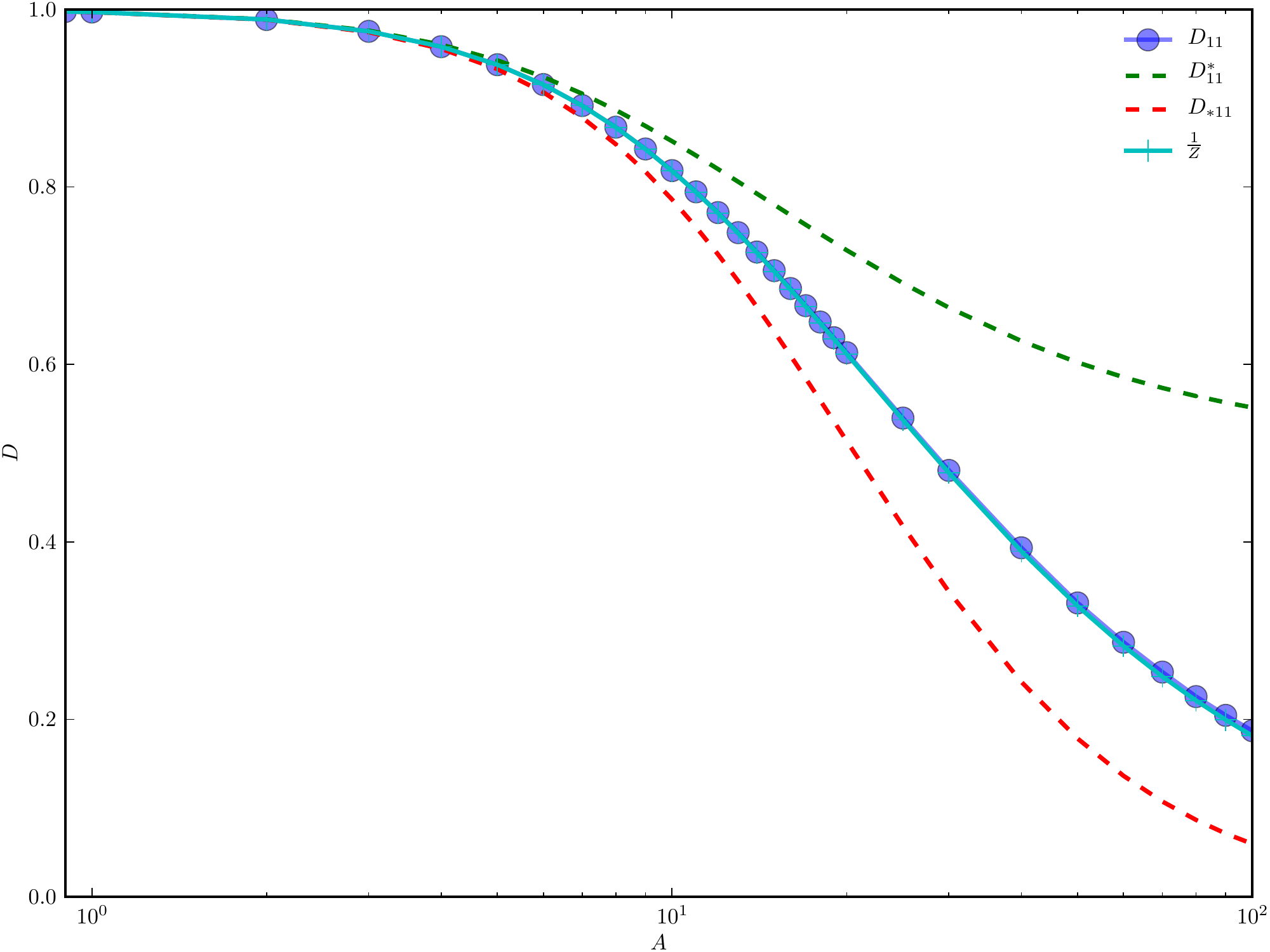}
	\caption[Effective diffusion tensor for a surface consisting of periodically tiled ``bumps".]{The effective diffusion for diffusion on a periodic surface, where each cell is a "bump" with width $0.45$ and amplitude $A$ given by the graph of (\ref{eq:bump_function}).}
	\label{fig:eff3}
\end{figure}

\section{Case I: Diffusion on a Surface with Quenched Fluctuations}
\label{sec:caseI}
We can apply the results of the previous sections to study the effective behaviour of the fluctuating membrane model (\ref{eq:periodic_sde}) in the Case I regime where $(\alpha, \beta) = (1, -\infty)$,  which models a particle diffusing laterally on a static surface obtained by a stationary realisation of the process $\eta(t)$ given in  (\ref{eq:nondim_diffusion_sde_coupled}).  This particular regime had been previously studied in \cite{naji2007diffusion} for lateral diffusion over a Helfrich elastic membrane with quenched fluctuations.  In this section we will study how the distribution of the surface realisation affects the averaged effective diffusion.  In Section \ref{sec:case1_helfrich} we focus on the specific case of a fluctuating Helfrich elastic surface.
\\\\
By homogenizing the diffusion process over a stationary realisation of the surface, we obtain a homogenized diffusion tensor $D(h)$ dependent on the particular realisation.  We will define the effective diffusion coefficient $\overline{D}$ to be the value of $D(h)$ averaged over all stationary surface realisations.
\\\\
Consider the stationary measure $\mu_{\eta}$ of the OU process given by (\ref{eq:mu_infty}).  Let $\mathbb{P}$  be the probability measure on $C(\T^2)$ given by the pushforward of $\mu_\eta$ under the map $P:\R^K \rightarrow C(\T^2)$ where
$$P(\zeta) = \zeta \cdot e = \sum_{k \in \mathbb{K}}\zeta_k e_k. $$
Then the effective diffusion tensor $\overline{D}$ is given by
\begin{equation}
	\label{eq:periodic_D_av}
	\overline{D} = \int D(h) \P(dh).
\end{equation}
The first result we show is an analogue of Proposition \ref{thm:isotropy_periodic}.

\begin{proposition}
	\label{prop:isotropy_stat_random}
 	Suppose $d= 2$  and let $Q \in \R^{2\times 2}$ be an orthogonal matrix, such that $Q \neq \pm I$.  Define $\mathcal{Q}:C(\T^2) \rightarrow C(\T^2)$ by $$\left(\mathcal{Q}f\right) = f(Q^\top \cdot).$$
	Suppose $\P$ is invariant with respect to $\mathcal{Q}$, that is, $\mathcal{Q}^{-1}\circ \P = \P,$
	then $\overline{D}$ is isotropic.
\end{proposition}
\begin{proof}
	Let $h$ be a realisation of $\P$.  Similar to the proof of Proposition \ref{thm:isotropy_periodic}, we have that
	$$\nabla \left(\mathcal{Q} h\right)(x) = Q\, \nabla h(Q^\top\, x).$$
	Denoting by $g(x,h)$ the metric tensor for the graph of $h$ evaluated at $x \in \R^d$, that is $g(x, h)  := I + \nabla h(x) \otimes \nabla h(x)$, then
	$$g^{-1}(x, \mathcal{Q}h) = Q\, g^{-1}(Q^\top\, x, h)\,Q^\top$$
	and 
	$$ \norm{g}(x,\mathcal{Q}h) =  \norm{g}(Q^\top\, x,h).$$
	Let $D(h)$ be the homogenized diffusion tensor for a particular realisation $h$ of $\P$. Then using a similar argument to that of Proposition \ref{thm:isotropy_periodic} gives
	\begin{equation*}
	\begin{split}
&		e\cdot D(\mathcal{Q}h)e \\
&= \inf_{v \in H^1_{per}(\T^2)}\int_{\T^2}\left( \nabla v(x) + e\right)\cdot Q\, g^{-1}(Q^\top x, h)Q^\top\left( \nabla v(x) + e\right)\sqrt{\norm{g}(Q^\top x, h)}\, dy \\	
	&= \inf_{w \in H^1_{per}(\T^2)}\int_{\T^2}\left( Q^\top\nabla w(Q^\top x) + Q^\top e\right)\cdot g^{-1}(Q^\top x, h)\left( Q^\top\nabla w(Q^\top x) + Q^\top e\right)\sqrt{\norm{g}(Q^\top x, h)}\, dy\\
&= e\cdot Q\, D(h) Q^\top e.
	\end{split}
	\end{equation*}

The result follows immediately from the previous relation since using the invariance of the measure $\P$ with respect to $\mathcal{Q}$:
	\begin{equation}
		\begin{split}
			\overline{D} &= \int D(h) \P(dh) \\
				    &= \int D(\mathcal{Q} h) \P(dh) \\
				    &= \int Q D(h) Q^\top \P(dh) \\
				    &= Q \,\overline{D}\, Q^\top.
		\end{split}
	\end{equation}
It follows from Schur's lemma that $\overline{D}$ is isotropic.
\smartqed \qed
\end{proof}

Although $\overline{D}  = \mathbb{E}_{P}\left[D(h)\right]$ is isotropic, one cannot directly apply the area scaling approximation from Section \ref{sec:static_isotropy} to obtain a closed-form expression for $D$.  Two estimates were proposed for $\overline{D}$ in \cite{naji2007diffusion}, namely the averaged area scaling estimate $\overline{D}_{as} = \mathbb{E}_{\P}\left[\frac{1}{Z(h)}\right] \mathbf{I}$ and the effective medium approximation $\overline{D}_{ema} = \mathbb{E}_{\P}\left[\frac{Z(h)}{\int \norm{g}(y,h)\,dy}\right] \mathbf{I}$.  Based on numerical experiments,  the authors conclude that the area scaling estimate $D_{as}$ gives the best agreement with $D$.  
\\\\
Let $Q$ be a $90^\circ$ rotation and define $\mathcal{Q}:C(\T^d) \rightarrow C(\T^d)$ to be $$\mathcal{Q} \left(h\right) = h(Q\cdot).$$ Using a formal regular perturbation argument, the following result shows that in the low disorder limit when $\delta = \mathbb{E}_{\P}\norm{\nabla h}^2 \ll 1$, the averaged diffusion tensor $\overline{D}$ is well approximated by $D_{as}$, provided the random surface field measure $\P$ is invariant under $\mathcal{Q}$:
\begin{theorem}
\label{thm:as_quenched}
	Suppose that 
	\begin{equation}
		\label{eq:as_quenched_condition}
		\P\circ \mathcal{Q}^{-1} = \P,
	\end{equation}
	and that $\delta = \mathbb{E}_{\P}\left[\norm{\nabla h(y)}^2\right] \ll 1$, then for any unit vector $e \in \R^2$,
	\begin{equation}
		e\cdot \overline{D}e = \overline{D}_{as}+ O(\delta^{2}),
\end{equation}
where $\overline{D}_{as} = \mathbb{E}_{\P}\left[\frac{1}{Z(h)}\right]$.
\end{theorem}
\begin{proof}

We look for solutions of the cell equation
\begin{equation}
	\label{eq:celleqn_e}
	\nabla\cdot\left(\sqrt{\norm{g}(y,h)}g^{-1}(y,h)\left(\nabla \chi^{e}(y,h) + e\right)\right) = 0.
\end{equation}
We look for solutions $\chi^{e}$ of (\ref{eq:celleqn_e}) satisfying $\int_{\T^{2}}\chi^{e}(y, h)\,dy = 0$ and in the form of a power series in $\delta$
 \begin{equation}
 	\chi^{e}(y,h) = \chi^{e}_0(y,h) + \delta \chi^{e}_2(y, h) + O(\delta^{2}).
\end{equation}
Write $\nabla h(y)  = \delta^{\frac{1}{2}} \nabla h_0(y)$, where $\mathbb{E}[\norm{\nabla h_0}^2] = 1$.
By Taylor's theorem we have that
\begin{equation}
\label{eq:asympt1}
	\sqrt{1 + \delta\norm{\nabla{h_0}(y)}^{2}} = 1 + \frac{\delta\norm{\nabla{h_0}(y)}^{2}}{2} + O(\delta^{2}).
\end{equation}
Similarly we can write
\begin{equation}
\label{eq:asympt2}
\sqrt{\norm{g}(y,h)}g^{-1}(y,h) =  \left(1 - \frac{\delta\norm{\nabla{h_0}(y)}^{2}}{2}\right)\left( I + \delta H(y,h_0)\right) + O(\delta^{2}),
\end{equation}
  where ${H}(y, h_0) =  (\nabla h_0)^{\bot}\otimes{(\nabla h_0)^{\bot}}$.  Substituting (\ref{eq:asympt1}) and (\ref{eq:asympt2}) into (\ref{eq:celleqn_e}) neglecting terms of order $\delta^{2}$ and higher, it follows that:
\begin{eqnarray*}
-&\nabla&\cdot\left(\left( 1 - \frac{\delta\norm{\nabla{h_0}(y)}^{2}}{2}\right)\left( I + \delta{H}(y, h_0)\right)\left(\nabla\chi^{e}_{0}  + \delta\nabla\chi^{e}_{2} \right)\right) = \\ & &  -\nabla\cdot\left(\left( 1 - \frac{\delta\norm{\nabla{h_0}(y)}^{2}}{2}\right)\left( I + \delta{H}(y, h_0)\right)e\right)
\end{eqnarray*}
Collecting terms of order $0$ we get:
\begin{equation}
	-\Delta \chi^{e}_{0}(y,h_0) = \nabla\cdot e = 0,
\end{equation}
which implies that $\chi^{e}_{0} = 0$ as expected.  Similarly, collecting $O(\delta)$ terms:
\begin{equation}
	\label{eq:chi_e_2}
	-\Delta \chi^{e}_{2} = \nabla\cdot\left({H}(y, h_0)e - \frac{\norm{\nabla h_0(y) }^{2}}{2}e\right).
\end{equation}
Since the integral of the right hand side is 0,  by the Fredholm alternative there is a unique solution $\chi_{e,2}$  with mean zero.  The effective diffusion tensor $D$ can be computed from $\chi_{e}$ as follows
\begin{equation*}
	\begin{aligned}
	\label{eq:diff}
	e\cdot D(h) e = &\frac{1}{Z(h)}\int_{\T^{2}}e\cdot g^{-1}(y, h)e \sqrt{\norm{g}(y,h)}\,dy \\ -  &\frac{1}{Z(h)}\int_{\T^{2}}\nabla \chi^{e}(y,h)\cdot g^{-1}(y,h)\nabla \chi^{e}(y,h) \sqrt{\norm{g}(y,h)}\,dy.
	\end{aligned}
\end{equation*}
  Substituting the above expansions in (\ref{eq:diff}) 
\begin{equation*}
\label{eqn:eff_diff_asympt}
\begin{split}
e\cdot D(h)e &= \frac{1}{Z(h)}\int_{\T^{2}} e\cdot \left(1 - \frac{\delta\norm{\nabla{{h_0}(y)}}^{2}}{2} +O(\delta^{2})\right)\left( I + \delta{H}(y,h_0)\right) e\, dy  \\&-  \frac{\delta^{2}}{Z(h)}\int_{\T^{2}}\nabla \chi^e_{2}(y,h_0)\cdot \left(1 - \frac{\delta\norm{\nabla{{h_0}(y)}}^{2}}{2} + O(\delta^{2})\right)\left( I + \delta{H}(y, h_0)\right)\nabla \phi^e_{2}(y)\,dy.
\end{split}
\end{equation*}
Collecting terms of equal powers of $\delta$:
\begin{equation*}
e\cdot D(h)e = \frac{1}{Z(h)} + \frac{\delta}{Z(h)}\int_{\T^{2}}e\cdot \left({H}(y, h_0) - \frac{\norm{\nabla {h_0}(y)}^{2}}{2}\right) e \,dy + \delta^{2}K(h_0).
\end{equation*}
Taking expectation with respect to $\P$ and applying Fubini's theorem, we see that the $O(\delta)$ term is given by
\begin{equation}
	\label{eq:periodic_alpha2}
	\begin{aligned}
	\int_{\T^2} e\cdot \left(\int\,\frac{1}{Z(h)}\left[H(y, h_0) - \frac{\norm{h_0(y)}^2}{2}I\right] \, \P(dh)  \right)e\, dy.
	\end{aligned}
\end{equation}
By the assumption of invariance with respect to $\mathcal{Q}$, it follows that (\ref{eq:periodic_alpha2}) equals 0.  Therefore,  taking expectation on both sides we have that
$$\overline{D} = \mathbb{E}_{\P}\left[D(h)\right]  = \mathbb{E}_{\P}\left[\frac{1}{Z(h)}\right] + \delta^2 \mathbb{E}_{\P}\left[K(h)\right].$$   Finally, it is straightforward to see that $\mathbb{E}_{\P}\left[K(h)\right]$ is $O(1)$ with respect to $\delta$, so that the result follows.
\smartqed \qed
\end{proof}

It follows by Taylor expanding $\overline{D}_{as}$ that 
\begin{equation}
	\label{eq:weak_disorder_approx}
	\overline{D} = 1 - \frac{1}{2}\delta + O(\delta^2),
\end{equation}
for $\delta = \mathbb{E}_{\mathbb{P}}\norm{\nabla h(y)}^2,$  which gives a first order approximation for $\overline{D}$ in the weak disorder limit (i.e. where $\norm{\nabla h}$ is small).

\subsection{Diffusion on a Helfrich Surface in the $(\alpha , \beta) = (1, -\infty)$ Regime}
\label{sec:case1_helfrich}
We can apply the results of the previous section to study the macroscopic behaviour of particles diffusing on a two dimensional quenched Helfrich elastic membrane.  To this end,  as in Section \ref{sec:helfrich} we set $\mathbb{K} = \lbrace k \in \mathbb{Z}^2\setminus \lbrace (0,0)\rbrace \, | \, \norm{k} \leq c \rbrace,$  and set the coefficients of $\eta^\epsilon(t)$ to be
	$\Gamma = \mbox{diag}\left(\Gamma_k\right)_{k \in \mathbb{K}}$  and $\Pi = \mbox{diag}\left(\Pi_k\right)_{k \in \mathbb{K}}$,
where $\Gamma_k$ and $\Pi_k$ are given by (\ref{eq:helfrich_drift}) and (\ref{eq:helfrich_diffusion}) respectively.  The spatial variation is then determined by $\lbrace e_k \rbrace_{k \in \mathbb{K}}$ where $e_k(x) = e^{2\pi i x}$.
\\\\
Since $\Gamma_k$ and $\Pi_k$ depend only on $\norm{k}$, the condition for Proposition \ref{prop:isotropy_stat_random} holds trivially, and so the average effective diffusion tensor is isotropic although individual realisations are not isotropic, with the anisotropy growing as ${\kappa}^*$ and $\sigma^*$ approaching zero.  Moreover,  for large values of $\kappa^*$ and $\sigma^*$, we expect that the averaged effective diffusion tensor $\overline{D}$ is well approximated by $\overline{D}_{as}$,  by Theorem \ref{thm:as_quenched}.
\\\\
To verify these two predictions we approximate $\overline{D}$ numerically for various parameter values.   Realisations of the stationary surface field were generated by sampling the Fourier modes $\eta_k$ from their respective invariant distribution and performing a Fast Fourier Transform.  For each realisation of the surface, $D(h)$ was computed using the numerical scheme described in Section \ref{sec:static_numerics}.  In Figure \ref{fig:case1_sigma_0_100} we plot $\overline{D}$ for varying bending modulus $\kappa^*$, surface tension set to  $\sigma^* = 0$, $100$ and $500$ and $K = 32$. The effect of $\kappa^*$ and, to a lesser extent $\sigma^*$ on the variance of the effective diffusion tensor is clear.  We also plot the averaged area scaling approximation for this case.  As predicted by Theorem \ref{thm:as_quenched} for large values of $\kappa^*$,  which corresponds to the weak disorder limit, that is $\mathbb{E}_{\P}\left[\norm{\nabla h}^2\right] \ll 1$,  the averaged area scaling approximation $\overline{D}_{as}$ provides a good approximation to $\overline{D}$, but as $\kappa^* \rightarrow 0$ the disparity between $\overline{D}$ and $\overline{D}_{as}$ increases,  with $\overline{D}_{as}$ underestimating  the average diffusion tensor.  
 
 \begin{figure}[h!]
  \centering
	\includegraphics[width=\textwidth]{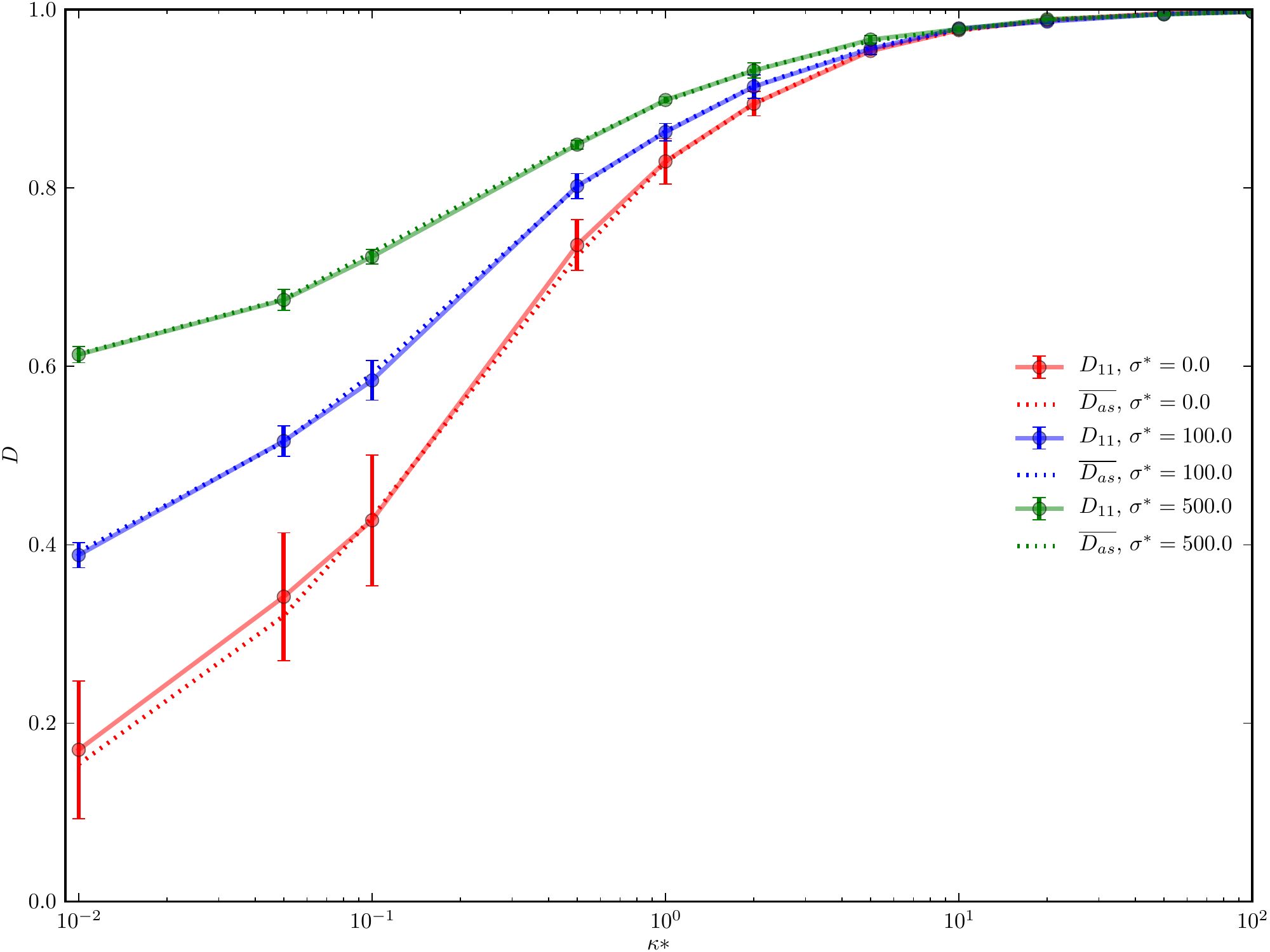} 
\label{fig:case1_sigma_0_100}
\caption[Effective diffusion tensor for a Helfrich elastic membrane in the Case I regime.]{Plot of the distributions of the isotropic effective diffusion tensor $D$ for a quenched realisation of a Helfrich surface, with $K = 32$, and $\sigma^* = 0$, $100$ and $100$. Dots denote the mean of the distribution for each $\kappa^*$ while error bars denote the standard deviation.  The dotted line denotes the average area scaling approximation.}
\end{figure}

\section{Case II: Diffusion on surfaces possessing purely temporal fluctuations}
\label{sec:caseII}
In this section we study the Case II regime, where the fast-scale fluctuations are entirely temporal, corresponding to $(\alpha ,\beta) =  (0 ,1)$ in equation (\ref{eq:periodic_sde}). In Section \ref{sec:av_formal} we use formal expansions to identify the drift and diffusion tensors of the annealed limit process, which are given by the ergodic averages of the drift and diffusion tensors of the multiscale problem. The subsequent sections will then focus on the Helfrich elastic model where we derive exact and asymptotic expressions for the effective diffusion tensor providing a rigorous justification of the ``preaveraging" approximation described in   \cite{reister2007lateral,gustafsson1997diffusion,naji2007diffusion}.
 
\subsection{Averaging Result}
\label{sec:av_formal} 
 In this regime,  (\ref{eq:periodic_sde}) can be written as
\begin{subequations}
\label{eq:nondim_diffusion_sde_averaging}
\begin{align}
d{X^\epsilon}(t) &= F(X^\epsilon(t), \eta^\epsilon(t))\, dt + \sqrt{2\Sigma\left(X^\epsilon(t),\eta^\epsilon\left(t\right)\right)}\,dB(t), \\
	d\eta^\epsilon(t) &=\ -\frac{1}{\epsilon}\Gamma \eta + \sqrt{\frac{2}{\epsilon}\Gamma\Pi}\,dW(t),
	\end{align}
\end{subequations}
where $F$ and $\Sigma$ are given by (\ref{eq:drift_term}) and (\ref{eq:diffusion_term}) respectively and where $B(\cdot)$ is a standard $d$-dimensional Brownian motion. The process $W(\cdot)$ is a standard $K$-dimensional Brownian motion.  The generator of the fast process $\eta^\epsilon(t)$ is $\frac{1}{\epsilon}\mathcal{L}_0$, with
\begin{equation}
	\notag
\label{eq:case2_generator}
\mathcal{L}_{0}f(\eta) = -\Gamma \eta \nabla f(\eta) + \Gamma \Pi:\nabla \nabla f(\eta), \quad f \in C^2_b(\R^d).
\end{equation}
The fast process $\eta^\epsilon(t)$ is geometrically ergodic with invariant distribution $\mathcal{N}(0, \Pi)$.  In particular $$\mathcal{N}[\mathcal{L}_0] = \lbrace \mathbf{1} \rbrace \quad \mbox{ and } \quad
\mathcal{N}[\mathcal{L}_0^*] = \lbrace \rho_\eta \rbrace,$$
where 
\begin{equation}
	\label{eq:timedep_ou_invariant}
\rho_\eta(\eta) = \frac{1}{\sqrt{\left(2\pi\right)^d\, \norm{\Pi}}}\exp\left(-\frac{1}{2}\eta \cdot \Pi^{-1} \eta\right),
\end{equation}  
The corresponding backward Kolmogorov equation for this coupled system is given by 
\begin{subequations}
\begin{align}
\label{eq:kbe_case3}
	\frac{\partial v^\epsilon}{\partial t}(x, \eta, t) & = \mathcal{L}^\epsilon v^\epsilon(x,\eta, t), \qquad (x,\eta, t) \in \R^d\times \R^K \times (0,T] \\
	v^\epsilon(x,\eta, 0) &= v(x, \eta), \qquad (x, \eta) \in \R^d\times \R^K.
\end{align}
\end{subequations}
where
\begin{equation}
\mathcal{L}^\epsilon = \frac{1}{\epsilon}\mathcal{L}_{0} + \mathcal{L}_{1},
\end{equation}
for
\begin{equation*}
\begin{split}
\mathcal{L}_{1} f(x, \eta) &= \frac{1}{\sqrt{\norm{g}(x,\eta)}}\nabla \cdot \left(\sqrt{\norm{g}(x,\eta)}g^{-1}(x)\nabla f(x,\eta)\right).
\end{split}
\end{equation*}
We now state the averaging result for this regime.  A formal justification using perturbation expansions is provided in Appendix \ref{sec:case2_app}.  For a rigorous justification refer to \cite{duncan2013thesis}.

\begin{theorem}
	\label{thm:homog_case3}
	Let $T > 0$ and suppose that $\eta^\epsilon(t)$ is stationary.  Then the process $X^\epsilon$ converges weakly in $C([0,T]\,; \,\R^d)$  to a Wiener process $X^0(t)$ which is the unique  solution of the following It\^{o} SDE:

\begin{equation}
	\label{eq:case2_homogenized_sde}
	dX^0(t) = \overline{F}(X^0(t))dt + \sqrt{2\overline{\Sigma}(X^0(t))} \,dB(t),
\end{equation}
where 
\begin{equation}
	\label{eq:eff_drift_av}
\overline{F}(x) =  \int_{\R^K}F(x, \eta)\,\rho_\eta({d\eta})
\end{equation}
and \begin{equation}
	\label{eq:eff_diffusion_av}
	\overline{\Sigma}(x) =  \int_{\R^K}\Sigma(x, \eta) \rho_{\eta}(d\eta).
\end{equation}
  Moreover, assume that the backward Kolmogorov equation (\ref{eq:kbe_case3}) has  initial data $v$ independent of $\eta$ such that $v \in C_b^2(\R^d)$, then the solution $v^\epsilon$ of  (\ref{eq:kbe_case3}) converges pointwise to the solution $v^0$ of the following PDE,  

\begin{equation}
	\notag
\begin{aligned}
	\frac{\partial v^0}{\partial t}(x,t) & = \overline{F}(x)\cdot \nabla v^0(x,t) + \overline{\Sigma}(x):\nabla \nabla v^0(x,t), \qquad (x, t) \in \R^d \times (0,T]\\
	v^0(x, 0) &= v(x), \qquad x \in \R^d.
\end{aligned}
\end{equation}

 uniformly with respect to $t$ over $[0,T]$.
\end{theorem}
\smartqed\qed

\subsection{Diffusion on a Helfrich Surface in the $(\alpha , \beta) = (0, 1)$ Regime}
\label{sec:timedep:case2}

We can apply Theorem \ref{thm:homog_case3} to obtain the annealed limit equations for diffusion on a rapidly fluctuating Helfrich elastic membrane.   Indeed, we will show that as $\epsilon \rightarrow 0$, the process $X(\cdot)$ converges weakly to a pure diffusion process with constant diffusion tensor.   To this end,  as in Section \ref{sec:helfrich} we set $\mathbb{K} = \lbrace k \in \mathbb{Z}^2\setminus \lbrace (0,0)\rbrace \, | \, \norm{k} \leq c \rbrace,$  and set the coefficients of $\eta^\epsilon(t)$ to be
$$ \Gamma = \mbox{diag}\left(\Gamma_k\right)_{k \in \mathbb{K}} \mbox{ and } \Pi = \mbox{diag}\left(\Pi_k\right)_{k \in \mathbb{K}},$$
where $\Gamma_k$ and $\Pi_k$ are given by (\ref{eq:helfrich_drift}) and (\ref{eq:helfrich_diffusion}) respectively.  The spatial functions $\lbrace e_k \rbrace_{k \in \mathbb{K}}$ are given by the standard $L^2(\T^2)$ Fourier basis $e_k(x) = e^{2\pi i x}$.  The invariant distribution of $\eta_k^\epsilon(t)$ is then given by $\mu_k = \mathcal{N}(0, \Pi_k)$.
\\\\
The form of the limiting equation is strongly dependent on the symmetry properties of the stationary random field. 
\begin{lemma}
	\label{lem:helfrich_case2}
	Let $h(x)$ be a stationary realisation of the random field, that is,
	\begin{equation}
	\notag
		h(x) = \sum_{k \in \mathbb{K}} \eta_k e_k(x),
	\end{equation}
	where $(\eta_k)_{k \in \mathbb{K}} \sim \mu_k$. Then for each $x \in \T^2$, the vectors 

\begin{equation*}
	\left(h_{x_1}(x), h_{x_2}(x), h_{x_1x_2}(x), h_{x_1x_1}(x)\right),
\end{equation*} 
and 
\begin{equation*}
	\left(h_{x_1}(x), h_{x_2}(x), h_{x_1x_2}(x), h_{x_2x_2}(x)\right),
\end{equation*} 
are both jointly Gaussian with mean zero, and the components of each vector are independent.
\end{lemma}

\begin{proof}
 	Since a finite linear combination of centered Gaussian random variables is again a centered Gaussian random variable, it is clear that both vectors are centered Gaussian random vectors.  Moreover,  the components of each vector are pairwise uncorrelated.  To see this for $h_{x_1}(x)$ and $h_{x_2}(x)$:
\begin{equation*}
\begin{aligned}
	\mathbb{E}\left[h_{x_1}(x)h_{x_2}(x)\right] &= \mathbb{E}\left[\left(\sum_{k \in \mathbb{K}} (2\pi i k_1)\eta_k e_k(x)\right)\left(\sum_{j \in \mathbb{K}} (2\pi i j_2)\eta_j e_{j}(x)\right)^*\right] \\
	&= (2\pi)^2\sum_{k \in \mathbb{K}} k_1 k_2 \Pi_k.
\end{aligned}
\end{equation*}
Due to the symmetry of $\mathbb{K}$ around $0$, it follows that the term on the RHS is $0$, so that $h_{x_1}(x)$ and $h_{x_2}(x)$ are uncorrelated.  Similar arguments follow for the other pairs of components.
\end{proof}

We now state the limit theorem for diffusion on a rapidly fluctuating Helfrich elastic membrane. A formal derivation of formula (\ref{eq:effdiff_helfrich}) has been derived in \cite{naji2007diffusion} and \cite{reister2007lateral}.

\begin{theorem}
	\label{thm:case2_helfrich}
	Let $T > 0$, the process $X(\cdot)$ converges weakly in $C([0,T]; \R^2)$ to a  Brownian motion with scalar diffusion tensor given by
	\begin{equation}
	\label{eq:effdiff_helfrich}
		D = \frac{1}{2}\left(1 + \bigintsss_{\R^K}\left[\frac{1}{\norm{g}(x, \eta)}\right]\rho_\eta(d\eta)\right) \mathbf{I}.
	\end{equation}
Furthermore the resulting diffusion tensor $D$ is independent of $x$.
\end{theorem}

\begin{proof}
By Proposition \ref{thm:homog_case3}, the process $X(\cdot)$ converges weakly to a process with drift coefficient $\overline{F}(x)$ and diffusion tensor $\overline{\Sigma}(x)$ given by (\ref{eq:eff_drift_av}) and  (\ref{eq:eff_diffusion_av}) respectively.  Consider first the drift coefficient

\begin{equation*}
\overline{F}(x) = \bigints_{\R^K}\left[\frac{(1+ h_{x_1}^{2})h_{x_2 x_2} - 2h_{x_1}h_{x_2}h_{x_1x_2} + (1 + h_{x_2}^{2})h_{x_1x_1}}{(1 + h_{x_1}^{2} + h_{x_2}^{2})^{2}}\left(\begin{array}{c}h_{x_1} \\ h_{x_2}\end{array}\right)\right] \rho_\eta(d\eta),
\end{equation*}
Applying Lemma \ref{lem:helfrich_case2}, every term in the above sum is an odd function of a centered, Gaussian random vector.  Thus each term equals  $0$.
\\\\
Consider now the effective diffusion tensor
\begin{equation*}
	\overline{\Sigma} =  \bigintsss_{\R^K}\left[ \frac{1}{1 + h_{x_1}^{2} + h_{x_2}^{2}}\left(\begin{matrix}
  1 + h_{x_2}^{2} & -h_{x_1}h_{x_2} \\
  -h_{x_1}h_{x_2} & 1 + h_{x_1}^{2} 
 \end{matrix}
\right) \right] \rho_\eta(d\eta).
\end{equation*}
By the symmetry of $h_{x_1}$ and $h_{x_2}$ the off-diagonal terms are also zero, moreover, the diagonal terms are equal.   Thus 
\begin{equation*}
\begin{split}
	 \bigintsss_{\R^K}\left[\frac{1  + h_{x_2}^{2}}{1 + h_{x_1}^{2} + h_{x_2}^{2}}\right]\rho_\eta(d\eta) &= \frac{1}{2} \bigintsss_{\R^K}\left[\frac{1  + h_{x_2}^{2}}{1 + h_{x_1}^{2} + h_{x_2}^{2}}\right]\rho_\eta(d\eta) \\ &+ \frac{1}{2} \bigintsss_{\R^K}\left[\frac{1  + h_{x_1}^{2}}{1 + h_{x_1}^{2} + h_{x_2}^{2}}\right]\rho_\eta(d\eta) \\ &\quad = \frac{1}{2}\left( 1 + \bigintsss_{\R^K}\left[\frac{1}{1 + h_{x_1}^{2} + h_{x_2}^{2}}\right]\rho_\eta(d\eta)\right),
\end{split}
\end{equation*}
as required.  Finally,  
\begin{equation}
	\notag
\begin{split}
	\nabla_x D &=  \left[\bigintsss_{\R^K} \nabla_x\left[\frac{1}{\norm{g}(x, \eta)}\right]\rho_\eta(d\eta)\right]\, \mathbf{I} \\ &=  \left[-2\bigintsss_{\R^K}\left[\frac{\nabla_x \nabla_x h(x, \eta) \nabla_x h(x,\eta)}{\norm{g}(x, \eta)^2}\right]\rho_\eta(d\eta)\right]\, \mathbf{I} = 0,
\end{split}
\end{equation}
by the symmetry arguments of Lemma \ref{lem:helfrich_case2}, so that $D$ is independent of $x$. 
\end{proof}

Besides $\kappa^*$ and $\sigma^*$, the effective diffusion tensor also depends on the ultraviolet cut-off $c$ (or equivalently $K$).  One can observe that 
\begin{equation}
	\notag
	\lim_{K \rightarrow \infty } \bigintsss_{\R^K}\left[\frac{1}{\norm{g}(x, \eta)}\right]\rho_\eta(d\eta)  = 0,
\end{equation}
so that for any fixed $\kappa^*$, $\sigma^*$, the effective diffusion $D$ will approach $\frac{1}{2}$ as $K$ approaches $\infty$. In fact, for fixed $K$, $\sigma^*$ and $\kappa^*$, the effective diffusion tensor $D$ satisfies
	\begin{equation}
	\notag
		\frac{1}{2} < D < 1,
	\end{equation}
	and recalling that the molecular diffusion tensor $D_0$ was rescaled to $1$, this implies that the diffusion is depleted in the limit of $\epsilon \rightarrow 0$.  In the weak-disorder regime (i.e. when $\mathbb{E} \norm{\nabla h}^2 \ll 1$), which corresponds to the large $\kappa^*$ or large $\sigma^*$ regime, it is possible to derive estimates for $D$ by applying Taylor's theorem and using the fact that $\nabla h(x)$ is Gaussian to get, as a first order approximation
\begin{equation}
	\label{eq:case2_asymptotic_D}
D = 1 - \frac{1}{2}\delta + O(\delta^\frac{3}{2}),
\end{equation}
where, as in Section \ref{sec:caseI}, the constant $\delta$ quantifies the surface disorder and is given by:
\begin{equation}
	\notag
	\delta = \mathbb{E}\norm{\nabla h}^2 = \sum_{k \in \mathbb{K}}\frac{1}{ \kappa^* \norm{2 \pi k}^2 +\sigma^*},
\end{equation}
Comparing with the corresponding equation for Case I given in (\ref{eq:weak_disorder_approx}), we see that, to first order, both regimes exhibit the same behaviour in the weak disorder limit.

\subsection{Numerical Examples}
We can study how the bending modulus $\kappa^*$, surface tension  $\sigma^*$ and the ultraviolet cut-off $c$ (or equivalently $K$) affect the effective diffusion tensor numerically.  The ensemble average in (\ref{eq:eff_diffusion_av}) is computed using a straightforward Monte-Carlo method, taking the sample average of $\frac{1}{\norm{g}(x, \eta)}$  for $\eta$ sampled from its corresponding stationary measure. As (\ref{eq:case2_asymptotic_D}) suggests, for small values of $\kappa^*$ and $\sigma^*$, the larger thermal fluctuations of the surface cause a greater reduction in the speed of diffusion of a particle diffusing laterally on the surface. Indeed one can see that $(1 - D) \, \approx \, \frac{1}{{\kappa}^*}$ for fixed ${\sigma}^*$ and $(1 - D) \, \approx \, \frac{1}{{\sigma}^*}$ for fixed ${\kappa}^*$. In Figure \ref{fig:eff_diff_av} we plot  $D$ for varying values of ${\kappa}^*$, $K$ and for ${\sigma}^* = 0, 100, 500$.   The convergence of $D$ to $\frac{1}{2}$ becomes immediately apparent.  As expected, $D$ decays with $c$,  converging to  $\frac{1}{2}$ as $c \rightarrow \infty$.

 \begin{figure}[htpb]
\centering
	\includegraphics[width=\textwidth]{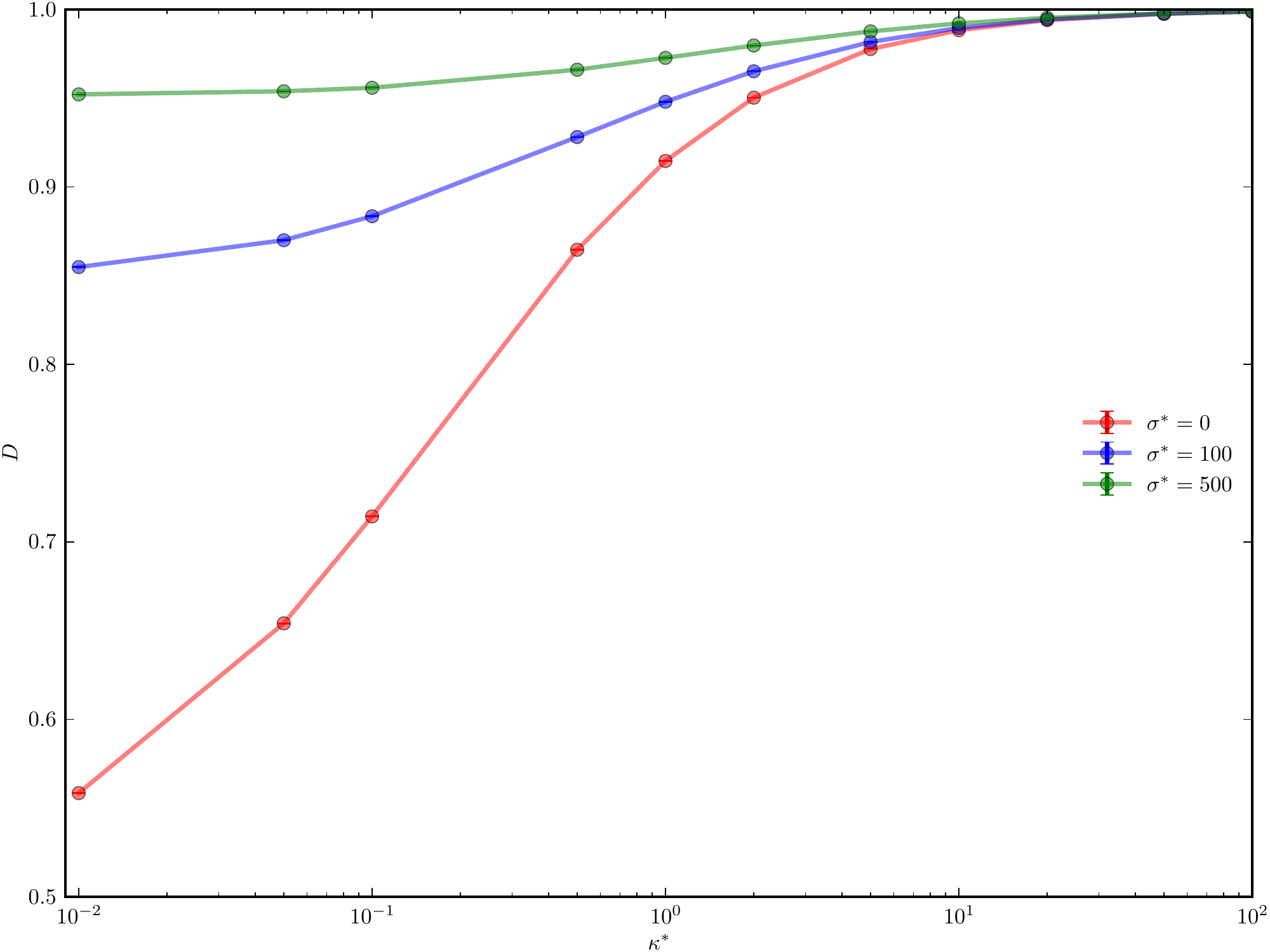} 

	\caption[Effective diffusion tensor for a Helfrich elastic membrane in the Case II regime.]{The effective diffusion tensor for a diffusion on a fluctuating Helfrich elastic membrane in the $(\alpha, \beta) = (0, 1)$ scaling for varying $\kappa^*$ and $K=32$.}
\label{fig:eff_diff_av}
\end{figure}

 \begin{figure}[htpb]
  \centering
	\includegraphics[width=\textwidth]{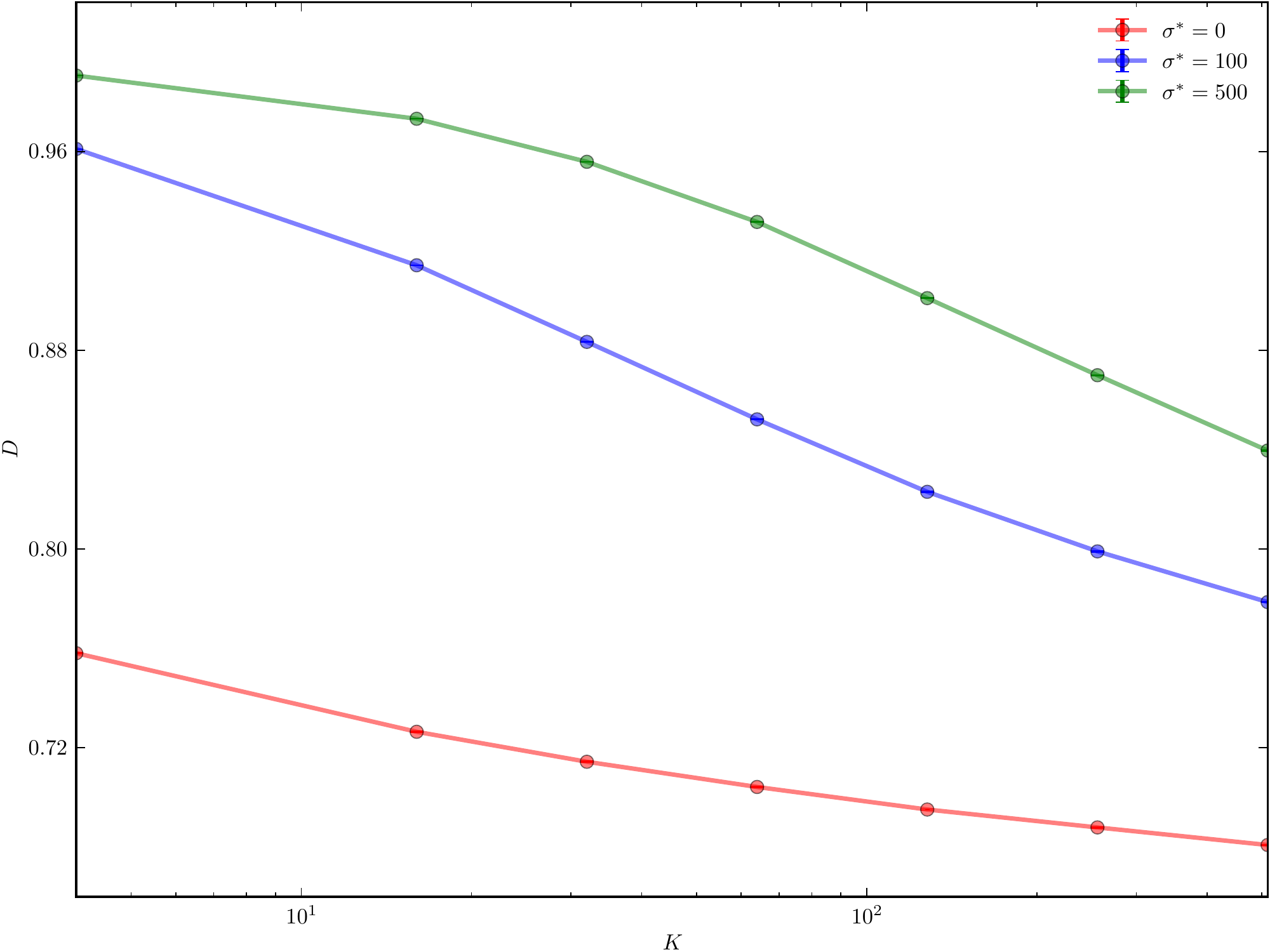} 

	\caption[Effective diffusion tensor for a Helfrich elastic membrane in the Case II regime.]{The effective diffusion tensor for a 		diffusion on a fluctuating Helfrich elastic membrane in the $(0, 1)$ scaling for varying $K$ and $\kappa^*=0.1$.}
	\label{fig:eff_diff_K}
\end{figure}

\section{Case III: Diffusion on surfaces with comparable spatial and temporal fluctuations}
\label{sec:caseIII}
In this section we consider the Case III regime where $(\alpha, \beta) = (1,1)$ in (\ref{eq:nondim_diffusion_sde_coupled}).  This scaling describes lateral diffusion on a rough surface which is also fluctuating rapidly, but the temporal surface fluctuations occur at a scale commensurate to the characteristic scale of the spatial fluctuations.  This scaling was considered for SDEs with periodic spatial and temporal fluctuations in \cite{garnier1997homogenization}.
A unique characteristic of this scaling regime is that it gives rise to a macroscopic drift term in the limit as $\epsilon \rightarrow 0$ which is determined by the rate of change of the corrector $\chi(y, \eta)$ with respect to the temporal fluctuations.
\\\\
A similar effective drift term arises in the model considered here. It is not clear that this drift is identically zero in general.  However, we identify a natural symmetry condition for the surface fluctuations for which we can prove the effective drift term vanishes.   In the remainder of this section we identify and study the properties of the macroscopic diffusion tensor and provide sufficient conditions for the isotropy of the effective diffusion tensor.  Finally, we study the limiting properties of  the Helfrich model in this regime.

\subsection{Homogenization Result}
\label{sec:homogenisation_case2}

Introducing the fast process $Y^\epsilon(t) = \frac{X^\epsilon(t)}{\epsilon} \mod \T^d$, equations (\ref{eq:periodic_sde}) can be written as the following  fast-slow system 
\begin{subequations}
\label{eq:case3_nondim_diffusion_sde}
\begin{align}
d{X^\epsilon}(t) &=\frac{1}{\epsilon}F\left({{Y^\epsilon}(t)}, \eta^\epsilon\left({t}\right)\right)\, dt + \sqrt{2\Sigma\left(Y^\epsilon(t),\eta^\epsilon\left({t}\right)\right)}\,dB(t), \\
d{Y^\epsilon}(t) &=\frac{1}{\epsilon^2}F\left({{Y^\epsilon}(t)}, \eta^\epsilon\left({t}\right)\right)\, dt + \sqrt{\frac{2}{\epsilon^2}\Sigma\left(Y^\epsilon(t),\eta^\epsilon\left({t}\right)\right)}\,dB(t), \\
	d\eta^\epsilon(t) &=\ -\frac{1}{\epsilon}\Gamma \eta^\epsilon(t)dt + \sqrt{\frac{2\Gamma\Pi}{\epsilon}}\,dW(t),
	\end{align}
\end{subequations}
where $F$ and $\Sigma$ are given by (\ref{eq:drift_term}) and (\ref{eq:diffusion_term}) respectively and where we impose periodic boundary conditions on $Y^\epsilon(\cdot)$. The processes $B(\cdot)$ and $W(\cdot)$ are standard $d$ and  $K$-dimensional Brownian motions, respectively.  The infinitesimal generator of the fast process $Y^\epsilon(t)$ is given by $\frac{1}{\epsilon^2}\mathcal{L}_0$,  where $\mathcal{L}_0$ is given by
\begin{equation}
\label{eq:L0}
\mathcal{L}_0 f(y) = \frac{1}{\sqrt{\norm{g}(y,\eta)}}\nabla_{y}\cdot\left(\sqrt{\norm{g}(y,\eta)}g^{-1}(y,\eta)\nabla_y f(y)\right), \quad f \in C^2(\T^d).
\end{equation}
We note that although the spatial and temporal fluctuations appear commensurate in the system of SDEs, the spatial fluctuations relax to equilibrium at a time scale faster than the temporal fluctuations.  The limiting equation can thus be considered the result of a reiterated homogenisation/averaging problem of the form described in \cite[Section 2.11.3]{bensoussan1978asymptotic}.  The limiting equation is thus obtained by  homogenising over $Y^\epsilon(t)$ for a frozen value of $\eta^\epsilon(t)$ and then averaging over the invariant measure $\rho_{\eta}(\cdot)$ of  $\eta^\epsilon(t)$. 
\\\\
For $\eta$ fixed $\mathcal{L}_0$ 
satisfies $$\mathcal{N}[\mathcal{L}_{0}] = \lbrace \mathbf{1} \rbrace\qquad  \mbox{ and } \qquad \mathcal{N}[\mathcal{L}_{0}^*] = \left\lbrace \rho_{y}(y,\eta) \right\rbrace,$$ where $$\rho_{y}(y, \eta) = \frac{\sqrt{\norm{g(y, \eta)}}}{Z(\eta)}$$  for $Z(\eta) = \int_{\mathbb{T}^{2}}\sqrt{\norm{g(y, \eta)}}\, dy.$
\\\\
For $\eta \in \R^K$ fixed, we define the corrector $\chi(y,\eta)$ to be the solution of the following cell equation
\begin{equation}
	\label{eq:case3_celleqn}
	\mathcal{L}_0 \chi(y, \eta) = -F(y, \eta), \qquad (y, \eta) \in \T^d \times \R^K.
\end{equation}

By the Fredholm alternative and elliptic regularity theory there exists a unique, mean zero solution  $\chi \in C^{2\times 2}(\T^d\times \R^K; \R^{d})$ to (\ref{eq:case3_celleqn}).
\\\\
The backward Kolmogorov equation corresponding to the coupled system (\ref{eq:case3_nondim_diffusion_sde}) is given by 
\begin{subequations}
\begin{align}
\label{eq:kbe_case2}
	\frac{\partial v^\epsilon}{\partial t}(x,y, \eta, t) & = \mathcal{L}^\epsilon v^\epsilon(x,y, \eta, t), \qquad (x,y,\eta, t) \in \R^d\times \T^d \times \R^K \times (0,T]	\\
	v^\epsilon(x, y, \eta, 0) &= v(x, y, \eta),\qquad x \in \R^d \times \T^d \times \R^K.
\end{align}
\end{subequations}
where 
\begin{equation}
 \mathcal{L}^\epsilon = \frac{1}{\epsilon^{2}}\mathcal{L}_{0} + \frac{1}{\epsilon}\mathcal{L}_{\eta} + \frac{1}{\epsilon}\mathcal{L}_{1} + \mathcal{L}_{2},
\end{equation}
for 
\begin{equation*}
\mathcal{L}_{1} f(x, y, \eta) = F(y, \eta)\cdot \nabla_{x}f(x,y,\eta) + 2\Sigma(y, \eta):\nabla_{x}\nabla_{y}f(x,y, \eta),
\end{equation*}
and 
\begin{equation*}
\mathcal{L}_{2} f(x, y,\eta) = \Sigma(y, \eta):\nabla_{x}\nabla_{x}f(x,y,\eta),
\end{equation*}
and $\mathcal{L}_\eta$ is the infinitesimal generator of the OU process and is given by
\begin{equation}
\label{eq:L_eta}
\mathcal{L}_{\eta} f(\eta) = -\Gamma\cdot \nabla_\eta f(\eta) + \Gamma\Pi: \nabla_\eta \nabla_\eta f(\eta).
\end{equation}
\\
 The following theorem states the homogenisation result for this scaling.  A formal justification will be given in given in Appendix \ref{sec:case3_app}.  A rigorous proof using probabilistic methods can be found in \cite{duncan2013thesis}.  As in the previous chapters we use the convention that $\left(\nabla_y \chi\right)_{ij} = \frac{\partial \chi^{e_i}}{\partial y_j}$.

\begin{theorem}
\label{thm:homog_case2}
	Let $0 < \epsilon \ll 1$ and $T = \mathcal{O}(1)$, and suppose $\eta^\epsilon(t)$ is stationary.  Then  as $\epsilon \rightarrow 0$, the process $X^{\epsilon}(\cdot)$ converges weakly in $C([0,T]; \R^{d})$ to $X^0$ which is a weak solution of the following It\^{o} SDE 
\begin{equation}
\label{eq:case3_homogenized_sde}
dX^0(t) =  L dt + \sqrt{2D}\,dB(t),
\end{equation}
 where the effective diffusion tensor $D$ is given by
\begin{equation}
\label{eq:case3_eff_diff}
D = 	\int_{\R^K}\int_{\mathbb{T}^{d}}\left(I + \nabla_{y}\chi\right) g^{-1}(y,\eta)\left(I + \nabla_{y}\chi\right)^\top \rho_{y}(y,\eta)\, \rho_\eta(\eta)\, dy \,d\eta,
\end{equation}
and the effective drift term $L$ is given by
\begin{equation}
\label{eq:case3_eff_drift}
L = \int_{\R^K}\int_{\T^d} \mathcal{L}_{\eta}\chi \rho_{y}(y, \eta)\,\rho_{y}(y,\eta)\, \rho_\eta(\eta)\, dy \,\,d\eta.
\end{equation}
Moreover, if the backward equation (\ref{eq:kbe_case2}) has initial data $v$, independent of the fast process, such that $v \in C^2_b(\R^d)$, then the solution $v^{\epsilon}$ of (\ref{eq:kbe_case2}) converges pointwise to the solution $v_{0}$ of
\begin{equation}
\label{eq:homog_eqn_3}
\begin{aligned}
\frac{\partial v^{0}}{\partial t}(x,t)  & =  L\cdot \nabla_x v^0(x,t) + D:\nabla_{x}\nabla_{x}v^{0}(x,t), \qquad (x , t) \in \R^d \times (0, T], \\
v^{0}(x, 0) &= v(x), \qquad x \in \R^d.
\end{aligned}
\end{equation}
uniformly with respect to $t$ over $[0,T]$.
\end{theorem}
\smartqed \qed

\subsection{Properties of the Effective Diffusion Process}
\label{sec:properties_case2}
Comparing the effective behaviour of the homogenized diffusion processes  in Case I and Case III, we see that the introduction of the fast temporal fluctuations gives rise to a time-averaging of the effective diffusion tensor, so that the effective diffusion in the $(\alpha,\beta)=(1, 1)$ case is equal to $\overline{D}$, the averaged effective diffusion tensor for diffusion on a surface with quenched fluctuations,  as described in Section \ref{sec:caseI}.   Thus all the properties proved for $\overline{D}$ hold equally for the effective diffusion tensor $D$.  The following proposition summarizes the most important properties.

\begin{proposition}
\label{prop:case3_properties}
Let $D$ be the effective diffusion given by (\ref{eq:case3_eff_diff}), then
\begin{enumerate}[(i)]
	\item $D$ is a strictly positive definite matrix.
	\item In particular, for a unit vector $e \in \R^d$
		\begin{equation}
			 	0 < \overline{D}_* \leq e \cdot D e \leq \overline{D}^* \leq 1,
		\end{equation}
		where $\overline{D}^* = \mathbb{E}\left[D^*(h)\right]$ and $\overline{D}_* = \mathbb{E}\left[D_*(h)\right]$, for $D^*$ and $D_*$ given by (\ref{eq:bound_upper}) and (\ref{eq:bound_lower}) respectively, and where $\mathbb{E}\left[\cdot\right]$ denotes expectation with respect to the invariant measure of $\eta^\epsilon(t)$.
	\item For $d = 2$, if the condition of Proposition \ref{prop:isotropy_stat_random} holds, then $D$ is isotropic.
	\item If additionally $\mathbb{E}\left[\norm{\nabla h(x)}^2\right] = \delta \ll 1$, then
		$$
			D = \overline{D}_{as} + O(\delta^2),
		$$
	where $\overline{D}_{as} = \mathbb{E}\left[\frac{1}{Z(h)}\right]$.
\end{enumerate}
\end{proposition}
\qed
$ $\linebreak
We turn our attention to the effective drift term $L$ given by (\ref{eq:case3_eff_diff}).  Unlike $D$, the effective drift depends on $\chi(y, \eta)$ which is only unique up to a constant depending on $\eta$.  However,  for any function $c(\eta)$ we have that 
\begin{equation*}
\begin{split} \int \int_{\T^d} \mathcal{L}_\eta c(\eta) \rho_y(y, \eta)\rho_\eta(\eta)\,dyd\eta &=  \int  \mathcal{L}_\eta c(\eta) \left(\int_{\T^d}\rho_y(y, \eta)\,dy\right) \, \rho_\eta(\eta)\,d\eta \\ &=\int \mathcal{L}_\eta c(\eta)\rho_\eta(\eta)\,d\eta = 0,
\end{split}
\end{equation*}
since $\int_{\T^d}\rho_y(y,\eta)\,dy = 1$ for all $\eta$.  It follows that the effective drift $L$ is uniquely defined independent of any additive terms independent of $y$.  
\\\\
The fact that a macroscopic drift $L$ arises in this scaling regime is surprising.  While numerical simulations suggest that $L$ is always zero,  we have not been able to prove this in general.  However, we can show that it is true for surfaces which satisfy the following  natural symmetry condition.   Suppose there exists a  linear  orthogonal map $\mathcal{C}:\R^K \rightarrow \R^K$ which commutes with $\Pi$ and  $\Gamma$ (in particular $\rho_\eta$ is invariant with respect to $\mathcal{C}$) such that
\begin{equation}
	\label{eq:symmetry_condition_drift}
	h(x, \mathcal{C}^\bot \,\eta) = h(x^\bot, \eta),
\end{equation}
where $x^\bot_i = 1 - x_i$ for $i \in \lbrace 1, \ldots d\rbrace$,  or equivalently that 
\begin{equation*}
	\mathcal{C}e(x) = e(x^\bot),
\end{equation*}
where $e(x) = \lbrace{e_k(x)} \rbrace_{k \in \mathbb{K}}$.
\\\\
Condition (\ref{eq:symmetry_condition_drift}) arises naturally in the case where $e_k$ are the Fourier basis for the Laplacian on $[0,1]^2$.  The surface perturbation $h$ can then be rewritten as 
\begin{equation*}
	h(x, \eta) = \sum_{k \in \mathbb{K}_{even}}\eta^e_k e^e_k(x) +  \sum_{k \in \mathbb{K}_{odd}}\eta^o_k e^o_k(x),
\end{equation*}
where $e_k^e$  and $e_k^o$ are respectively even and odd functions on $[0,1]^2$ for all $k \in \mathbb{K}$.  If $\mathcal{C}$ is the diagonal matrix defined by $$\mathcal{C}^\top\eta = \mathcal{C}^\top(\eta^e, \eta^o) = (\eta^e, -\eta^o),$$ for $\eta^e = (\eta_k^e)_{k \in \mathbb{K}_{even}}$ and $\eta^o = (\eta_k^o)_{k \in \mathbb{K}_{odd}}$, we see that condition (\ref{eq:symmetry_condition_drift}) is trivially satisfied.  We can show the following result.
\\
\begin{proposition}
\label{prop:eff_drift}
	Suppose (\ref{eq:symmetry_condition_drift}) holds, then the effective drift $L$ equals $\mathbf{0}$.
\end{proposition}
\begin{proof}
We first note that (\ref{eq:symmetry_condition_drift}) implies that

\begin{equation}
\notag
g^{-1}(x, \mathcal{C}\eta) = g^{-1}(x^\bot, \eta),
\end{equation} and 
\begin{equation}
\notag
\norm{g}(x, \mathcal{C}\eta) = \norm{g}(x^\bot, \eta).
\end{equation}
Consider the cell equation for the corrector $\chi^e(y, \eta)$ given by
\begin{equation}
\notag
\begin{split}
	\nabla\cdot \left(\sqrt{\norm{g}(y,\eta)}g^{-1}(y, \eta)\left(\nabla \chi^e(y, \eta) + e \right)\right) = 0.
\end{split}
\end{equation}
Making the substitution $\eta \rightarrow \mathcal{C}\eta$, then using the relations for $g^{-1}$ and $\norm{g}$ and changing variables in $y$ we have
\begin{equation}
\label{eq:case3_corrector_symmetry}
\begin{split}
	-\nabla\cdot \left(\sqrt{\norm{g}(y,\eta)}g^{-1}(y, \eta)\left(-\nabla \tilde{\chi}^e(y, \eta) + e \right)\right)\Bigg|_{y = y^\bot},
\end{split}
\end{equation}
where $\tilde{\chi}^e(y, \eta) = \chi^e(y^\bot, \mathcal{C}\eta)$. It follows that 
\begin{equation}
\label{eq:case3_corrector_sym}
{\chi}^e(y^\bot, \mathcal{C}\eta) = - \chi^e(y, \eta).
\end{equation}
 Applying (\ref{eq:case3_corrector_symmetry}) and using the fact that $\mathcal{C}$ commutes with $\Gamma$ and $\Pi$, we obtain
 \begin{equation}
\notag
	\begin{split}
		-\mathcal{L}_\eta \chi^e(y^\bot, \eta) &= -\Gamma \eta \cdot \mathcal{C}\nabla_\eta \chi^e(y, \mathcal{C}^\top\eta) + \mathcal{C^\top}\Gamma\Pi \mathcal{C}:\nabla \nabla \chi^e(y, \mathcal{C}^\top\eta) \\
		& = -\Gamma \mathcal{C}^\top\eta \cdot \nabla_\eta \chi^e(y, \mathcal{C}^\top\eta) + \Gamma\Pi:\nabla \nabla \chi^e(y, \mathcal{C}^\top\eta)\\ &= \mathcal{L}_\eta \chi^e(y, \mathcal{C}^\top\eta).
	\end{split}
    \end{equation}  
Using the invariance of $\rho_y$ with respect to $\mathcal{C}$ the effective drift term $V$ will then be given by
\begin{equation}
\notag
\begin{split}
 L &= \int \int_{\T^d} \mathcal{L}_\eta \chi^e(y^\bot,\eta)  \rho_y(y^\bot, \eta)\rho_\eta(\eta)\,dy\,d\eta \\
  &=  -\int \int_{\T^d} \mathcal{L}_\eta \chi^e(y,\mathcal{C}^\top\eta)  \rho_y(y, \mathcal{C}^\top\eta)\rho_\eta(\eta)\,dy\,d\eta \\
  &= -L,
 \end{split}
 \end{equation}
proving the result.
\smartqed \qed
\end{proof}

As an example we can consider the model for a thermally excited  Helfrich surface in the $(\alpha, \beta) = (1, 1)$ scaling.   
It follows from Proposition \ref{prop:case3_properties} that the effective diffusion tensor is isotropic.  Moreover, since the conditions of Proposition \ref{prop:eff_drift} hold, the effective drift is $\mathbf{0}$.  It follows that the diffusion process $X^\epsilon(t)$ converges to a Brownian motion on $\R^2$ with a scalar diffusion  tensor $D$.   As the effective diffusion $D$ is equal to the  effective diffusion tensor $\overline{D}$ of Section \ref{sec:case1_helfrich},  the dependence of $D$ on the parameters $\kappa^*$, $\sigma^*$ and $K$ hold equivalently.

\section{Case IV: Diffusion on surfaces with temporal fluctuations faster than spatial fluctuations}
\label{sec:caseIV}
In this section we consider the $(\alpha, \beta) = (1,2)$ scaling.  In this scaling the surface possesses rapid spatial and temporal fluctuations but  the temporal fluctuations much faster than the spatial fluctuations. Writing $Y^\epsilon(t) := \frac{X^\epsilon(t)}{\epsilon} \mod \T^d$ the fast-slow system for this regime is given by:
\begin{subequations}
\label{eq:case4_nondim_diffusion_sde}
\begin{align}
d{X^\epsilon}(t) &=\frac{1}{\epsilon}F\left({{Y^\epsilon}(t)}, \eta^\epsilon\left({t}\right)\right)\, dt + \sqrt{2\Sigma\left(Y^\epsilon(t),\eta^\epsilon\left({t}\right)\right)}\, dB(t), \\
d{Y^\epsilon}(t) &=\frac{1}{\epsilon^{2}} F\left({{Y^\epsilon}(t)}, \eta^\epsilon\left({t}\right)\right)\, dt + \sqrt{\frac{2}{\epsilon^{2}}\Sigma\left(Y^\epsilon(t),\eta^\epsilon\left({t}\right)\right)}\,dB(t), \\
	d\eta^\epsilon(t) &=\ -\frac{1}{\epsilon^2}\Gamma \eta^\epsilon(t) + \sqrt{\frac{2\Gamma\Pi}{\epsilon^2}}\,dW(t),
	\end{align}
\end{subequations}
where $B(\cdot)$ is a standard $d$-dimensional Brownian motion, and $W(\cdot)$ is a standard $K$-dimensional Brownian motion.  The infinitesimal generator of the underlying fast process is given by $\frac{1}{\epsilon^2}\mathcal{G}$ where
 $$\mathcal{G}f(y, \eta) = \left(\mathcal{L}_0 + \mathcal{L}_\eta\right) f(y,\eta), \qquad f \in C^2_c(\T^d \times \R^K),$$
where $\mathcal{L}_0$ and $\mathcal{L}_\eta$ are given by (\ref{eq:L0}) and (\ref{eq:L_eta}) respectively.  
\\\\
Unlike in the previous cases, it is not immediately clear that the fast process is geometrically ergodic, i.e. that the fast process converges exponentially fast to a unique invariant measure.  Moreover, due to the unbounded support of the surface fluctuations, the infinitesimal generator is no longer uniformly elliptic.   Thus we cannot apply standard elliptic theory to obtain a Fredholm alternative for this operator.   In Proposition \ref{prop:case4_inv} we prove the geometric ergodicity of the fast process.  The proof is a straightforward application of the results in \cite{mattingly2002ergodicity,mattingly2002geometric} which are based on the results of the classical Meyn \& Tweedie theory \cite{meyn2009markov}.   In Proposition \ref{prop:case4_inv} we show that there exists a unique, smooth solution of the Poisson problem for this scaling regime,  provided the centering equation holds.  

\subsection{Homogenization Result}
We first identify the fast process $(Y^\epsilon(t), \eta^\epsilon(t))$ as a rescaling of a $\T^d \times \R^K$-valued process independent of $\epsilon$.   Indeed, define  $(Y(t), \eta(t))$ as follows
\begin{subequations}
	\begin{align}
		dY(t) &= F(Y(t), \eta(t))dt + \sqrt{2 \Sigma(Y(t), Z(t))}\, d\hat{B}(t),\\
		d\eta(t) &= -\Gamma \eta(t) + \sqrt{2\Gamma \Pi}d\hat{W}(t),	
	\end{align}
\end{subequations}
where $\hat{B}(t)$ is a standard $\R^d$-valued Brownian motion,  $\hat{W}(t)$ is a standard $\R^K$-valued Brownian motion.  The joint process $(Y(t),\eta(t))$ has infinitesimal generator $\mathcal{G}$. It is straightforward to show that the following equality holds (in law),
\begin{equation}
	\notag
	\left(Y^\epsilon(t), \eta^\epsilon(t)\right) = \left(Y(t/\epsilon^2), \eta(t/\epsilon^2))\right).
\end{equation}

\begin{proposition}
\label{prop:case4_inv}
	The process $(Y(t), \eta(t))$ possesses a unique invariant measure $\rho$ with smooth, positive density  with respect to the Lebesgue measure on $\T^d \times \R^K$ which is the unique normalizable solution of
\begin{equation}
	\label{eq:density_equation_case4}
	\mathcal{G}^* \rho = 0.
\end{equation}
  Let $P_t$ be the Markov semigroup corresponding to $(Y(t), \eta(t))$.  Then there exists a constant $\mu \in (0,1)$ such that for all functions $f: \T^d\times \R^K \rightarrow \R$, such that \begin{equation}
\label{eq:condition_bound}
|f|(y, \eta) \leq C V(\eta), \qquad (y, \eta) \in \T^d \times \R^K,
\end{equation} where 
\begin{equation}
\label{eq:V}
	V(\eta) := (1 + \norm{\eta}^2)
\end{equation}
 the following estimate holds

\begin{equation}
\label{eq:initial_condition_bound}
\norm{\mathbb{E}^{(y_0, \eta_0)}f(Y(t), \eta(t)) - \int f(y,\eta) \rho(dy, d\eta)} \leq C'V(\eta_0)e^{-\mu t},
\end{equation}
where $\mathbb{E}^{(y_0, \eta_0)}$ denotes expectation conditioned on $(Y(0), \eta(0)) = (y_0, \eta_0) \in \T^d\times \R^K$.   In particular, this implies that

\begin{equation}
	\label{eq:geometric_ergodicity_ineq}
		\Norm{P_tf - \int f(y,\eta) \rho(dy,d\eta)}_{L^2(\rho)} \leq C'' e^{-\mu t},
	\end{equation} 
for some positive constants $C', C''$.
\end{proposition}

The Poisson problem for this scaling regime takes the following form
\begin{equation}
\label{eq:case4_celleqn}
	\mathcal{G}\chi(y, \eta) = -F(y, \eta), \qquad (y, \eta) \in \T^d \times \R^K.	
\end{equation}
\\
The existence of a smooth solution $\chi$ to (\ref{eq:case4_celleqn}) is guaranteed by the following result.
\begin{proposition}
\label{prop:case4_celleqn}
	Suppose the following centering assumption holds
	\begin{equation}
		\label{eq:case4_centering}
		\int_{\T^d \times \R^K} F(y, \eta) \, \rho(dy,d\eta)= 0.
	\end{equation}
	Then there exists a unique, smooth solution $\chi \in D(\mathcal{G})$ such that $$\int_{\T^d \times \R^K} \chi(y, \eta)\,\rho(dy,d\eta) = 0$$ which solves (\ref{eq:case4_celleqn}). The solution $\chi$ satisfies
\begin{equation}
	\notag
	\norm{\chi(y,\eta)} \leq C(1 + \norm{\eta}^2),	\qquad (y, \eta) \in \T^d \times \R^K
\end{equation}
where $C > 0$ is a constant independent of $(y, \eta)$.   Moreover, 
\begin{equation}
\begin{split}
\label{eq:case4_corrector_norm_bound}
 \int \nabla_y \chi^\top g^{-1}\nabla_y\chi\,  \rho(dy,d\eta)  & + \int \nabla_\eta \chi(y, \eta)^\top\,\Gamma\Pi \, \nabla_\eta \chi(y, \eta)\,\rho(dy,d\eta)\\  & = -2\int  \chi(y, \eta)\otimes\,\mathcal{G}\chi(y, \eta)\, \rho(dy,d\eta)  < \infty.
\end{split}
\end{equation}
\end{proposition}
\smartqed \qed
As before, the backward  Kolmogorov equation corresponding to (\ref{eq:case4_nondim_diffusion_sde}) is given by
\begin{subequations}
\begin{align}
\label{eq:kbe_case4}
	\frac{\partial v^\epsilon}{\partial t}(x, y, \eta, t) & = \mathcal{L}^\epsilon v^\epsilon(x,y,\eta,t ), \qquad (x,y,\eta,t) \in \R^d \times \T^d \times \R^K \times (0,T]\\
	v^\epsilon(x,0) &= v(x),
\end{align}
\end{subequations}
where
\begin{equation}
\mathcal{L}^\epsilon = \frac{1}{\epsilon^{2}}\mathcal{G} + \frac{1}{\epsilon}\mathcal{L}_{1} + \mathcal{L}_{2}
\end{equation}
for
\begin{equation*}
\begin{split}
\mathcal{L}_{1} f(x, y, \eta) = &\frac{1}{\sqrt{\norm{g}(y,\eta)}}\nabla \cdot \left(\sqrt{\norm{g}(y,\eta)}g^{-1}(y,\eta)\right)\cdot \nabla_{x}f(x,y,\eta) \\ &+ 2g^{-1}(y,\eta):\nabla_{x}\nabla_{y}f(x,y, \eta),
\end{split}
\end{equation*}
and 
\begin{equation*}
\begin{split}
\mathcal{L}_{2} f(x, y,\eta) &= g^{-1}(y, \eta):\nabla_{x}\nabla_{x}f(x,y,\eta).
\end{split}
\end{equation*}
\\
We assume that the initial condition $v$ is independent of the fast processes. Having Propositions \ref{prop:case4_inv} and \ref{prop:case4_celleqn} we can state the homogenization result for this regime.  As in the previous cases, we provide a formal derivation based on multiscale expansions in Appendix \ref{sec:case4_app}, and refer interested readers to \cite{duncan2013thesis} for a rigorous proof, based on the central limit theorem for additive functionals of Markov processes \cite{komorowski2012fluctuations}.
\\
\begin{theorem}
	\label{thm:homog_case4}
	Suppose Assumption (\ref{eq:case4_centering}) holds and $\eta^\epsilon(t)$  is stationary.  Then as $\epsilon \rightarrow 0$, the process $X^{\epsilon}(\cdot)$ converges weakly in $C([0,T]; \R^d)$  to a Brownian motion with diffusion tensor $D$ given by
\begin{equation}
	\label{eq:effdiff_case4}
	D = \int \left(I + \nabla_y\chi\right) g^{-1}\left(I + \nabla_y\chi\right)^\top\, \rho(dy, d\eta) + \int  \nabla_\eta \chi\, \Gamma\Pi \nabla_\eta \chi^\top  \, \rho(dy, d\eta).
\end{equation}
Moreover, if the backward equation (\ref{eq:kbe_case4}) has initial data $v$, independent of $\epsilon$  such that $v \in C^2_b(\R^d)$, then the solution $v^{\epsilon}$ of (\ref{eq:kbe_case4}) converges pointwise to the solution $v_{0}$ of
\begin{subequations}
\begin{align}
\frac{\partial v^{0}}{\partial t}(x,t)  & =  D:\nabla_{x}\nabla_{x}v^{0}(x,t), \qquad (x, t) \in \R^d \times (0,T]\\
v^{0}(x, 0) &= v(x), \qquad x \in \R^d,
\end{align}
\end{subequations}
uniformly with respect to $t$ over $[0,T]$.
\end{theorem}
\smartqed\qed

$ $\linebreak
Due to the lack of an explicit invariant measure for the fast process it is not clear whether or not the centering condition holds.  Numerical experiments suggest that the centering condition does hold for the surfaces we consider,  however it is not clear that this holds in general.  If the centering condition is not satisfied then one can consider the effective behaviour of the process $X^\epsilon(t)$ close to where the mean flow $Vt$ has taken it, where $$V = \int_{\T^d\times \R^K} F(y,\eta)\,\rho(dy, d\eta).$$   Indeed,  we can show that the process $X^\epsilon(t) - \frac{Vt}{\epsilon}$ converges to a Brownian motion for small $\epsilon$ \cite{pavliotis2007homogenization}. 
\\\\
For surfaces which satisfy the symmetric condition given by (\ref{eq:symmetry_condition_drift}) we are able to show that the centering condition holds.  The proof is very similar to that of Proposition \ref{prop:eff_drift} and is omitted.

\begin{proposition}
\label{prop:case4_centering_condition}
	Assume that condition (\ref{eq:symmetry_condition_drift}) is satisfied, then the centering condition (\ref{eq:case4_centering}) holds.
\end{proposition}
\smartqed\qed

\subsection{Numerical Experiments} 
\label{eq:timedep_numerics_case4}
Rather than resort to direct numerical simulations of the coupled SDEs,  we instead use a finite element scheme to solve the equations for the invariant measure and the corrector.   The finite element approximation then becomes $K+2$ dimensional problem.   For the sake of tractability, we restrict our interest to when $d=2$ and $K = 1$.  We calculate $D$ numerically as follows:
\begin{enumerate}
	\item We construct a piecewise linear finite element approximation to equation (\ref{eq:density_equation_case4}) on a regular, triangulated mesh of the domain $$ \Omega_M = \lbrace (y_1, y_2, \eta) \in [0,1]\times [0,1]\times [-M, M] \rbrace,$$  where $M$ is chosen so that the support of $\rho$ outside $[-M,M]$ is small.  We impose periodic boundary conditions on the boundaries in the $y_1$ and $y_2$ directions,  and no-flux boundary conditions in the $\eta$ direction.
	\item The solution $\rho$ of (\ref{eq:density_equation_case4}) is then obtained by solving the corresponding generalised eigenvalue problem for the eigenvector corresponding to the zero eigenvalue.   The resulting eigenvector is then normalised over $\Omega_M$ to give an approximation to $\rho$.
	\item The components of the corrector $\chi^{e_1}$ and $\chi^{e_2}$ are then computed by solving the Poisson equation (\ref{eq:case4_celleqn}) using a piecewise linear finite element scheme on the same mesh.  
	\item Finally, the components of the effective diffusion tensor are computed by integrating (\ref{eq:effdiff_case4}) using standard quadrature over $\Omega_M$.
\end{enumerate}

We apply the above steps to compute the effective diffusion tensor for the surface given by $h(x, \eta(t))$ where
	\begin{equation}
		\label{eqn:case4_surface}
		h(x,\eta) = \eta\, \sin(2\pi x)\sin(2\pi y),
	\end{equation}
and $\eta(t)$ is an OU process with SDE
\begin{equation}
		\notag
		d\eta(t) = -\Gamma\eta(t)\,dt + \sqrt{2\Gamma\Pi}W(t),
\end{equation}
where $\Gamma = \frac{\kappa^* \norm{2\pi {k}}^4 + \sigma^* \norm{2\pi k}^2}{\norm{2\pi k}}$,  $\Pi = \frac{1}{\kappa^* \norm{2\pi k}^4 + \sigma^*\norm{2\pi k}^2},$ and where $k = (1,1)^\top$.
\\\\
 In Figure \ref{fig:case3_diff_kappa} we plot the components of $D$ for $\kappa^* \in [10^{-3}, 1.0]$.  We note immediately that the symmetry in $h(x,\eta)$ is sufficient to ensure that $D$ is isotropic.  Moreover,   as in the previous macroscopic limits being considered, $D$ appears to be bounded above by $1$, so that the macroscopic diffusion is depleted with respect to the molecular diffusion tensor.  

\begin{figure}[H]
	\includegraphics[width=\textwidth]{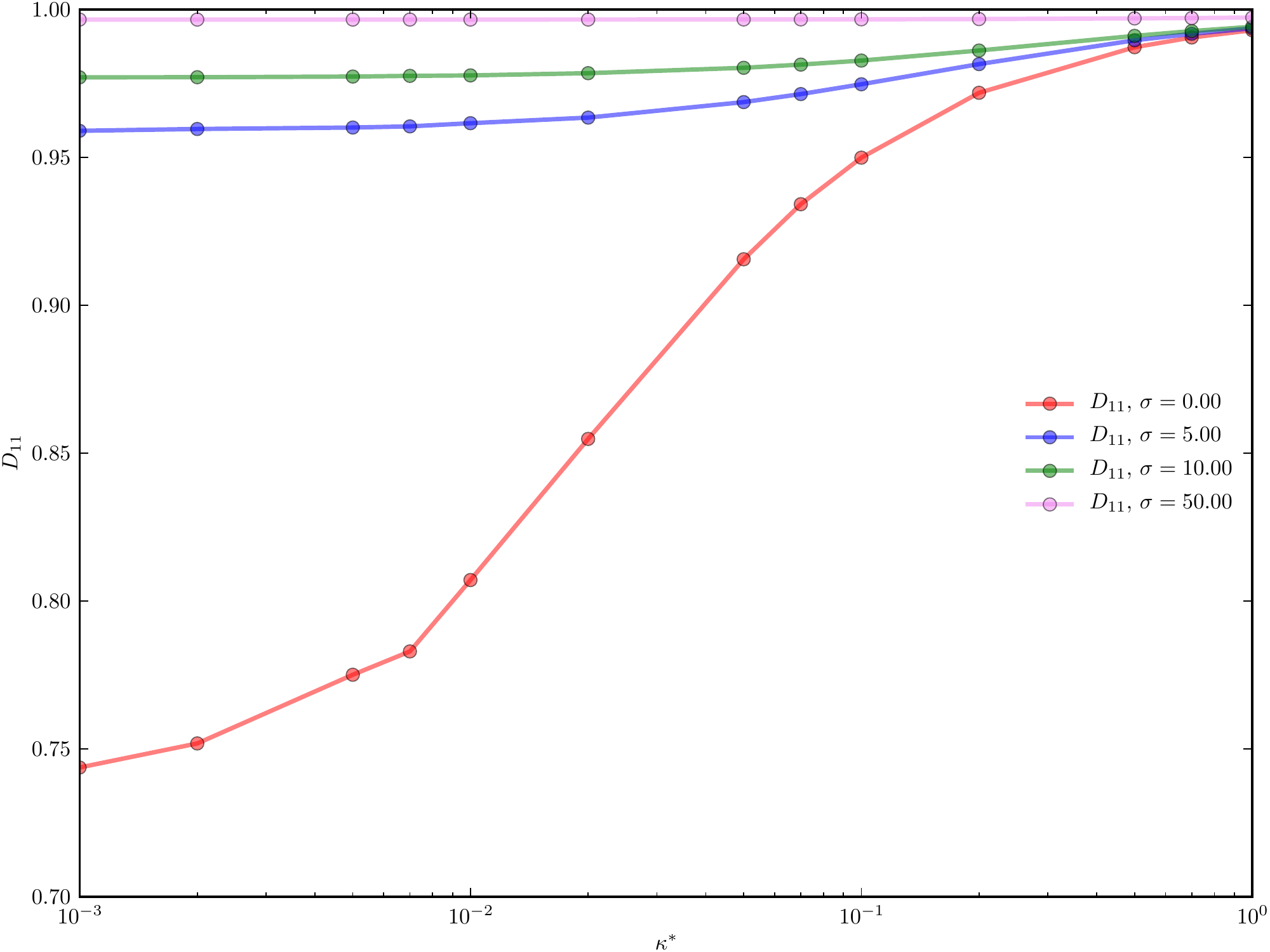}
\caption[Effective diffusion tensor for a simple fluctuating surface model in the Case IV regime]{Plots of the $D_{11}$ component of the effective diffusion tensor for the surface given by (\ref{eqn:case4_surface}) in the $(\alpha, \beta) = (1,2)$ regime.}
\label{fig:case3_diff_kappa}
\end{figure}

\section{Other distinguished limits}
\label{sec:other_scalings}
The four cases considered in this paper are not exhaustive,  and indeed for any $\alpha$ and $\beta$ one can use a similar approach to derive a well defined limit.   While not every choice of $(\alpha, \beta)$ will give rise to a limiting SDE without further assumptions,   in the particular case of lateral diffusion on a Helfrich surface the limiting equations can be described very succintly.  
\\\\
By relabeling $\epsilon^\alpha$ as $\epsilon$, we need only consider the regimes $(\alpha, \beta) = (0,1)$ and  $ (\alpha,\beta) \in \lbrace 1 \rbrace \times [-\infty, \infty)$. If we denote by $D_1(\eta)$, $D_2$, $D_3$ and $D_4$ the effective diffusion tensors  given by (\ref{eq:eff_diff_matrix}), (\ref{eq:effdiff_helfrich}), (\ref{eq:case3_eff_diff}) and (\ref{eq:effdiff_case4}) respectively, corresponding to Case I to IV, then one can show that the process $X^\epsilon(t)$ will converge weakly in $C([0,T]; \R^d)$ to a process $X^0(t)$ as defined in Table \ref{tab:distinguished_limits}.
\\
\begin{table}[h]
\centering
    \begin{tabular}{|c| c   |}
    \hline
    Scaling regime & Macroscopic limit \\
    \hline
   	 $\alpha = 0$, $\beta = 1$, & $X^0(t) = \sqrt{2 D_2}B(t)$\\
	 $\alpha = 1$, $-\infty \leq \beta < 0$, & $X^0(t) = \sqrt{2 D_1(\eta(0))}B(t)$ \\
	 $\alpha = 1$ and $\beta = 0$,  & $X^0(t) = \sqrt{2 D_1(\eta(t))}B(t)$\\
	 $\alpha = 1$ and $0 < \beta < 2$,  & $X^0(t) = \sqrt{2 D_3}B(t)$\\
 $\alpha = 1$ and $\beta = 2$,  & $X^0(t) = \sqrt{2 D_4}(t)$\\
  $\alpha = 1$ and $2 < \beta \leq 3$, & Not determined\\
	 $\alpha = 1$ and $\beta > 3$, &$X^0(t) = \sqrt{2 D_2}B(t)$\\
\hline
    \end{tabular}
  \caption{Distinguished limits for the system $(X^\epsilon(t), \eta^\epsilon(t))$ describing the evolution of a particle diffusion on a Helfrich elastic surface undergoing thermal fluctuations.}
    \label{tab:distinguished_limits}  
\end{table}

The justification of this result can be found in \cite{duncan2013thesis} where a probabilistic approach similar to \cite{garnier1997homogenization} is adopted.  To prove the weak convergence of $X^\epsilon(t)$ to $X^0(t)$ as $\epsilon\rightarrow 0$,  we use It\^{o}'s formula and the solution of an auxiliary PDE to decompose each singular drift into a sum of a martingale and a number of ``remainder" drift terms.  While some of these remainder terms may also be singular,  they are of lower order and this process may be iterated until $X^\epsilon(t)$ has been decomposed into a sum of martingale terms and remainder terms which are $O(1)$ with respect to $\epsilon$.   The result then follows by the martingale central limit theorem.
\\\\
When applying this approach to the $\alpha = 1$, $2 < \beta \leq 3$ regime, it is not clear whether the auxiliary PDEs that arise are well posed.  By the Fredholm alternative,  existence of a solution depends on a solvability condition,  which is not satisfied in general for the fluctuating Helfrich surface.  Thus it does not appear possible to remove the singular drift terms that arise in this regime, so that a well-defined limit will not exist as $\epsilon \rightarrow 0$.
\\\\
As can be seen,  the different effective behaviour can be broadly split into three separate classes depending on the relative speed of the spatial and temporal fluctuations.  For $\beta \leq 0$ the fast fluctuations are contributed entirely by the small scale spatial structure of the surface and no averaging over the fluctuating surface modes occurs. This regime can thus be considered to be a trivial extension of the macroscopic limit derived in Case I.   If the relaxation time of the Fourier modes is comparable with the time scale of the lateral diffusion process then the effective diffusion tensor will depend on the current state $\eta(t)$ of the surface.  If the surface is fluctuating at an much shorter time scale,  then at the $O(1)$ time scale the surface is quenched  and the effective diffusion tensor will depend only on the initial surface configuration.  
\\\\
For $0 < \beta  < 2$,  the OU process will relax to equilibrium sufficiently fast for averaging to occur at $O(1)$  scales.  At an $O(\frac{1}{\epsilon})$ time scale, the process will have homogenized over the spatial fluctuations for a ``frozen" surface configuration.  At the $O(1)$ time scale additional averaging will take place due to temporal fluctuations.  The  effective diffusion tensor $D_3$ will be the spatially homogenized diffusion tensor $D_1(\eta)$ averaged over the invariant measure of $\eta(t)$,  as was described in Case III.
\\\\
For $\beta > 3$ the rapid temporal fluctuations dominate the fast process, and the diffusion process will have been averaged over the surface Fourier modes even at the characteristic time scale of the rapid spatial fluctuations.  Thus over macroscopic time scales the diffusion process is well-approximated by its annealed disorder limit,  as in Case II for $(\alpha, \beta) = (0,1)$.

\section{Conclusion}
\label{sec:conclusion}
We have studied a model for diffusion of particles on a rough, rapidly fluctuating  quasi-planar surface where the surface is periodic in space and with temporal fluctuations modelled by a stationary, ergodic Markovian process.  By considering the coupled system of particle diffusion and surface fluctuations, we identified a natural set of distinguished limits which arise from different relationships between the spatial and temporal fluctuations.  Through the application of multiscale expansions we thus provided a unified approach to studying the macroscopic effects of both rapid spatial and temporal fluctuations on a laterally-diffusing particle.   We identified four natural scaling regimes which possess interesting limiting behaviour and in each case identified which properties of the surface fluctuations have a dominant effect on the macroscopic diffusion tensor.
\\\\
The regimes described in Case I and Case II have previously been considered in \cite{gustafsson1997diffusion} and \cite{naji2007diffusion} to study diffusion on quenched and annealed Helfrich fluctuating membranes respectively.  Using multiscale analysis techniques we rederived the main results in each of these papers and recovered bounds which had been previously obtained heuristically.  After making explicit the fact that these two cases correspond to different scalings of the same model, we then consider other possible scalings limits for this model, namely Case III and Case IV as examples of scalings containing both rapid spatial and temporal fluctuations.  We believe that this work provides a clear unified approach to the problem of lateral diffusion on rapidly varying surfaces, bringing together previously derived results in a single framework which can be analyzed with a common set of methods.
\\\\
While the effective diffusion and drift coefficients do not always have closed forms, we identified natural symmetry assumptions which guarantee that the effective diffusion has an explicit expression depending only on the surface area.  Moreover, we derive bounds of varying degrees of tightness which are satisfied by the effective diffusion tensor for general surfaces. To illustrate the derived results, numerical schemes which compute the effective diffusion tensor using finite elements were implemented.
\\\\
As noted in \cite{naji2007diffusion}, for the fluctuating Helfrich elastic membrane model, the measured diffusion tensor will deviate significantly from the effective diffusion tensor when the small scale parameter $\epsilon$ is not small. In such cases the multiscale approach can still be applicable in these regimes by computing higher order correctors \cite{pavliotis2008multiscale, bensoussan1978asymptotic} which quantify the disparity between the multiscale process and the homogenized process in terms of powers of $\epsilon$.
\\\\
The model presented in this paper can be extended in several directions.  It would be interesting to consider extensions of this problem such as curvature-coupled diffusion (similar to the models presented in \cite{reister2007lateral,leitenberger2008curvature, reister2010diffusing} or diffusion on membranes with non-thermal fluctuations, such as \cite{lin2006nonequilibrium}. Furthermore, it would be interesting to extend the approach to study more general surfaces, possibly even closed surfaces embedded in $\R^3$, which to our knowledge has not been previously considered analytically.

\section*{Acknowledgement}
\label{sec:acknowledgement}
AD is grateful to EPSRC for financial support and thanks the Centre for Scientific Computing @ Warwick for computational resources.  GP acknowledges financial support from EPSRC grant nos. EP/J009636/1 and EP/H034587/1.  AMS is grateful to EPSRC and ERC for financial support.  

\appendix
\section{Formal Multiscale Expansions}
\label{app:A}
In this section we formally derive the limiting equations in each of the scaling regimes using multiscale expansions.   
\subsection{Case I}
\label{sec:case1_app}
To derive the homogenised equation in this regime we make the ansatz that the solution $v^\epsilon$ of (\ref{eq:kbe_ms}) is of the form
\begin{equation}
	\label{eq:v_ansatz}
	v^{\epsilon} = v_{0}(x,y,t) + \epsilon v_{1}(x,y,t) + \epsilon^{2} v_{2}(x,y,t) + \ldots,
\end{equation}
for smooth $v_{i}:\R^{d}\times \T^{d} \times [0,T] \rightarrow \R^d$.  Substituting (\ref{eq:v_ansatz}) in (\ref{eq:kbe_ms}) and identifying equal powers of $\epsilon$ we obtain the following equations
\begin{align}
\centering
	\label{eq:e_min_2}
	&O\left(\frac{1}{\epsilon^2}\right): &\mathcal{L}_{0}v_{0}(x,y,t) = 0,\qquad\\
	\label{eq:e_min_1}
	&O\left(\frac{1}{\epsilon}\right):  &\mathcal{L}_{0}v_{1}(x,y,t) = -\mathcal{L}_{1}v_{0}(x,y,t), \qquad\\
	\label{eq:e_o_1}
	&O\left(1\right): &\mathcal{L}_{0}v_{2}(x,y,t) = \frac{\partial v_{0}}{\partial t} -\mathcal{L}_{1}v_{1}(x,y,t) -\mathcal{L}_{2}v_{0}(x,y,t), \qquad
\end{align}
for $(x,y,t) \in \R^d \times \T^d \times (0,T]$.
\\\\
Since the nullspace of $\mathcal{L}_{0}$  contains only constants in $y$, equation (\ref{eq:e_min_2}) thus implies that $v_{0}$ is a function of $x$ and $t$ only.   Equation (\ref{eq:e_min_1}) becomes
\begin{equation}
	\begin{split}
	\mathcal{L}_{0}v_{1}(x,y,t) &=  -F(y)\cdot\nabla_{x}v_{0}(x,t).
	\end{split}
\end{equation}
Let $\chi \in C^2(\T^d; \R^{d})$ be the unique, mean-zero solution of the cell equation (\ref{eq:celleqn_case1}).  If we choose $v_1 = \chi \cdot \nabla_x v_0(x,t)$ then it is clear that $v_{1}$ solves (\ref{eq:e_min_1}).  
\\\\
Finally, by the Fredholm alternative on $\mathcal{L}_0$,  a neccessary condition for equation (\ref{eq:e_o_1}) to have a solution is that the RHS of (\ref{eq:e_o_1}) has mean zero with respect to the measure $\rho$, that is,

\begin{equation}
	\notag
\frac{\partial v_{0}(x,t)}{\partial t} = \frac{1}{Z}\int_{\T^{d}} \mathcal{L}_{1}v_{1}(x,y,t) \rho(y)\,dy + \frac{1}{Z}\int_{\T^{d}}\mathcal{L}_{2}v_{0}(x,t) \rho(y)\,dy.
\end{equation}
Substituting $v_{0}$ and $v_{1}$ we obtain
\begin{equation}
	\notag
\begin{split}
\frac{\partial v_{0}(x,t)}{\partial t}  & =  \frac{1}{Z}\int_{\T^{d}} \nabla_{y}\cdot\left(\sqrt{\norm{g}(y)}g^{-1}(y)\right)\cdot\nabla_{x}\left(\chi\cdot \nabla_{x}v_{0}(x,t)\right)\, dy  \\ &+ \frac{2}{Z}\int_{\T^{d}}\sqrt{\norm{g}(y)}g^{-1}(y):\nabla_{x}\nabla_{y}\left(\chi\cdot \nabla_{x}v_{0}(x,t)\right) \,dy \\ &+\frac{1}{Z}\int_{\T^{d}}\sqrt{\norm{g}(y)}g^{-1}(y):\nabla_{x}\nabla_{x}v_{0}(x,t)\,dy.
\end{split}
\end{equation}
Integrating the second term by parts with respect to $y$ and simplifying we obtain:
\begin{equation}
	\notag
\begin{split}
\frac{\partial v_{0}(x,t)}{\partial t}  &  = \Bigl(\frac{1}{Z}\int_{\T^{d}} g^{-1}(y)\bigl(I + \nabla_{y}\chi(y)\bigl) \sqrt{\norm{g}(y)}\, dy\Bigl):\nabla_{x}\nabla_{x}v_{0}(x,t),\end{split}
\end{equation}
where we have used the symmetry of $g^{-1}$. Thus the homogenized diffusion equation for $v_0$ is 
\begin{equation}
	\label{eq:case1_homogenized_kbe}
	\frac{\partial v_{0}(x,t)}{\partial t}   = D:\nabla_{x}\nabla_{x}v_{0}(x,t),
\end{equation}
where $$D = \frac{1}{Z}\int_{\T^{d}} g^{-1}(y)\bigl(I + \nabla_{y}\chi(y)\bigl) \sqrt{\norm{g}(y)}\, dy.$$
 Multiplying (\ref{eq:celleqn_case1}) by $\chi(y)\rho(y)$ and integrating by parts gives:
\begin{equation}
	\notag
	\int_{\T^{d}}\left(I + \nabla_{y}\chi(y)\right)^{\top} g^{-1}(y)\nabla_{y}\chi(y)\,\sqrt{\norm{g}(y)}\,dy = 0,
\end{equation}
so that the effective diffusion matrix can be written in the following symmetric form
$$D = \frac{1}{Z}\int_{\T^{d}} \bigl(I + \nabla_{y}\chi(y)\bigl)^\top g^{-1}(y)\bigl(I + \nabla_y \chi(y)\bigl) \sqrt{\norm{g}(y)}\, dy.$$

From the limiting Backward Kolmogorov equation (\ref{eq:case1_homogenized_kbe}) we can read off the limiting SDE $dX^0(t) = \sqrt{2D}\,dB(t)$. A rigorous proof of this result can found in \cite{duncan2013thesis}.

\subsection{Case II}
\label{sec:case2_app}
Analogous to the previous case, we look for solutions $v$ of the form
\begin{equation}
	\notag
	v^\epsilon(x ,\eta, t) = v_{0}(x,\eta,t) + \epsilon v_{1}(x,\eta,t) + \ldots,
\end{equation}
for some smooth functions $v_i:\R^d \times \R^K\times [0,T] \rightarrow \R^d$.  Substituting this ansatz in (\ref{eq:kbe_case3}) and identifying equal powers of $\epsilon$ we obtain the following pair of equations:

\begin{align}
\centering
&O\left(\frac{1}{\epsilon}\right): &\mathcal{L}_0 v_0(x, \eta ,t) = 0, \\
&O(1): &\frac{\partial v(x, \eta, t)}{\partial t}  = \mathcal{L}_{0}v_1(x,\eta, t) + \mathcal{L}_{1}v_0(x, \eta, t), 
\end{align}
where $(x, \eta, t) \in \R^d \times \R^K \times (0,T].$
\\\\
The $O(\frac{1}{\epsilon})$ equation immediately implies that $v_0$ is independent of the fast-scale fluctuations.  The second equation then becomes
\begin{equation}
	\notag
\mathcal{L}_0 v_1(x, \eta, t) = \frac{\partial v(x, \eta, t)}{\partial t} - \mathcal{L}_1 v_0(x, \eta, t).
\end{equation}
Applying the Fredholm alternative, a necessary condition for the existence of a solution $v_1$ is that the RHS is orthogonal to the invariant measure $\rho_\eta$,  that is,
\begin{equation}
	\notag
\begin{split}
	\frac{\partial v_0}{\partial t}(x,t) & = \left[\int_{\R^K}F(x,\eta)\rho_\eta(\eta)\right]\cdot \nabla v_0(x,t) + \left[\int_{\R^K}\Sigma(x,\eta)\rho_\eta(\eta)\right]:\nabla \nabla v_0(x,t),
\end{split}
\end{equation}
which is the backward equation for SDE (\ref{eq:case2_homogenized_sde}).

\subsection{Case III}
\label{sec:case3_app}
We make the ansatz that $$v^\epsilon = v_{0} + \epsilon v_{1} + \epsilon^{2} v_{2} + \cdots,$$ for some smooth functions $v_i:\R^d \times \T^d \times \R^K \times [0,T] \rightarrow \R$.   Substituting $v^\epsilon$ in (\ref{eq:kbe_case2}) and identifying equal powers  of $\epsilon$ we obtain the following equations:

\begin{align}
\centering
&O\left(\frac{1}{\epsilon^{2}}\right): &\mathcal{L}_{0}v_{0} = 0, \\
&O\left(\frac{1}{\epsilon}\right):  &\mathcal{L}_{0}v_{1} = -\mathcal{L}_{\eta}v_{0}  -\mathcal{L}_{1}v_{0} ,\\
&O(1): &\mathcal{L}_{0}v_{2}  = -\left(\frac{\partial v_{0} }{\partial t} - \mathcal{L}_{\eta}v_{1}  - \mathcal{L}_{1}v_{1}  - \mathcal{L}_{2}v_{0}\right).
\end{align}

The first equation implies that $v_{0} \in \mathcal{N}[\mathcal{L}_{0}]$ so that $v_{0}$ is a constant in $y$. The second equation thus becomes

\begin{equation}
	\notag
\mathcal{L}_{0}v_1(x,y, \eta, t) = \left(\mathcal{L}_{\eta}v_{0}(x, \eta, t) + F(y,\eta)\nabla_{x}v_{0}(x, \eta, t)\right).
\end{equation}

By the Fredholm alternative applied to $\mathcal{L}_{0}$ we require that the RHS is centered with respect to $\sqrt{\norm{g}}$, for each fixed $x$ and $\eta$ that is,
\begin{equation}
	\notag
	\int_{\mathbb{T}^{d}}\left(F(y,\eta)\nabla_{x}v_{0}(x,\eta, t) + \mathcal{L}_{\eta}v_{0}(x, \eta, t) \right) \sqrt{\norm{g}(y,\eta)}\,dy = 0.
\end{equation}
The first term in the above integral is clearly $0$.  Since $\mathcal{L}_{\eta}v_{0}$ is independent of $y$, the centering condition becomes
\begin{equation}
	\notag
	Z(\eta)\mathcal{L}_{\eta}v_{0}(x, \eta, t)  = 0.
\end{equation}
Since $Z > 1$, it follows that $v_{0} \in \mathcal{N}[\mathcal{L}_{\eta}]$ is a sufficient condition for the centering condition to hold, which we therefore will assume.   By ergodicity of the Ornstein Uhlenbeck process $\eta(t)$ over $\R^{K}$ it follows that $v_{0}$ is also independent of $\eta$ so that $v_{0}$ is a function of $x$ only.  The second equation thus becomes
\begin{equation}
	\notag
\mathcal{L}_{0}v_1 = F(y,\eta)\cdot\nabla_{x}v_{0}.
\end{equation}
 Let $\chi(\cdot, \eta)$ be the unique, mean-zero solution of the cell equation solution by the Fredholm alternative, since the centering condition holds. Choosing $v_{1} = \chi \cdot\nabla_{x}v_{0}$ it is clear that $v_1$ solves the $O(\frac{1}{\epsilon})$ equation.
\\\\
We now consider the $O(1)$ equation.  By the Fredholm alternative,  a necessary condition for the existence of a unique solution $v_{2}$ is that the RHS is centered with respect to the invariant measure of $\mathcal{L}_{0}$.  That is,
\begin{equation*}
\frac{\partial v_{0}}{\partial t} =\int_{\mathbb{T}^{d}}\left(\mathcal{L}_{\eta}v_{1} + \mathcal{L}_{1}v_{1} + \mathcal{L}_{2}v_{0}\right)\rho_{y}\, dy,
\end{equation*}
which, substituting the definitions of the $\mathcal{L}_{i}$'s and  $v_{j}$'s, can be written as follows
\begin{subequations}
\label{eq:homog1}
\begin{align}
\frac{\partial v_{0}}{\partial t} & = \int_{\mathbb{T}^{d}} F\otimes \chi\, \rho_{y}\,dy: \nabla_{x}\nabla_{x}v_{0} \\  
&+ \int_{\mathbb{T}^{d}} g^{-1}\nabla_{y}\chi\rho_{y} + \nabla_{y}\chi\, g^{-1}\rho_{y}\,dy: \nabla_{x}\nabla_{x}v_{0} \\ &+  \int_{\mathbb{T}^{d}} g^{-1}\rho_{y}\,dy:\nabla_{x}\nabla_{x}v_{0} \\ &+ \int_{\mathbb{T}^{d}} \mathcal{L}_{\eta}\chi\rho_{y}\,dy\, \nabla_{x}v_{0}.
\end{align}
\end{subequations}

First we note that 
\begin{equation*}
\begin{split}
\int_{\mathbb{T}^{d}} F(y,\eta)\otimes \chi(y, \eta) \sqrt{\norm{g}(y, \eta)}\,dy & = -\int_{\mathbb{T}^{d}} \mathcal{L}_{0}\chi(y, \eta)\otimes \chi(y, \eta) \sqrt{\norm{g}(y, \eta)}\,dy \\ &= \int_{\mathbb{T}^{d}}\nabla_{y}\chi(y, \eta)^\top g^{-1}(y, \eta) \nabla_{y}\chi(y, \eta)\sqrt{\norm{g}(y,\eta)}\, dy,
\end{split}
\end{equation*}
so that we can write (\ref{eq:homog1}) as 
\begin{equation}
	\notag
\begin{split}
	\frac{\partial v_{0}}{\partial t} = 	\int_{\mathbb{T}^{d}}\left(I + \nabla_{y}\chi\right)^\top g^{-1}\left(I + \nabla_{y}\chi\right)\rho_{y}\, dy: \nabla_{x}\nabla_{x}v_{0} + \int_{\mathbb{T}^{d}} \mathcal{L}_{\eta}\chi\rho_{y}\,dy \cdot \nabla_{x}v_{0}.
\end{split}
\end{equation}
\\
Averaging with respect to the invariant measure $\rho_{\eta}$ of $\mathcal{L}_{1}$ we derive the effective diffusion equation
\begin{equation}
	\notag
\begin{split}
	\frac{\partial v_{0}}{\partial t} = 	\int\int_{\mathbb{T}^{d}}\left(I + \nabla_{y}\chi\right)^\top g^{-1}\left(I + \nabla_{y}\chi\right) \rho_{y}\,  \rho_{\eta}(h)\, dy\,d\eta: \nabla_{x}\nabla_{x}v_{0} + \int\int_{\mathbb{T}^{d}} \mathcal{L}_{\eta}\chi\rho_{y}\rho_{\eta}\,dy\, d\eta\cdot\nabla_{x}v_{0},
\end{split}
\end{equation}
or more compactly:
\begin{equation}
	\notag
\begin{split}
	\frac{\partial v_{0}}{\partial t} = 	 D: \nabla_{x}\nabla_{x}v_{0}  + L\cdot\nabla_{x}v_{0},
\end{split}
\end{equation}
where $D$ and $L$ are given by (\ref{eq:case3_eff_diff}) and (\ref{eq:case3_eff_drift}) respectively.   From the limiting backward Kolmogorov equation we can read off the limiting SDE (\ref{eq:case3_homogenized_sde}) for the process $X^{\epsilon}(t)$.
\\\\
\subsection{Case IV}
\label{sec:case4_app}
We look for solutions $v$ of the form $$v^\epsilon = v_0 + \epsilon v_1 + \epsilon^2 v_2 + \ldots$$  of (\ref{eq:kbe_case4}) for some smooth functions $v_i:\R^d \times \T^d \times \R^K \times [0,T]\rightarrow \R$.   Substituting this ansatz in (\ref{eq:kbe_case4}) and equating equal powers of $\epsilon$ we obtain the following 3 equations
\begin{align}
	O\left(\frac{1}{\epsilon^{2}}\right)&: &\mathcal{G}v_{0} = 0,\\
	O\left(\frac{1}{\epsilon}\right)&: &\mathcal{G}v_{1} = -\mathcal{L}_{\eta}v_{0} -\mathcal{L}_{1}v_{0},\\
	O(1)&: &\mathcal{G}v_{2} = -\left(\frac{\partial v_{0}}{\partial t} - \mathcal{L}_{1}v_{1} - \mathcal{L}_{2}v_{0}\right).
\end{align}

As the fast process is ergodic,  the first equation implies that $v_0$ is independent of $y$ and $\eta$.  The second equation thus becomes
\begin{equation}
	\notag
	\mathcal{G}v_1 = -F(y, \eta)\cdot \nabla_x v_0.
\end{equation}
Since we are assuming Assumption (\ref{eq:case4_centering}),  there exists a unique solution of the Poisson problem (\ref{eq:case4_celleqn}), by  Proposition \ref{prop:case4_celleqn}.  By choosing $v_1 = \chi \cdot \nabla_x v_0$ we see that the second equation is satisfied.
\\\\
By Proposition \ref{prop:case4_celleqn}, a sufficient condition for the final equation to have a solution is that the RHS is orthogonal to the measure $\rho(dy, d\eta)$, (assuming that the RHS grows at most polynomially). That is,
\begin{equation}
	\notag
\begin{split}
	\frac{\partial v_0}{\partial t}(y, \eta)  = \int F(y, \eta)\cdot\nabla_x v_1 \, \rho(dy, d\eta)  &+ \int 2\Sigma(y, \eta):\nabla_x \nabla_y v_1 \, \rho(dy, d\eta) \\ &+ \int \Sigma(y,\eta):\nabla_x\nabla_x v_0 \, \rho(dy, d\eta),
\end{split}
\end{equation}  
which we can rewrite as
\begin{equation}
	\notag
\begin{split}
	\frac{\partial v_0}{\partial t}  &= D:\nabla_x\nabla_x v_0,
\end{split}
\end{equation}  
where the effective diffusion tensor $D$ is given by
\begin{equation}
	\notag
D = \int \left[\frac{1}{\sqrt{\norm{g}}}\nabla_y\cdot\left(\sqrt{\norm{g}}g^{-1}\right)\otimes \chi + g^{-1}\nabla_y \chi^\top  +  \nabla_y \chi g^{-1} +  g^{-1}\, \right] \rho(dy, d\eta)
\end{equation}
Note that the first term on the RHS $$\displaystyle\int \frac{1}{\sqrt{\norm{g}}}\nabla_y\cdot\left(\sqrt{\norm{g}}g^{-1}\right)\otimes \chi\, \rho(dy, d\eta):\nabla_x \nabla_x v_0,$$  
can be rewritten as $\mathcal{K}:\nabla_x \nabla_x v_0,$ where $$\mathcal{K} = \mbox{Sym}\left[\displaystyle\int \frac{1}{\sqrt{\norm{g}}}\nabla_y\cdot\left(\sqrt{\norm{g}}g^{-1}\right)\otimes \chi\, \rho(dy, d\eta)\right],$$
where $\mbox{Sym}\left[\cdot\right]$ denotes the symmetric part of the matrix.  Let $e \in \R^d$ be a unit vector and consider $$\mathcal{K}^e := e \cdot \mathcal{K} e = \displaystyle\int \frac{1}{\sqrt{\norm{g}}}\nabla_y\cdot\left(\sqrt{\norm{g}}g^{-1}e\right) \chi^e\, \rho(dy, d\eta),$$
where $\chi^e = \chi \cdot e$.   Noting that $$-\mathcal{G}\chi^e = \frac{1}{\sqrt{\norm{g}}}\nabla_y\cdot\left(\sqrt{\norm{g}}g^{-1}e\right),$$
it follows that
$$\mathcal{K}^e = \int \chi^e\left(-\mathcal{G}\chi^e\right)  \, \rho(dy, d\eta),$$
which, by (\ref{eq:case4_corrector_norm_bound}) can be written as 
\begin{equation*}
	\mathcal{K}^e = \int \nabla_y \chi^e\cdot g^{-1} \nabla_y \chi^e \,\rho(dy, d\eta) + \int \nabla_\eta \chi^e \cdot \Gamma\Pi \nabla_\eta \chi^e \, \rho(dy, d\eta),
\end{equation*}
so that 
\begin{equation}
	\notag
	e\cdot D e = \int \left(e + \nabla_y\chi^e\right)\cdot g^{-1}\left(e + \nabla_y\chi^e\right)\, \rho(dy, d\eta) + \int \nabla_\eta \chi^e \cdot \Gamma\Pi \nabla_\eta \chi^e \,\rho(dy, d\eta), 
\end{equation}
or in matrix notation
\begin{equation}
	\notag
	D = \int \left(I + \nabla_y\chi\right) g^{-1}\left(I + \nabla_y\chi\right)^\top\, \rho(dy, d\eta) + \int  \nabla_\eta \chi\, \Gamma\Pi \nabla_\eta \chi^\top  \,\rho(dy, d\eta), 
\end{equation}
as required.  A rigorous proof of this result can be found in \cite{duncan2013thesis}.

\bibliographystyle{plain}
\bibliography{refs}
\end{document}